\documentclass[reqno]{amsart}
\usepackage{enumerate}
\usepackage{amssymb}
\usepackage{graphicx}
\usepackage[centertags]{amsmath}
\usepackage{amsfonts}
\usepackage{amsthm}
\usepackage{a4}
\usepackage{latexsym}
\usepackage{amsfonts}
\usepackage{amssymb}
\usepackage{amscd}
\usepackage{amsthm}
\usepackage{layout}
\usepackage{color}
\usepackage{mathrsfs}
\usepackage{commath}
\usepackage{xfrac}
\usepackage[nameinlink,noabbrev]{cleveref}
\usepackage{chngcntr} 
\counterwithin{section}{part}

\theoremstyle{definition}

\newtheorem{theorem}{Theorem}[section]

\newtheorem{corollary}[theorem]{Corollary}
\newtheorem{proposition}[theorem]{Proposition}
\newtheorem{lemma}[theorem]{Lemma}

\newtheorem*{thm*}{Theorem}
\newtheorem*{prp*}{Proposition}
\newtheorem*{crl*}{Corollary}
\newtheorem*{notn*}{Notation}
\newtheorem*{thm1*}{Theorem 1}
\newtheorem*{qst3*}{Question 3}
\theoremstyle{definition}
\newtheorem*{def1*}{Definition 2}
\newtheorem{definition}[theorem]{Definition}

\newtheorem{notation}[theorem]{Notation}

\newtheorem{parttheorem}{Theorem}[part]
\newtheorem{partcorollary}[parttheorem]{Corollary}
\newtheorem{partproposition}[parttheorem]{Proposition}

\theoremstyle{definition}
\newtheorem{remark}[theorem]{Remark}

\usepackage{enumitem}
\DeclareMathOperator{\kerr}{ker}
\DeclareMathOperator{\spann}{span}
\DeclareMathOperator{\dist}{dist}
\usepackage{mathtools}

\DeclarePairedDelimiter{\floor}{\lfloor}{\rfloor}

\newcommand{\R}{\mathbb{R}}
\newcommand{\N}{\mathbb{N}}

\newcommand{\Q}{\mathbb{Q}}

\newcommand{\dt}{\mathscr{T}}
\newcommand{\dtd}{\mathcal{D}}
\newcommand{\tree}{\mathcal{T}}

\DeclareMathOperator{\supp}{supp}
\newcommand{\minsupp}{\textnormal{minsupp}}
\newcommand{\rang}{\textnormal{ran}}

\newcommand{\segs}[1][n]{s_{#1}^{\brn}}
\newcommand{\segss}[2]{s_{#1}^{\brn_{#2}}}

\DeclareMathOperator{\ran}{ran}

\newcommand{\jtp}{JT_{2,p}}
\newcommand{\jt}{JT}
\newcommand{\jtg}{JT_{G}}
\newcommand{\xeh}{\mathcal{X}_{eh}}
\newcommand{\cod}{\gamma}
\newcommand{\e}{\epsilon}
\newcommand{\mc}{\mathcal}
\newcommand{\normm}[1][\cdot]{\lVert #1\rVert}
\renewcommand{\norm}[1]{\lVert #1\rVert}
\newcommand{\normg}[1]{\lVert #1\rVert}
\newcommand{\brn}{\sigma}
\newcommand{\brnt}{\tau}
\newcommand{\brT}{\mathcal{B}(\tree)}
\newcommand{\Pbrn}[1][\brn]{P_{#1}}
\newcommand{\sgm}{s}
\newcommand{\sgmt}{t}
\newcommand{\s}{s}
\newcommand{\schreier}[1][1]{\mathcal{S}_{#1}}
\newcommand{\xns}[1][n]{x_{#1}^{\brn}}
\newcommand{\sxn}{(x_{n})_{n\in\N}}
\newcommand{\syn}{(y_{n})_{n\in\N}}

\newcommand{\fns}[1][n]{\phi_{#1}^{\brn}}

\newcommand{\snbrn}{s_{n}^{\sigma}}

\newcommand{\fnss}[2]{\phi_{#1}^{\brn_{#2}}}
\newcommand{\xn}[1][x]{(#1_{n})_{n\in\mathbb{N}}}
\newcommand{\abss}[1]{\lvert #1\rvert}
\newcommand{\bwn}{\bar{w}_{n}}
\newcommand{\bxnk}{\bar{x}_{n_k}}

\newcommand{\bxnkr}{\bar{x}_{n_{k_{r}}}}

\begin{document}
\title{Variants of  the  James tree  space}
\author[S. A. Argyros]{S. A. Argyros}
\address{National Technical University of Athens, Faculty of Applied Sciences,
	Department of Mathematics, Zografou Campus, 157 80, Athens, Greece}
\email{sargyros@math.ntua.gr}
\author{A. Manoussakis}
\address{Technical  University of Crete, School of Environmental Engineering, Campus, 73100 Chania}
\email{amanousakis@isc.tuc.gr}
\author{P. Motakis}
\address{Department of Mathematics and Statistics, York University, 4700 Keele Street, Toronto, Ontario, M3J 1P3, Canada}
\email{pmotakis@yorku.ca}

	\begin{abstract}
Recently, W. Cuellar Carrera, N. de Rancourt, and V. Ferenczi introduced the notion of $d_2$-hereditarily indecomposable Banach spaces, i.e., non-Hilbertian spaces that do not contain the direct sum of any two non-Hilbertian subspaces. They posed the question of the existence of such spaces that are $\ell_2$-saturated. Motivated by this question, we define and study two variants $JT_{2,p}$ and $JT_G$ of the James Tree space $JT$. They are meant to be classical analogues of a future space that will affirmatively answer the aforementioned question.
	\end{abstract}

        \maketitle
\setcounter{tocdepth}{1}        
\tableofcontents
\section*{Introduction}
The present work is the first part of two papers which have as a goal
the definition of the  Banach space $\xeh$ with extremely heterogeneous
structure. As it happens in several exotic constructions, there are
classical or semi-classical structures  describing locally the global
properties of the spaces. For example, in the case of Hereditarily
Indecomposable   spaces (i.e., Banach spaces with no unconditional
basic sequence)  the summing basis of $c_{0}$  is a classical analogue
while the B. Maurey and H. P. Rosenthal  weakly  null basic
sequence with no unconditional  subsequence  \cite{MR} is the
semi-classical example which is  closer to the construction of the first
H.I. space by W.T.  Gowers and B. Maurey \cite{GMa}. Another example of
such a relation is the James Tree space \cite{J2}, which is the classical
example reflecting  the structure  of the subspaces of the famous Gowers Tree
space \cite{G}. To some extend    the two
James Tree spaces presented in the paper can be viewed as the classical and semi-classical  examples related to the space $\xeh$. More precisely, the space $\xeh$
is a non-Hilbertian space which satisfies the following two
contradicting properties.

a) It is  $\ell_{2}$-saturated and every  Hilbertian  subspace is a
complemented one.

b)  Every two non-Hilbertian subspaces $Y,Z$ of $\xeh$  satisfy
\[\dist (S_{Y},S_{Z})=\inf\{\norm{y-z}: y\in S_{Y},z\in S_{Z}\}=0.
\]
These two properties show that $\xeh$ is an $\ell_{2}$ saturated
$d_{2}$-H.I. space,
a class of Banach spaces defined recently by
W. Cuellar Carrera, N. de Rancourt and V. Ferenczi in \cite{CFdeR}. Our aim
in the present paper is to present James Tree spaces containing $\ell_{2}$
and satisfying the property that every Hilbertian subspace is
complemented. Thus, the first space $\jtp$ $(2<p<\infty)$
is a non Hilbertian reflexive space with an unconditional basis such
that the following properties hold.
\begin{enumerate}
\item\label{dichotomyproperty}  Every subspace is either  Hilbertian or contains $\ell_{p}$.
 \item\label{complementationproperty} Every Hilbertian subspace is complemented.
\end{enumerate}
The space $\ell_{2}\oplus\ell_{p}$ ($p\ne 2$) also satisfies label{dichotomyproperty} and \eqref{dichotomyproperty} and \eqref{complementationproperty} while, as it
was pointed out to us by W.B. Johnson, the space $L_{p}(0,1)$ ($2<p<\infty$)
satisfies  \eqref{complementationproperty}. The significant point of $\jtp$ is that, as we will
explain later, it reflects the structure of non-Hilbertian subspaces of
$\xeh$.

The second example $\jtg$ is a tree space satisfying the following
properties.
\begin{enumerate}
\item The space $\jtg$ is non reflexive.
\item Every subspace either contains $c_{0}$ or it is Hilbertian.
\item Every Hilbertian subspace   is a complemented one.
\end{enumerate}
In particular, the set  $G$ will impose the whole $\ell_{2}$ structure
in the space $\xeh$. In the last part of the introduction
we discuss in more detail the structures of the spaces $\jtp$ and $\jtg$.

Let us  point out   that the
understanding  of classical structures related to a desired exotic one
is sometimes  a necessary and important step in the process of applying state-of-the-art technology to the construction of spaces with saturated norms. In
some cases, like the present one, this step is not immediate.

The cornerstone  of our study is the classical James Tree space $JT$.
It was presented in
1974 by R. C. James as the first example of a separable space that
does not contain $\ell_1$ and  has non-separable dual (\cite{J2}).
Its
norm is defined on the unit vector basis $(e_n)_{n\in\N}$ of $c_{00}$.
Throughout this paper we will consider two partial orders on $\N$,
denoted by $\prec_\mathcal{D}$ and  $\prec_\mathcal{\tree}$, compatible with the
usual one and such that $(\N,\prec_\mathcal{D})$ is a dyadic tree   and
$(\N,\prec_\mathcal{\tree})$ is an infinitely  branching  tree.  As usual,
by $\brn$ we will denote the brances of the tree (i.e. the maximal
linear ordered subsets of the tree) and by $\s$ the segments of the tree.
To every segment $s$ of $\mathcal{D}$ we associate the functional $s:c_{00}\to\R$ given by $s(x) = \sum_{i\in s}x(i)$. The James Tree space is the completion of $c_{00}$ with the norm
\[
 \|x\| = \sup\Big\{\Big(\sum_{k=1}^n(s_k(x))^2\Big)^{1/2}:n\in\N\text{ and
  }s_1,\ldots,s_n\text{ are disjoint segments of }\mathcal{D}\Big\}.
\]

We list the basic properties of $JT$.

\begin{enumerate}[label={\textbullet}]

\item The unit vector basis of $c_{00}$ is a boundedly complete basis of $JT$.

\item The space $JT$ is $\ell_2$-saturated.

\item The dual $JT^*$ of $JT$ is non-separable.

\item The second dual $JT^{**} = JT\oplus \ell_2(\mathcal{B})$.

\end{enumerate}
Here, $\mathcal{B}$ denotes the set of branches of $\mathcal{D}$ and it has the cardinality of the continuum.

A great number of authors have studied the space $JT$. In
1981 I. Amemiya and T. Ito proved that every normalized weakly null
sequence in $JT$ has an $\ell_2$ subsequence (\cite{AI}). The proof of
this result plays a particularly important role in our work. In 1977
J. Hagler used techniques developed for $JT$ to study spaces of
continuous functions on certain topological spaces (\cite{H}). In 1984
G. A. Edgar and R. F. Wheeler (\cite{EW}) and in 1987 W. Schachermayer
(\cite{S}) studied topological properties of $JT$. In 1988
R. E. Brackebusch (\cite{brack}) studied JT on general trees
$\mathcal{T}$. In 1989 S. F. Bellenot  R. G. Haydon, and E. Odell
introduced and studied variants of $JT$ that are somewhat similar to
those studied herein. In 1985 N. Ghoussoub and B. Maurey (\cite{GhM})
studied the space $JT^\infty$, i.e., a James Tree space defined on the
infinitely branching countable tree of infinite height. They showed in
particular that the natural preduals of $JT$ and $JT^\infty$ are not
isomorphic.
There is an extreme structure related to James tree space and this is
the famous Gowers tree space  $X_{GT}$ \cite{G},  which answers in the negative direction the classical problem
whether every Banach space contains either a reflexive subspace or a
subspace with  an unconditional basis. The  Gowers tree space is a separable  Banach space not containing $\ell_{1}$ such that  every infinite dimensional  subspace has non-separable dual.
An extensive study of Gowers  Tree type spaces was conducted by the first author, A. Arvanitakis, and A. Tolias in \cite{AArvT}. Among other examples, they construct such a space $X$ with the property $X^{**} = X\oplus \ell_{2}(\R)$.

One of the major appeals of $JT$ is that it is the archetypal specimen for spaces that do not contain $\ell_1$, are separable, and have non-separable dual. Let us list some additional properties of the basis $(e_n)_n$ of $X = JT$ that are particularly relevant to this feature.
\begin{enumerate}[label=(\roman*)]

\item For every $\prec_D$-incomparable infinite subset $M$ of $\N$, the sequence $(e_n)_{n\in M}$ is weakly null.\label{ADK1}

\item For every $\prec_D$-branch $\sigma$ the sequence $(e_n)_{n\in\sigma}$ is non-trivial $w$-Cauchy, i.e., there exists $x_\sigma^{**}\in X^{**}\setminus X$ so that $w^*$-$\lim_{n\in\sigma }e_n = x^{**}_\sigma$.\label{ADK2}

\item The family $(x_\sigma^{**})_{\sigma\in\mathcal{B}}$ is unconditional.\label{ADK3}

\end{enumerate}
For $X  = JT$, in \ref{ADK1} the sequence $(e_n)_{n\in M}$ is equivalent
to the unit vector basis of $\ell_2$ and in \ref{ADK3} the collection
$(x_\sigma^{**})_{\sigma\in\mathcal{B}}$ is equivalent to the unit vector basis of
the non-separable Hilbert space $\ell_2(\mathcal{B})$. It follows, e.g.,
from a Baire Category argument that any space containing a sequence
that satisfies \ref{ADK1}, \ref{ADK2}, and \ref{ADK3} must have
non-separabe dual.
The intrinsic connection between the basis of $JT$ and the generic Banach space that contains no copy of $\ell_1$ and has non-separable dual is  made especially clear by a result of the first author, P. Dodos and V. Kanellopoulos from 2008 (\cite{ADK}).
\begin{thm1*} [\cite{ADK}]
For a separable Banach space $X$ that does not contain $\ell_1$ the following assertions are equivalent.
\begin{enumerate}[label=\textbullet]

\item The space $X$ has non-separable dual.

\item The space $X$ contains a sequence $(e_n)_{n\in\N}$ that satisfies \ref{ADK1}, \ref{ADK2}, and \ref{ADK3}.

\end{enumerate}
\end{thm1*}
The above theorem yields  that every subspace of the Gowers Tree space
contains a tree structure similar to the basis of the James tree space.

 A different example of a separable space not containing $\ell_1$ and
 having non-separable dual was later discovered by J. Lindenstrauss and
 C. Stegall in \cite{LS}.
 The authors modestly gave their example the name James function space
and  denoted  it by $JF$.  A detailed study of the
 subspaces of $JT$ and $JF$ can be found in the papers \cite{AAK1},\cite{AAK2} of
 D. Apatsidis, the first  author, and V. Kanellopoulos and in
 the paper \cite{AMPe}  of the first two authors and  M. Petrakis where the
 James  function space is defined and studied with domain subsets of
$\R^{n}$.

In 2002  W.T. Gowers, based on his famous dichotomy \cite{G2},\cite{G3}, proposed
a classification program  concerning the structure of the subspaces of
a Banach space. Important contributions to the  program have been  made by V. Ferenczi and C. Rosendal \cite{FR0},\cite{FR1}, \cite{FR2}.
Recently, in \cite{CFdeR}, W. Cuellar Carrera, N. de Rancourt  and
V. Ferenczi initiated a theory of classifying Banach spaces with
respect to the properties of their subspaces restricted within a
predetermined class.  The class of non-Hilbertian spaces is of particular
interest. The following definition of
W. Cuellar Carrera, N. de Rancourt  and
V. Ferenczi  is the foremost motivating concept behind this work.
\begin{def1*}[\cite{CFdeR}]
A Banach space is called $d_2$-hereditarily indecomposable (or $d_2$-HI) if it is non-Hilbertian and it does not contain the direct sum of two non-Hilbertian spaces.
\end{def1*}

One of the restricted dichotomy theorems from \cite{CFdeR} states that a non-Hilbertian Banach space contains a non-Hilbertian subspace that either has an unconditional finite dimensional decomposition or is $d_2$-HI. The list of currently known $d_2$-HI spaces is brief. It contains $X$ and $X\oplus\ell_2$, for $X$ HI, and an HI-sum of Hilbert spaces defined by the first author and Th. Raikoftsalis in \cite{ARa}. Noteworthily, all these spaces have HI subspaces. The following question serves as the launch pad of our investigation.

\begin{qst3*}[\cite{CFdeR}]
Does there exist an $\ell_2$-saturated $d_2$-HI Banach space?
\end{qst3*}
Our final purpose  is to provide an example of such a space. 
In particular there exists a space $\xeh$ satisfying the following properties:
\begin{enumerate}[label=(\alph*)]

\item It is reflexive and non-Hilbertian.\label{nonhilbertian}

\item It is $\ell_2$-saturated.

\item It does not contain the direct sum of two non-Hilbertian spaces.

\item Every unconditional sequence in the space is equivalent to the
  unit vector basis of $\ell_2$.

\item Every Hilbertian subspace is complemented in the space.\label{alwayscomplemented}

\item Every infinite dimensional subspace $Y$ is quasi-prime, i.e., it
  admits a unique decomposition $Y\oplus\ell_2$.

\end{enumerate}
More precisely, $\mathcal{X}_{eh}$ is a reflexive  $d_2$-HI space
that is $\ell_2$-saturated.

We now turn our attention to the first variant of the James Tree space
that is studied here, i.e., the space $JT_{2,p}$, for $2<p<\infty$. This is
a space with an unconditional basis that satisfies \ref{nonhilbertian}
and \ref{alwayscomplemented}.
It is defined using a compatible partial order $\prec_\mathcal{T}$ on $\N$ which induces an infinitely branching tree $\mathcal{T} = (\N,\prec_\mathcal{T})$. The space $JT_{2,p}$ is the completion of $c_{00}$ with the norm
\[\|x\| = \sup\Big\{\Big(\sum_{k=1}^n\|x|_{s_k}\|_{\ell_2}^p\Big)^{1/p}:n\in\N\text{ and }s_1,\ldots,s_n\text{ are disjoint segments of }\mathcal{T}\Big\}.\]
The fundamental properties of the basis $(e_n)_{n\in\N}$ of $JT_{2,p}$ are the following.
\begin{enumerate}[label=(\Roman*)]

\item For every $\prec_\mathcal{T}$-branch $\sigma$ of $\mathcal{T}$, $(e_n)_{n\in\sigma}$ is equivalent to the unit vector basis of $\ell_2$.\label{tbranchell2}

\item For every $\prec_\mathcal{T}$-incomparable infinite subset $M$ of $\N$, the sequence $(e_n)_{n\in M}$ is equivalent to the unit vector basis of $\ell_p$.
\end{enumerate}
This further mirrors the properties of the space $\mathcal{X}_{eh}$. In any non-Hilbertian subspace of $\mathcal{X}_{eh}$ there exists a sequence $(e_n)_{n\in\N}$ that satisfies \ref{tbranchell2} as well as
\begin{enumerate}[label=($\widetilde{\text{II}}$)]

\item For every $\prec_\mathcal{T}$-incomparable infinite subset $M$ of $\N$, the sequence $(e_n)_{n\in M}$ is not unconditional.\label{nonunconditionalantichain}

\end{enumerate}
This means that every subspace of $\mathcal{X}_{eh}$ is either small, i.e., isomorphic to $\ell_2$ or it contains a tree basis satisfying \ref{tbranchell2} and \ref{nonunconditionalantichain}.

At the heart of the proof of the properties of $JT_{2,p}$ lies a lemma
in the style of Amemiya and Ito for the space $JT$ (\cite{AI}) and the
concept of  $\ell_{p}$-vectors associated to certain weakly null sequences. The
behavior of a normalized block sequence $(x_n)_{n\in\N}$ in the space
$JT_{2,p}$ is characterized by the behavior of $(x_{n|\sigma})_{n\in\N}$, for
all branches $\sigma$ or $\mathcal{T}$. If, for all $\sigma$, $\|x_{n|\sigma}\|\to 0$
then $(x_n)_{n\in\N}$ has a subsequence that is equivalent to the unit
vector basis of $\ell_p$. Otherwise, it has a subsequence that is
equivalent to the unit vector basis of $\ell_2$. Consequently, every
normalized block sequence has a subsequence that is equivalent to the
unit vector basis of $\ell_p$ or of $\ell_2$. Thus, $JT_{2,p}$ is
reflexive.
For every subspace $Y$ of $JT_{2,p}$ that is isomorphic to $\ell_2$ there exist branches $\sigma_1,\ldots,\sigma_k$ of $\mathcal{T}$ so that if we denote by $P$ the canonical basis projection onto the Hilbertian space $Z$ spanned by $(e_n)_{n\in\sigma_1\cup\cdots\cup\sigma_k}$, then $P|_Y$ is an isomorphism.

After studying the space $JT_{2,p}$ we concentrate on the second variant of James Tree space.
For the definition of the space $\mathcal{X}_{eh}$ a norming set
$W_{eh}$ is used. Its construction requires the attractors method from
\cite{AArvT} (see also \cite{AMP}). A necessary ingredient is a ground
set $G$. This ground set defines a James Tree variant that we denote
by $JT_G$. Intuitively, in every subspace of $\mathcal{X}_{eh}$ we
find a sequence that ``attracts'' its properties from a block sequence
of the basis of $JT_G$ (this is the main function of the attractors
methods). Therefore, the space $JT_G$ is an important component that
is used directly in the construction of $\mathcal{X}_{eh}$. The space
$JT_G$ falls under a general definition for James Tree spaces that we
introduce here. Such general schemes have been considered in the past,
e.g., by S.F. Bellenot, R.G. Haydon, and E. Odell in \cite{BHO} as well as by F. Albiac and N. J. Kalton in \cite{AK}. Our version is a combination of the two aforementioned ones. It is based on an unconditional norm $\|\cdot\|_*$ as well as a collection of compatible  seminorms $(\|\cdot\|_\sigma)_\sigma$, where $\sigma$ varies over the branches of the tree $\mathcal{T}$. This means that for any two branches $\sigma_1$, $\sigma_2$, any segment $s\subset \sigma_1\cap\sigma_2$, and any $x\in c_{00}$ we have $\|x|_s\|_{\sigma_1} = \|x|_s\|_{\sigma_2}$. We denote this quantity by $\|x\|_s$. The general $JT$ space associated to $\|\cdot\|_*$ and $(\|\cdot\|_\sigma)_\sigma$ is the completion of $c_{00}$ under the norm
\begin{align*}
\|x\| = \max\big\{\norm{x}_{\infty},\sup\Big\{\Big\|\sum_{k=1}^n\|x\|_{s_k}&e_{\min(s_k)}\Big\|_*\!\!\!: n\in\N\\
&\text{ and }s_1,\ldots,s_n\text{ are disjoint segments of }\mathcal{T}\Big\}.
\end{align*}
By taking $\|\cdot\|_{*} = \|\cdot\|_2$ and $\|\cdot\|_\sigma$ to be the 
summing norm we retrieve the norm of $JT$ (technically, of $JT^\infty$). Similarly, for $\|\cdot\|_* = \|\cdot\|_p$ and $\|\cdot\|_\sigma = \|\cdot\|_2$ we retrieve $JT_{2,p}$.

Let us discuss how the norm of $JT_G$ is given by this scheme. For each branch $\sigma$ we define a sequence of functionals $(\phi_n^\sigma)_n$ on $c_{00}$ with successive supports so that $\cup_n\mathrm{supp}(\phi^\sigma_n) = \sigma$. This collection is compatible, in the sense that for any two branches $\sigma_1$, $\sigma_2$, segment $s\subset \sigma_1\cap\sigma_2$, and $n\in\N$ we have $\phi^{\sigma_1}_n|_s = \phi^{\sigma_2}_n|_s$. Then for each $\sigma$ we let
\[W(\sigma) = \Big\{\Big(\sum_{k=1}^n\lambda_i\phi^\sigma_i\Big)\Big|_E:n\in\N,(\lambda_i)_{i=1}^n\in B_{\ell_2^n}\cap\Q^n,\text{ and }E\text{ is an interval of }\N\Big\}\]
and $\|x\|_\sigma =\sup\{f(x):f\in W(\sigma)\}$. We also let $ \|\cdot\|_{S^{(2)}} $ denote the norm of the 2-convexified Schreier space, i.e.,
$\|x\|_{S^{(2)}} =  \sup_{F\in\mathcal{S}_1}\|x|_F\|_2$. Here, $\mathcal{S}_1 = \{F\subset \N:\#F\leq \min F\}$ is the Schreier family of order one. The space $JT_G$ is the general James Tree space associated to $\|\cdot\|_{S^{(2)}}$ and $(\|\cdot\|_\sigma)_\sigma$.

The specific choice of the functionals $\phi_n^\sigma$ is based on the notion of $\ell_2$-repeated averages from \cite{AMT}. For each $n\in\N$ and maximal $\mathcal{S}_n$-set $F$, the $\ell_2$-repeated average $x(n,F)$ is a specific $\ell_2$-convex combination of $(e_i)_{i\in F}$. For each $\sigma$ we appropriately choose increasing sequences of natural numbers $(m_{n,\sigma})_n$, $(k_{n,\sigma})_n$. We partition $\sigma$ into successive sets $(F_{n,\sigma})_n$ so that each $F_{n,\sigma}$ is a maximal $\mathcal{S}_{k_{n,\sigma}}$ set. We define $\phi_n^\sigma = m^{-2}_{n,\sigma} x(k_{n,\sigma},F_{n,\sigma})$ and $x_n^\sigma = m_{n,\sigma}^2x(k_{n,\sigma},F_{n,\sigma})$. Then, $(x_n^\sigma)_{n\in\N}$ is equivalent to the unit vector basis of $\ell_2$ and the map $P_\sigma(x) = \sum_{n=1}^\infty \phi_n^\sigma(x)x_n^\sigma$ defines a norm-one projection.

The space $JT_G$ satisfies similar properties to the ones satisfied by
the space $JT_{2,p}$ and the techniques used to prove them are
related.
As we have mentioned before,  the space $\jtg$ satisfies the following
main properties.
\begin{enumerate}[label=\textbullet]

\item The space $JT_G$ is non-reflexive, and thus non-Hilbertian.
\item Every Hilbertian subspace of $JT_G$ is complemented.
\item Every non-Hilbertian subspace of $JT_G$ contains a copy of $c_0$.
\end{enumerate}
The paper is divided into two parts. The first one is devoted to the study of
the space $\jtp$ and the second one
to the study of   $\jtg$. At the beginning of each part we have
included
a short section  where we  discuss  in further detail the properties of
each one of the  two spaces.

\part{The  $\jtp$-space}

This part concerns the study of the space $\jtp$ and it is organized in
four sections.  The first section is devoted to the definition of the
space $\jtp$ and  some basic  inequalities of the
norm of  some vectors,  that we use throughout  this part.

The second  section  concerns the study of  the weakly null
sequences $\sxn$ such that for every  branch $\brn$ of the tree
$\lim_{n}\norm{x_{n|\brn}}=0$.
The main result of that section is the following theorem.
\begin{parttheorem}
Let $(x_{n})_{n\in\N}$ be  a normalized block sequence in $\jtp$  such
that  for  every branch $\sigma$ of $\tree$  $\lim_{n}\norm{x_{n|\sigma}}=0$.  Then
there exists a subsequence $(x_{n})_{n\in N}$ of $(x_{n})_{n}$ that is
equivalent to the   unit vector basis of $\ell_{p}$.

In  particular, for every $\e>0$ there exists a subsequence of
$(x_{n})_{n}$ which is  $(\sqrt[q]{2}+\e)$-equivalent to the unit
vector basis  of $\ell_{p}$.
\end{parttheorem}
In order to prove this theorem we provide a refinement of the  Amemiya-Ito
lemma that enables us to use a gliding hump type argument.

In the third section  we study the weakly null sequences  that do
not satisfy  the assumption  of the above theorem. This
leads us to define the notion of $\ell_{p}$-vectors that can be traced
back  to 
\cite{LS}.  We prove the
following.
\begin{partproposition}
Let  $\sxn$ be a normalized weakly null sequence in $\jtp$. Then
there exists a subsequence $(x_{n_{k}})_{k}$ of $\sxn$    which is either isomorphic to  the unit
vector basis of $\ell_{p}$ or it  accepts a 
non-zero $\ell_{p}$-vector.

That means that 
there
exists  an  at most countable 
set of branches  $\brn_{i}$, $i\in N$, such that
$\lim_{k}\norm{x_{n_{k}|\brn_{i}}}=c_{i}>0$ for every $i\in N$,
and
$\lim_{k}\norm{x_{n_{k}|\tau}}=0$ for every branch $\tau\ne\brn_{i}$,
$i\in N$.
\end{partproposition}
This yields  that  in the case where the $\ell_{p}$-vector  of the sequence
$\sxn$  exists and it is not zero then there exists  an at most countable 
set  of branches determining the limit behavior of the norm of the
elements of the sequence.
 Using  this
as well as some stability results of  $\ell_{p}$-vectors,  which are  proved in this section,
we get one of  the main results of this section, which is the
following.
\begin{partproposition}
 Let $\sxn$  be a normalized weakly null sequence  in $\jtp$. Then
 there exists a subsequence of $\sxn$   that admits an upper
 $\ell_{2}$-estimate with constant  $4$.
\end{partproposition}
The  above result yields that every  normalized weakly null sequence
in $\jtp$ with a non-zero $\ell_{p}$-vector has a subsequence that is equivalent to
the unit vector basis of $\ell_{2}$.

Combining the above we get the following.
\begin{partcorollary}
  Let $\sxn$ be a normalized weakly null sequence  in $\jtp$. Then
  $\sxn$ has a subsequence which is  either equivalent to the unit vector basis of $\ell_{p}$  or equivalent to the unit vector basis of $\ell_{2}$.
\end{partcorollary}
The above also yields the following.
\begin{partcorollary}
The space $\jtp$ is reflexive.  
\end{partcorollary}
The main result of the fourth  section  concerns the complemented
subspaces of  $\jtp$. We prove the following.
\begin{parttheorem}
    Let $Y$ be a  subspace of $\jtp$  not containing
    $\ell_{p}$.
  Then $Y$ is   a  complemented subspace of $\jtp$ isomorphic to $\ell_{2}$.
\end{parttheorem}
To prove this theorem we employ  some  properties  of the  $\ell_{p}$-vector  of a normalized
weakly null sequence
to achieve the following.
\begin{partproposition}
For  every  subspace $Y$ of $\jtp$  not containing $\ell_{p}$ there
 exist a finite set of branches $T=\{\brn_{1},\dots,\brn_{d}\}$  and a finite codimensional
subspace  $Y_{0}$ of $Y$ such that  the  restriction  to $Y_{0}$  of the natural projection
 on the subspace $Z_{T}=\overline{\langle\{e_{n}:n\in \cup_{i=1}^{d}\brn_{i}\}\rangle}$
is an  isomorphism.
\end{partproposition}
 Since the subspace $Z_{T}$ is
 Hilbertian  it follows  that $Y_{0}$ is
 Hilbertian as well. This yields
that  the  subspace $Y$ is
also a  Hilbertian and  moreover complemented in $\jtp$.

 As a last result  of  Section 4 we prove that the space  $\jtp$ is
 isomorphic to the space $\jtp^{\dtd}$,  where $\jtp^{\dtd}$  is the space whose  norm 
 is defined as the norm of 
 $\jtp$
 but using the dyadic  tree $\dtd$ instead of the infinitely
 brancing 
 tree $\tree$.  The reflexivity of the spaces
$\jtp^{\dtd}$  and 
 $\jtp$ implies that the natural preduals of those spaces are
 isomorphic, in contrast  to the result of  N. Ghoussoub and  B. Maurey 
 who proved that the  preduals  of $JT$ and $JT^{\infty}$ are not
 isomorphic \cite{GhM}.

        \section{The definition of the norm of the space $\jtp$}
In this section we define the norm of the space $\jtp$ and we present
some of its basic properties. We start with the following notation.
\begin{notation}    We denote by $\N=\{1,2,\ldots\}$  the set of all positive integers.
Let  $\preceq$ be a partial order on $\N$, compatible with the standard order (i.e., $n \leq m$ whenever $n\preceq m$), defining an infinitely branching tree $\tree_{0}$ without a root.  Let $\tree$ be the tree we obtain by subjoining  to $\tree_{0}$ the empty set  as a root.

A segment is  any  set of the form  $[m,n]=\{k: m\preceq k\preceq n\}$.
We shall denote the segments by the letters  $\sgm,\sgmt$.
An initial segment   is a segment having as its  first element the root.
If $\s$ is a segment and $I$  an interval of   $\N$, the    restriction of $\s$ to $I$
is the segment  $\s_{|I}=\s\cap I$.

A branch is a maximal segment.   We shall denote the branches by the letters $\sigma,\tau$.
If $\sigma$ is a branch and $n\in\N$ we denote by  $\sigma_{> n}$
the set  $\sigma_{> n}=\{k\in \sigma: k>n\}$.
For a branch $\sigma$ and $n\in\N$ we shall call the set $\sigma_{> n}$ a final segment.

Also by $\brn_{< n}$ we shall denote the initial segment 
$\sigma_{<n}=\{k\in \sigma: k<n\}$.
In a similar manner, we define the initial segment   $\brn_{\leq n}$ and the final segment $\brn_{\geq n}$.

Two final   segments
$\sigma_{1,>n_{1}}$,$\sigma_{2,>n_{2}}$
are said to be incomparable  if their first elements are incomparable.
Let us note that

a)  if $\sigma_{1},\dots,\sigma_{l}$ are distinct branches then there exists
$k\in\N$  such that  the final segments $\sigma_{1,> k},\dots, \sigma_{l,>k }$ are  incomparable.

b) if $(\sgm_{n})_{n\in\N}$ is a sequence  of initial segments there exist a subsequence $(\sgm_{n})_{n\in L}$ and an  initial segment or branch $\sgm$ such that
$\sgm_{n}\xrightarrow[n\in L]{} \sgm$ in the pointwise topology.

If $M$ is an infinite subset of $\N$ we denote by $[M]$ the set of the
infinite subsets of $M$.

 For   $x\in c_{00}$  and a segment $\sgm$  of the tree we define
        $x_{|\s}$ be the restriction of $x$ to $\s$, i.e.,  $x_{|\sgm}=\sum_{n\in\sgm}x(n)e_{n}$.

        Given a Banach space $X$,
by $B_{X}$ we denote the unit ball of $X$ and  by $S_{X}$ its the unit
sphere.
If     $X$  has   a Schauder basis $(e_n)_n$, then for
every $x\in X$ with $x=\sum_na_ne_n$ we denote by $\supp(x)$  the
support of $x$, i.e.  $\supp(x)=\{n\in\N:a_n\neq0\}$.

For a given sequence $(a_{n})_{n}\in\ell_{p}$,   its  conjugate sequence
is the sequence $(b_{n})_{n}\in S_{\ell_{q}}$ such  that $\sum_{n}a_{n}b_{n}=\norm{(a_{n})_{n}}_{\ell_{p}}$.

In general, we follow \cite{LT} for standard notation and terminology concerning Banach space theory.
    \end{notation}
\subsection{The norm of the space $\jtp$}
 For $x\in c_{00}$ and a segment $\sgm$    we define 
\[\norm{x}_{\sgm}=\norm{x_{|\sgm}}_{\ell_{2}}=\left(\sum_{i\in \sgm}x(i)^{2}\right)^{1/2}.\]
		
		\begin{definition}Let  $p\in (2,\infty)$. On $c_{00}(\N)$ we define the following norm
                  \begin{equation}
\begin{split}\|x\|&=\sup\bigg\{
\left(\sum_{k=1}^{n}\norm{x}_{\sgm_{k}}^{p}\right)^{\frac{1}{p}}
  : \sgm_1,\ldots,\sgm_n\,\text{are pairwise disjoint segments}, n\in\N
  \bigg\}\notag
  \\
&=\sup\bigg\{
\left(\sum_{k=1}^{n}(\sum_{i\in \sgm_{k}}x(i)^{2})^{\frac{p}{2}}\right)^{\frac{1}{p}}
  : \sgm_1,\ldots,\sgm_n\,\text{are pairwise disjoint segments}, n\in\N
  \bigg\}.
\end{split}
\end{equation}
The completion of $c_{00}(\N)$ with respect to the above norm is
denoted by $JT_{2,p}$.
\end{definition}
Let us observe, for later  use, that
    $(e_{n})$ is a 1-unconditional and boundendly
                      complete basis of $\jtp$.	
Also it is easy to see that for every segment  $s$ and $x\in c_{00}(\N)$ 
it  holds 
$\norm{x}_{s}=\norm{x_{|s}}$.

\subsection{An alternative definition of the norm}
For every branch  $\sigma$ of  $\tree$ we set 
\[
B^{*}_{\sigma}=\{\sum_{i\in \sgm}a_{i}e^{*}_{i}:\s\,\,\text{is a segment of }\,\sigma\,\text{and}\, (a_{i})_{i\in \s}\in B_{\ell_{2}} \}.
\]
For $\phi\in B_{\sigma}^{*}$ we say that $\phi$ is supported   by the segment $\sgm$ if
$\phi=\sum_{i\in \sgm}a_{i}e^{*}_{i}$.
Let  $\mc{K}=\cup\{B_{\sigma}^{*}:\sigma\,\, \text{branch of}\,\,\tree \}$.

We set
\begin{align*}
    W=\Bigg\{\sum_{i=1}^{n}b_{i}\phi_{i}: \sum_{i=1}^{n}b_{i}^{q}\le1,&\;
                                                             \phi_{i}\in\mc{K}
                                                             \,\textnormal{
                                                             and
                                                             are supported by disjoint segments},n\in\N \Bigg\}.
\end{align*}
It follows easily that
\begin{enumerate}
\item  $W$  is symmetric and closed under projections on intervals.
\item  $W$ is a norming set  of $\jtp$ i.e.
  \[
\normm[x]=\sup\{ \phi(x): \phi\in W\}.
\]
\end{enumerate}
\begin{remark}\label{remnormjt}
  i)  For  $x\in\jtp$ and disjoint segments $\s_{1},\dots, \s_{n}$
  \[
\norm{x}^{p}\geq \sum_{i=1}^{n}\norm{x}_{\s_{i}}^{p}.
  \]
ii) Let $\varepsilon>0$, $x\in \jtp$ with $\|x\|=1$ and
$\s_i, i\in I$, be  pairwise disjoint segments,  with the property that
$\norm{x}_{\s_{i}}\ge\varepsilon$ for every $i\in I$.
Then i) yields that $\#I\le{1}/{\varepsilon^{p}}$.

iii) For every $n\in\N$, $\phi_{1},\dots,\phi_{n}\in \mc{K}$ that are supported by pairwise disjoint segments, $b_1,\ldots,b_n\in \R$ and  $x\in JT_{2,p}$ it holds that
			\[
			\abs{\sum_{i=1}^{n}b_{i}\phi_{i}(x)}\le \left(\sum_{i=1}^{n}b_{i}^{q}\right)^{1/q}\normm[x].
                      \]
                      iv)   If $x_{1},\dots, x_{n}$ is a finite block 
                      sequence and $\brn$ is a branch  then
\[                   \normm[x_{1|\brn}+\dots+x_{n|\brn}]=\left(\sum_{i=1}^{n}\norm{x_{i|\brn}}^{2}\right)^{1/2}.
                      \]
 \end{remark}
\begin{lemma}\label{suminbranches}
  Let $\brn_{1},\dots, \brn_{d}$ be distinct branches of $\tree$ and $\rho\in\N$ such that  the final segments $\brn_{i,>\rho}$ are incomparable. 
Then for
  a  finitely supported vector $x\in\jtp$ such that $\rho< \supp (x)\subset\cup\{\supp\brn_{i}:i\leq d\}$   it holds that
  \begin{equation*}
  \norm{x}=\norm{x_{|\brn_{1}}+\dots+x_{|\brn_{d}}}=
   \left(\sum_{i=1}^{d}\norm{x_{|\brn_{i}}}^{p}\right)^{1/p}.
  \end{equation*}
\end{lemma}
\begin{proof}
The incomparability   of the  final segments  $\brn_{i,\geq\min\supp x}$, $i\leq d$,  yields
  \[
 \left(\sum_{i=1}^{d}\norm{x_{|\brn_{i}}}^{p}\right)^{1/p}
\leq \norm{x_{|\brn_{1}}+\dots+x_{|\brn_{d}}}
\]
For the  reverse inequality,
let us observe that for any segment $\s$ with $\rho<\min\s$ there exists
at most one $i\leq d$ such that $\s\cap\brn_{i}\ne\emptyset$. Let  $\s_{1},\dots, \s_{k}$ be pairwise
disjoint  segments.
Since $\min\supp (x)>\rho$  we  may assume that $\s_{i}>\rho$  for every $i$.
For $i\leq d$ set  $J_{i}=\{l\leq k: x_{|\brn_{i}\cap \s_{l}}\ne 0\}$.  Then from the above observation,
  \begin{align*}
  \sum_{l=1}^{k}\norm{x_{|\s_{l}}}^{p}=    \sum_{l=1}^{k}\norm{(\sum_{i=1}^{d}x_{|\brn_{i}})_{|\s_{l}}}^{p}
    =\sum_{i=1}^{d}\sum_{l\in J_{i}}\norm{(x_{|\brn_{i}})_{|\s_{l}}}^{p}
   \leq\sum_{i=1}^{d}\norm{x_{|\brn_{i}}}^{p}.
  \end{align*}
  The above inequality implies
$ \norm{x_{|\brn_{1}}+\dots+x_{|\brn_{d}}}\leq\left(\sum_{i=1}^{d}\norm{x_{|\brn_{i}}}^{p}\right)^{1/p}$.
\end{proof}                      
\section{$\ell_{p}$-sequences in $\jtp$.}
It follows easily from the definition of the norm  that for any
subsequence  $(e_{k_{n}})_{n}$ of the basis such that $(k_{n})_{n}$
are incomparable nodes of the tree, the subspace $[e_{k_{n}}:n\in\N]$ is
isometric to  $\ell_{p}$ and in particular $(e_{k_{n}})_{n}$ is
$1$-equivalent to the unit vector basis of $\ell_{p}$. In this section we provide a
complete  characterization of the normalized block sequences which are
equivalent  to the unit vector basis of $\ell_{p}$. More precisely  we shall prove the
following theorem.

\begin{theorem}\label{jtunique}
Let $(x_{n})_{n\in\N}$ be  normalized block sequence in $\jtp$  such
that  $\lim_{n}\norm{x_{n|\sigma}}=0$ for every branch $\sigma$ of $\tree$.  Then,
there exists a subsequence $(x_{n})_{n\in N}$ of $(x_{n})_{n}$ that is
equivalent to the   unit vector basis of $\ell_{p}$. In  particular for every $\e>0$ there exists a subsequence of
$(x_{n})_{n}$ which is  $(\sqrt[q]{2}+\e)$-equivalent to the unit vector basis of $\ell_{p}$.
\end{theorem}	
The proof of the theorem relies on a refinement
of the well known result due to
I. Amemiya and T. Ito \cite{AI}, that every normalized weakly null
sequence in $JT$ contains a subsequence which is $2$-equivalent to the
usual basis of $\ell_2$,  which  appeared in \cite{AGLM}.
We break up the proof of the theorem into several
lemmas and we start with the following Ramsey-type
                result.
\begin{lemma}\label{jtramsey}
                  Let $(x_n)_{n\in\N}$ be a  normalized block sequence  in
                $\jtp$ such that $\lim_{n}\normm[x_{n|\sigma}]=0$ for every
                branch $\sigma$.
Then, for every $\varepsilon>0$ and $n_{0}\in\N$, there exists an $L\in[\N]$ such that for every  segment $\s$  with $\min \s\le n_{0}$ there exists at most one $n\in L$ with the property that $\normm[x_{n|\s}]\ge\varepsilon$.
		\end{lemma}
\begin{proof}
  Set $[p_{n,}q_{n}]=\ran(x_{n})$, $n\in\N$. If the conclusion is false,
  using Ramsey's theorem  we may assume that there exists $L\in[\N]^\infty$
  such that, for every pair $m<n$ in $L$, there exists  a  segment
  $\s_{m,n}$ such that $\min \s_{m,n}\leq n_{0}$,
  $\normm[x_{m|\s_{m,n}}]\ge\varepsilon$ and $\normm[x_{n|\s_{m,n}}]\ge\varepsilon$.
For every segment $\s_{m,n}$ we denote by $\bar{\s}_{m,n}$ the unique
initial segment determined by $\s_{m,n}$.
\\\par

{\em\underline{Claim 1}} : For every $n\in L$ it holds  $\#\{ \bar{\s}_{m,n}|_{[1,p_{n})}:m\in L,m<n\}\le[\e^{-p}]$.
\begin{proof}[Proof of Claim  1]\renewcommand{\qedsymbol}{}
If $\# \{ \bar{\s}_{m,n}|_{[1,p_{n})}:m\in L, m<n \}>[\e^{-p}]$ for some $n\in L$,   it follows
                        that there exist 
 at least $[\e^{-p}]+1$ pairwise disjoint  segments
 $\bar{\s}_{m,n}|_{[p_{n},q_n]}$ such that   $\normm[x_{n|\bar{\s}_{m,n|[p_{n},q_{n}]}}]\ge\varepsilon$. This contradicts item (ii) of Remark \ref{remnormjt}.
		\end{proof}
For every $n\in L$, let $\{\bar{\s}_{m,n}|_{[1,p_{n})}:m\in
  L,m<n\}=\{s_{n,1},\ldots,s_{n,\mu(n)} \}$ with $\mu(n)\le[\e^{-p}]=\mu$. Set
\[L_{i}^{n}=\{m\in L:m<n\;\text{and } \normm[x_{m|s_{n,i}}]\geq\e \},\,\,\,
  1\le i\le \mu(n)\]
 and $L^{n}_{i}=\emptyset$, for $\mu(n)<i\le \mu$. Notice that
  $\{m\in L:m<n\}=\cup_{i=1}^{\mu}L_{i}^{n}$ for all $n\in L$. Passing to a
  further subsequence we may assume that, for every $1\le i\le \mu$,
  $(L_{i}^{n})_{n\in {M}}$ converges point-wise  and we denote that limit
  by $L_{i}$. Then it is easy to see that $L=\cup_{i=1}^\mu L_{i}$ and
  hence some $L_{i_0}$ is an infinite subset of $L$. Moreover  for
  every $n\in M$ and for all $m<n$,  $m\in L_{i_0}$,
there exists a  segment $s_{n,i_{0}}$ 
such that 
$\normm[x_{m|s_{n,i_{0}}}]\ge\varepsilon$.
Then
there exists $P\in [M]^{\infty}$
such  that $(s_{n,i_{0}})_{n\in P}$ converges point-wise to a branch or an
initial segment.
\\

{\em\underline{Claim 2}:}  $(\s_{n,i_{0}})_{n\in P}$  converges to a
branch.
\begin{proof}[Proof of Claim 2.]\renewcommand{\qedsymbol}{}
 On the contrary,  assume that there exists an initial segment $\s_{0}$
  such that $\s_{n,i_{0}}\xrightarrow[P\ni n]{} \s_{0}$.   Since
  $\tree$ is an infinitely branching  tree we get  $P_{0}\in[P]^{\infty}$ such
  that the segments  $\s_{n,i_{0}|(\max\s_{0},+\infty)}, n\in P_{0}$ are
   incomparable.
  Then for every   $m\in L_{i_{0}}$ with $\min\supp(x_{m})>\max
  (\s_{0})$  it follows that  $\norm{x_{m|(\s_{n,i_{0}}-\s_{0}})}\geq\e$
  for infinitely many $n\in P_{0}$.
 This contradicts item (ii) of Remark \ref{remnormjt}.
\end{proof}
Let 
$\brn$  be the branch such that  $\s_{n,i_{0}}\xrightarrow[P\ni
n]{}\brn$.   It follows 
that   $\normm[x_{n|\sigma}]\ge\varepsilon$ for  infinitely many $n$. This yields a
contradiction and completes the proof of the  lemma.
 \end{proof}
\begin{lemma}\label{l13}
  Let $\varepsilon>0$  and  $(x_n)_n$ be a normalized block sequence,
  such that
$\lim_{n}\normm[x_{n|\sigma}]=0$ for every branch $\sigma$. Let $q_{n}=\max\supp(x_{n})$, $n\in\N$.
Then there exist an increasing sequence $(n_k)_k$ in $\N$ and a decreasing sequence $(\varepsilon_k)_k$ of positive reals such that
	\begin{enumerate}[resume]
		\item[(1)] For every $k\in\N$ and every segment $\s$  with $\min \s\leq q_{n_{k}}$ there exists at most one $k'>k$ such that $\normm[x_{n_{k^{\prime}}|\s}]\ge\varepsilon_{k}$.
		\item[(2)] $\sum_{k=1}^{\infty} q_{n_k}\sum_{i=k}^\infty (i+1)\varepsilon_{i}<\varepsilon$.
	\end{enumerate}
\end{lemma}
\begin{proof}
  Let $(\delta_n)_n$ be a sequence of positive reals such that $\sum_{n=1}^\infty\delta_n<\varepsilon$. We will construct $(n_k)_k$ and $(\varepsilon_k)_k$ by induction  as follows. We set $n_1=1$ and $L_1=\N$ and choose $\varepsilon_1$ such that $q_{1}2\varepsilon_1<\delta_1$.
  Suppose that $n_1<\dots <n_k$, $\varepsilon_1>\dots >\varepsilon_k$
and $L_{1}\supset \dots\supset L_{k}$
have been chosen  such that for every $j\leq k$,
\begin{enumerate}
\item[(i)]
  For every segment $\s$ with $\min\s\leq q_{n_j}$ there exists at most one $n\in L_{j+1}$ such that $\norm{x_{n|\s}}\geq \e_{j}$,
\item[(ii)]$q_{n_{j}}(j+1)\varepsilon_{j}<\delta_{j}$
\item[(iii)] $q_{n_m}\sum_{i=m}^{j}(i+1)\varepsilon_{i}<\delta_m$ for every $m\le j$.
\end{enumerate}
Choose $n_{k+1}\in L_k$ with $n_{k+1}>n_k$ and $\varepsilon_{k+1}<\varepsilon_k$ such that
\[
q_{n_{k+1}}(k+2)\varepsilon_{k+1}<\delta_{k+1}\quad\text{and}\quad q_{n_m}\sum_{i=m}^{k+1}(i+1)\varepsilon_{i}<\delta_m \,\,\text{for all}\, m\leq k.
\]
Using  Lemma \ref{jtramsey} we get  $L_{k+1}\in [L_{k}]^\infty$ such that for every segment $\s$ with $\min \s\leq q_{k+1}$ there exist at most one $n\in L_{k+1}$ with $\normm[x_{n|\s}]\ge\varepsilon_{k+1}$.

It is easy to see that $(n_k)_k$ and $(\varepsilon_k)_k$ are as desired.
\end{proof}

\begin{lemma}\label{lpequiv}
	Let $\varepsilon>0$, $(\varepsilon_n)_n$ be a decreasing sequence of positive
        reals and 
 $(x_n)_n$ be a normalized block sequence. For every $n\in\N$ let
 $q_{n}=\max\supp x_{n}$.
Assume also that
\begin{enumerate}[resume]
		\item[i)]for every $n\in\N$ and every  segment $\s$ with $\min\s\leqq q_{n}$ there exists at most one $m>n$ such that $\normm[x_{m|\s}]\ge\varepsilon_{n}$  and
\item[ii)] $\sum_{n=1}^\infty q_{n}\sum_{i=n}^\infty (i+1)\varepsilon_{i}<\varepsilon$.
	\end{enumerate}
	Then for every  $n\in\N$ and every choice of scalars $a_1,\ldots,a_n$ we have that
	\[
\Big(\sum_{i=1}^{n}\abs{a_{i}}^p\Big)^{\frac{1}{p}}\le\Big\|\sum_{i=1}^{n}a_{i}x_{i}\Big\| \le (\sqrt[q]{2}+\varepsilon)\Big(\sum_{i=1}^{n}\abs{a_{i}}^{p}\Big)^{\frac{1}{p}}.
	\]
\end{lemma}
\begin{proof} Let us first observe that for every $n\in\N$ and every
  segment $\s$ with
  $q_{n-1}< \min\s \le q_{n}$,  the following hold due to i).
  \begin{enumerate}
	\item[iii)] $\#\{i>n:\norm{x_{i|\s}}\ge\varepsilon_{n} \}\le 1$.
		\item[iv)] $\#\{i>n:\varepsilon_{k}\leq \normm[x_{i|\s}]<\varepsilon_{k-1} \}\le k$ for every $k>n$.
                \end{enumerate}
               	Now for each $1\le i\le n$, there exists
        $\phi_{i}=\sum_{j=1}\beta_{j}\phi_{i,j}\in W$ supported by  $ \ran(x_{i})$ such that
        $\phi_{i}(x_{i})=\norm{x_{i}}$.
If  $(b_{i})_{i=1}^{n}$ is the conjugate      sequence of $(a_{i})_{i=1}^{n}$ w.r the $\ell_{p}$-norm,   it follows
	\[
          \Big\|\sum_{i=1}^{n}a_{i}x_{i}\Big\|\ge
          \sum_{i=1}^{n}b_{i}\phi_{i}(\sum_{i=1}^{n}a_{i}x_{i})
\geq\sum_{i=1}^{n}b_{i}a_{i}\norm{x_{i}}=
     \Big(\sum_{i=1}^{n}\abs{a_{i}}^{p}\Big)^{\frac{1}{p}}.
	\]
        For the reverse inequality, let  $\phi_{1},\dots,\phi_{m}\in\mc{K}$ supported by  pairwise disjoint segments  $s_{1},\dots,s_{m}$,
  and $b_1,\ldots,b_m$ reals with $\sum_{j=1}^{m}b_{j}^{q}\le1$. For given $1\le j\le m$, we will
        denote by
\begin{enumerate}
\item[v)]       $i_{j,1}$ the unique $1\le i\le n$ such that
  $q_{i-1}< \min s_{j}\leqq q_{i}$,
\item[vi)]   $i_{j,2}$ the unique, if there exists, $i_{j,1}<i\le n$ such that $\norm{x_{i|s_{j}}}\ge\varepsilon_{i_{j,1}}$.
\end{enumerate}
We  set $G_{i}=\{ j\leq m: i_{j,1}=i\}$ and we have that $\{1,\ldots,m \}=\cup_{i=1}^{n}G_{i}$.
 We also set $J_{i}=G_{i}\cup\{j\leq m :\;i_{j,2}=i\}$, for
  $1\le i \le n$.  It follows,
  \begin{align}
     \Big|(\sum_{j=1}^{m}b_{j}\phi_{j})&\Big(\sum_{i=1}^{n}a_{i}x_{i}\Big)\Big|=
   \Big|\sum_{j=1}^{m}b_{j}\phi_{j}\Big(\sum_{i\geq
                                   i_{j,1}}^{n}a_{i}x_{i}\Big)\Big|
                                   \notag
   \\
&\leq  \sum_{j=1}^{m}\abs{b_{j}}\left(\abs{a_{i_{j,1}}}\norm{x_{i_{j},1|s_{j}}}
        + \abs{a_{i_{j,2}}}\norm{x_{i_{j},2|s_{j}}}+\sum_{i_{j,1}<i
    \neq i_{j,2}}\abs{a_{i}}\norm{x_{i|s_{j}}}\right)
        \notag\\
    &=\sum_{i=1}^{n}\left(\abs{a_{i}}\sum_{j\in J_{i}}\abs{b_{j}}\norm{x_{i|s_{j}}}+
      \sum_{j\in G_{i}}\abs{b_{j}}\sum_{i<l
    \neq i_{j,2}}\abs{a_{l}}\norm{x_{l|s_{j}}}\right)\label{secondterm}
        \end{align}
 Note that by iii), each $j$ appears in  at most  two $J_{i}$,
 so
 \begin{equation}
   \label{eq:21}
   \sum_{i=1}^{n}\sum_{j\in
     J_{i}}\abs{b_{j}}^{q}\le2\sum_{j=1}^{m}\abs{b_{j}}^{q}.
 \end{equation}
   It follows
 \begin{align}
\sum_{i=1}^{n}\abs{a_{i}}\sum_{j\in J_{i}}\abs{b_{j}}\norm{x_{i|s_{j}}}
&\le\Big(\sum_{i=1}^{n}\abs{a_{i}}^{p} \Big)^{\frac{1}{p}}
\Big(\sum_{i=1}^{n}\big(\sum_{j\in J_{i}}\abs{b_{j}}\norm{x_{i|s_{j}}}\big)^{q} \Big)^{\frac{1}{q}}
\notag           \\
          &\le\Big(\sum_{i=1}^{n}\abs{a_{i}}^{p}
            \Big)^{\frac{1}{p}}\Big(\sum_{i=1}^{n}\sum_{j\in
            J_{i}}\abs{b_{j}}^{q}\Big)^{\frac{1}{q}}
\notag            
   \\
&\le\sqrt[q]{2}\Big(\sum_{i=1}^{n}\abs{a_{i}}^{p} \Big)^{\frac{1}{p}
        }\,\,\,\text{by}\,\, \eqref{eq:21}.
        \label{firstest}
 \end{align}	
 To get un upper bound for the remaining term of \eqref{secondterm}
 notice  that    iv) implies
 \begin{equation*}
 \sum_{i<l\ne
   i_{j,2}}\abs{a_{l}}\norm{x_{l|s_{j}}}<\max_{k}\abs{a_{k}}\sum_{l=i}^\infty
 (l+1)\varepsilon_{l}\,\,\
 \text{for any $j\in G_{i}$}.  
 \end{equation*}
Using that     $\#G_{i}\leqq q_{i}$ and 
ii) we get,
\begin{equation}
  \label{eq:22}
  \sum_{i=1}^{n}\sum_{j\in G_{i}}\abs{b_{j}}\sum_{l>i,l\ne i_{j,2}}\abs{a_{l}}\norm{x_{l|s_{j}}}
<\max_{k}\abs{a_{k}}\sum_{i=1}^{n}q_{i}\sum_{l\geq i}(l+1)\e_{l}
\leq
\epsilon\max_{k}\abs{a_{k}}.
\end{equation}
Combining  \eqref{firstest} and \eqref{eq:22}  we get the upper  $\ell_{p}$-estimation.
\end{proof}
\begin{proof}[Proof of Theorem \ref{jtunique}]
  Since $\norm{x_{n|\brn}}\xrightarrow[n]{}0$ for every branch $\brn$,
  passing to a subsequence we may assume that  (1) and (2) of Lemma~\ref{l13}
holds. Lemma~\ref{lpequiv} yields that this subsequence
  is equivalent to the unit vector basis of $\ell_{p}$.
\end{proof}

\begin{remark}
Let us note   that the converse of Theorem~\ref{jtunique}  holds, i.e.,
if  $(x_{n})_{n}$ is a normalized block sequence  equivalent to the unit vector basis of $\ell_{p}$ then
$\lim_{n}\norm{x_{n|\brn}}=0$ for every branch.  Indeed, if
$\norm{x_{n_{k|\brn}}}\geq\delta$ for infinitely many $k$, then
\[
  \norm{\sum_{k=1}^{m}a_{k}x_{n_{k}}}\geq
  \norm{\sum_{k=1}^{m}a_{k}x_{n_{k}|\brn}}=
  \left(\sum_{k=1}^{m}a_{k}^{2}\norm{x_{n_{k}|\brn}}^{2}\right)^{1/2}\geq\delta  \left(\sum_{k=1}^{m}a_{k}^{2}\right)^{1/2}.
\]
Since $2<p$ it follows that $(x_{n})_{n}$ cannot be equivalent to the
unit vector basis of $\ell_{p}$, a contradiction
\end{remark}

\section{The $\ell_{p}$
  vector of a block sequence}
In this section we introduce the  $\ell_{p}$   vector of  a
block sequence. This is a
critical concept that will be used  repeatedly in the main results of
this part.
 \begin{definition}
  Let  $(x_{n})_{n\in\N}$ be a normalized block sequence in $\jtp$.
   A decreasing sequence  $(c_{i})_{i\in N}$, where  $N$ is either an
   initial finite interval of $\N$  or the set of the natural numbers,
is said to be the
   $\ell_{p}$-vector of  $(x_{n})_{n\in\N}$ with associated set of branches
 the set   $\mc{B}=(\brn_{i})_{i\in N}$,
   if $(c_{i})_{i\in N}\in B_{\ell_{p}}$,  
   \begin{equation*}
     \lim_{n}\norm{x_{n|\brn_{i}}}=c_{i}>0 \,\,\text{for every $i\in N$ and }
  \lim_{n}\norm{x_{n|\brn}}=0 \,\,\text{for every $\brn\ne \brn_{i}$}, i\in N.
\end{equation*}
If $\lim_{n}\norm{x_{n|\brn}}=0$ for every branch $\brn,$ then the
$\ell_{p}$-vector of  $(x_{n})_{n}$ is the zero one and the associated
set of branches  $\mc{B}$ is the empty one.
 \end{definition}
\begin{proposition}\label{lpdecomp}
 Let $(x_{n})_{n\in\N}$ be a normalized block sequence   in $\jtp$. Then
 $(x_{n})_{n\in\N}$  either has a subsequence equivalent to the unit vector basis of
 $\ell_{p}$ or has  a subsequence accepting an $\ell_{p}$-vector.
\end{proposition}
The proof relies on the following lemma due to Lindenstrauss-Stegall \cite{LS}, whose proof is presented for sake of completeness.
                \begin{lemma}[\cite{LS}]\label{lslemma}
Let $(x_{n})_{n\in\N}$  be a normalized block sequence in $\jtp$. Then there exist  a
subsequence $(x_{n})_{n\in L}$  of  $(x_{n})_{n\in\N}$  and an at most
countable set of distinct branches $(\sigma_{i})_{i\in N}$ such that
\[\lim_{n}\norm{x_{n|\sigma_{i}}}=c_{i}>0\,\,\text{for every }\,i\in N \,\,\text{and}\,\,
  \lim_{n}\norm{x_{n|\sigma}}=0\,\,\text{for every }\,  \sigma\ne \sigma_{i}, i\in N.
\]
Moreover $(c_{i})_{i\in N}\in B_{\ell_{p}}$  and  the sequence $(c_{i})_{i\in N}$ may chosen to be decreasing.
\end{lemma}
\begin{proof}
  Let  $\e>0$.
From the definition of the norm it follows that for every choice of distinct branches $\sigma_{1},\dots, \sigma_{k}$  it holds  that
\begin{equation}
  \label{eq:1a} 
\limsup_{n}\sum_{i=1}^{k}\norm{x_{n|\sigma_{i}}}^{p}\leq \sup_{n}\normm[x_{n}]^{p}\leq 1.
\end{equation}
Combining  \eqref{eq:1a} and  Remark~\ref{remnormjt} (ii) we get  that for every subsequence $(x_{n})_{n\in  L}$ of $(x_{n})_{n}$ there exist at most $k \in [0, \e^{-p}]$
distinct branches  $\sigma_{1},\dots,  \sigma_{k}$ such that
$\lim_{n\in L}\norm{x_{n|\sigma_{i}}}=c_{i}\geq\e$  and for every
branch $\sigma \ne \sigma_{j}$ it holds $\limsup_{n\in L}\norm{x_{n|\sigma}}<\e$.

Let $\e_{j}\searrow 0$. Applying the above for  $\e_{1},\e_{2},\dots$ we get
$L_{1}\supset L_{2}\supset\dots \supset L_{j}\supset L_{j+1}\supset\dots$    and
$\sigma_{j,1},\dots, \sigma_{j,k_{j}}$ distinct branches, $k_{j}\in [0,\e_{j}^{-p}]$, $j=1,2\dots$,  
such that
\begin{equation*}
  \begin{split}
 \lim_{n\in L_{j}}\norm{x_{n|\sigma_{j,k}}}=c_{j,k}\geq\e_{j}\,\text{for
   all}\, k\leq k_{j}\,\,\text{and}\\
 \,\, \limsup_{n\in
   L_{j}}\norm{x_{n|\sigma}}<\e_{j}\,\,\forall \sigma\ne \sigma_{i,k}, k=1,\dots,k_{i}, i=1,\dots,j.
 \end{split}
\end{equation*}
Reordering $(c_{j,k})_{k}$  we get  $c_{j-1,k_{j-1}}>c_{j,1}\geq\dots \geq c_{j,k_{j}}$.
The inductive choice stops in the case we cannot find $L\in [L_{j}]$ and
a branch $\brn$ such  that $\lim_{n\in L}\norm{x_{n|\brn}}>0$.
Enumerate the $c_{j,k}\ne0 $ and the associated  $\brn_{j,k,}$
as
$(c_{i})_{i\in N}$ and $(\brn_{i})_{i\in N}$ respectively, following the
natural order 
of $c_{j,k}, k\leq k_{j}, j=1,2,\dots$.
Let $\mc{B}_{0}=\{\sigma_{i}: i \in N\}$  and
$(x_{n})_{n\in L}$
be a  diagonal subsequence of $(x_{n})_{n\in L_{j}}$.     
Once we show that $(c_{i})_{i\in N}\in B_{\ell_{p}}$  we have the result for
$(x_{n})_{n\in L}$ and  $\mc{B}_{0}$.
Assume that $\sum_{i}\abs{c_{i}}^{p}>1$.
Then  $\sum_{i\in F}\abs{c_{i}}^{p}>1$ for some finite set  $F$.
Since the  branches $(\sigma_{i})_{i}$ are distinct
there exists $\rho\in\N$ such that the final  segments  $\sigma_{i,>\rho}, i\in F$ are
incomparable. Then,
\[
\lim_{n\in L}\norm{x_{n}}^{p}\geq \lim_{n\in L}\sum_{i\in F}\norm{x_{n|\sigma_{i,>\rho}}}^{p}\geq
\sum_{i\in F}c_{i}^{p}>
  1\quad\text{a contradiction}.
\]
\end{proof}
\begin{proof}[Proof of Proposition~\ref{lpdecomp}]
 If   $\lim_{n}\norm{x_{n|\brn}}=0$ for every branch $\brn$ then the
$\ell_{p}$-vector of  $(x_{n})_{n}$ is the zero vector.
Moreover from Theorem~\ref{jtunique} it follows that   $(x_{n})_{n\in\N}$ has a
subsequence equivalent to the unit vector basis of $\ell_{p}$.
Assume that   it does not hold that $\lim_{n}\norm{x_{n|\brn}}=0$ for every branch $\brn$ .
  From the previous lemma we get a subsequence  $(x_{n})_{n\in L}$ of
  $(x_{n})$, an at most  countable  of distinct branches
  $(\sigma_{i})_{i\in N}$ and a decreasing  sequence $(c_{i})_{i\in N}\in B_{\ell_{p}}$ such that
  \begin{equation*}
 \lim_{n}\norm{x_{n|\sigma_{i}}}= c_{i}>0\,\,\text{for every}\, i\,\,\text{and}\,\,
  \lim_{n}\norm{x_{n|\sigma}}=0\,\,\text{for every }\, \sigma\ne  \sigma_{i}, i\in N.
\end{equation*}
It follows that $(x_{n})_{n\in L}$ has $(c_{i})_{i\in N}$ as
$\ell_{p}$-vector with  associated  set of branches  the set  $(\sigma_{i})_{i\in N}$.
\end{proof}
\begin{remark}\label{remI3.4}
Let $(x_{n})_{n\in N}$ be a sequence in $JT_{2,p}$ having  $(c_{i})_{i\in N}$ as
an $\ell_{p}$-vector  and let $(\brn_{i})_{i\in N}$  be the associated set of brances .

i) For every $\lambda\ne 0$  the sequence  $(\lambda x_{n})_{n}$ has
$(|\lambda| c_{i})_{i\in N}$ as an 
$\ell_{p}$-vector   and associated set of brances the branches
$(\brn_{i})_{i\in N}$.

ii)      It is easy to check that the $\ell_{p}$-norm of $(c_{i})_{i\in N}$
satisfies $\norm{(c_{i})_{i\in N}}_{\ell_{p}}\leq\liminf\norm{x_{n}}$.

iii)   Setting
$\text{T}_{0}=\cup\{\brn_{i}:i\in N\}$  and  
$y_{n}=x_{n}-x_{n|\text{T}_{0}}$ it follows that
$\lim_{n}\norm{y_{n|\sigma}}=0$ for every branch $\sigma$.  In the case where 
$\limsup_{n}\norm{y_{n}}>0$ passing to a subsequence we may assume
that   $0<c\leq\norm{y_{n}}\leq 1$ for every $n$. From
Theorem~\ref{jtunique} we get a subsequence  $(y_{n})_{n\in  N}$ which
is  equivalent to the unit vector basis of $\ell_{p}$. 
\end{remark}
For the rest of  Part I, we will assume that a non-zero
$\ell_{p}$-vector  is an infinitely supported vector of $\ell_{p}$.  The
cases that  the $\ell_{p}$-vector  is  finitely supported  are treated in
similar manner and in most cases  the proofs are simpler.

\begin{lemma}\label{block}
Let $(x_{n})_{n}$ be a normalized weakly  null sequence with
a non-zero $\ell_{p}$-vector  $(c_{i})_{i\in \N}$ and let $\mc{B}$  be the
associated set of branches. Then, if $(y_{n})_{n}$
is a  block sequence  such that $\supp (y_{n})\subset\supp(x_{n})$ for every
$n$ and $\lim_{n}\norm{x_{n}-y_{n}}=0$, the  sequence $(y_{n})_{n}$ also has
  $(c_{i})_{i\in \N}$ as an $\ell_{p}$-vector with the same set of
associated branches.
  \end{lemma}
  \begin{proof}
  The proof is an immediate consequence    of the triangle inequality
  and the fact that $\norm{x_{|s}}\leq\norm{x}$  for every  segment $s$.
\end{proof}
\begin{lemma}\label{stabilization}
  Let  $(x_{n})_{n}$ be a normalized block sequence with a non-zero
  $\ell_{p}$-vector  $(c_{i})_{i\in \N}$
  and let $(\brn_{i})_{i\in \N}$ be the associated set of branches.
  Let $(I_{k})_{k}$ be an increasing sequence of  initial intervals of $\N$.
  Then, there exist  $n_{k}\nearrow+\infty$
  and $(l_{k})_{k}\subset\N$
 such that
 \begin{enumerate}
 \item[(i)]  $    (1-\e_{k})c_{i}\leq  \norm{x_{n|\brn_{i}}}\leq (1+\e_{k})c_{i}$ for every
   $i\in I_{k}$ and $n\geq n_{k}$,
 \item[(ii)]   the final  segments $\brn_{i,>l_{k}}, i\in I_{k}$ are
    incomparable and $\minsupp x_{n_{k}}>l_{k}$.  
 \end{enumerate}
 In particular, setting    $\mc{B}_{I_k}=\cup\{\brn_{i}:i\in I_{k}\}$  the
 following  holds:
 \begin{equation*}
   \lim_{k}\norm{x_{n_{k}|\mc{B}_{I_{k}}}}=\norm{(c_{i})_{i\in \N}}_{\ell_{p}}\,\,\text{and}\,\,
  \lim_{k}\norm{(x_{n_{k}}-x_{n_{k}|\mc{B}_{I_k}})_{|\brn}}=0\,\, \text{for every branch}\, \brn.
\end{equation*}
Moreover the   sequence $(x_{n_{k}}-x_{n_{k}|\mc{B}_{I_k}})_{k}$ 
admits upper $\ell_{2}$-estimate with constant $2$.
In particular if it does not converges to zero in norm
it is equivalent to the unit vector basis  of $\ell_{p}$.
\end{lemma}
\begin{proof}
Using the definition of the $\ell_{p}$-vector,  we choose inductively $n_{k}\nearrow+\infty$ and
$(l_{k})_{k}\subset\N$
 such that  (i) and (ii) holds.  Set  $\mc{B}_{I_{k}}=\cup\{\brn_{i}:i\in
 I_{k}\}$ and $\bar{x}_{n_{k}}=x_{n_{k}|\mc{B}_{I_{k}}}$.
From (ii)  and  Lemma~\ref{suminbranches} it
 follows that
 \begin{equation*}
 \norm{\bar{x}_{n_{k}}}=\left(\sum_{i\in
       I_{k}}\norm{x_{n_{k}|\brn_{i}}}^{p}\right)^{1/p}.
\end{equation*}
Using  (i) and that $\# I_{k}\nearrow+\infty$ it follows that
\begin{equation*}
(1-\e_{k})\left(\sum_{i\in I_{k}}c_{i}^{p}\right)^{1/p}
\leq\norm{\bar{x}_{n_k}}\leq (1+\e_{k})\left(\sum_{i\in
  I_{k}}c_{i}^{p}\right)^{1/p}
\Rightarrow\norm{\bar{x}_{n_k}}\xrightarrow[k]{} a.
\end{equation*}
Note also that $(\bar{x}_{n_{k}})_{k}$ has as $\ell_{p}$-vector the
$(c_{i})_{i\in \N}$ with the same set of associated brances.

Indeed it readily follows that for every
$\brn_{l}$,  $\lim_{k}\norm{\bar{x}_{n_{k}|\brn_{l}}}=c_{l}$.
If  $\brn\ne\brn_{l}$ for every $l$, then from the definition of the
$\ell_{p}$-vector it follows,
\[
\norm{\bar{x}_{n_{k}|\brn}}\leq\norm{x_{n_{k}|\brn}}\xrightarrow[k]{}0.
\]
Let 
  $y_{n_{k}}=x_{n_{k}}-\bar{x}_{n_{k}}$.
  
 \textit{Claim}.  $\norm{y_{n_{k}|\brn}}\to 0$ for every branch $\brn$.

 Indeed  if $\brn=\brn_{l}$ for some  $l$, then  $y_{n_{k}|\brn_{l}}=0$
 for every $k>l$.  If $\brn\ne\brn_{l}$ for every $l$  then
 \[
\norm{y_{n_{k}|\brn}}=\norm{(x_{n_{k}}-\bar{x}_{n_{k}})_{|\brn}}\leq\norm{x_{n_{k}|\brn}}\xrightarrow[k]{} 0
\]
from the properties of the $\ell_{p}$-vector.

In the case  where $\limsup \norm{y_{n_{k}}}>0$, from Theorem~\ref{jtunique}
passing to  a subsequence  we may assume that $(y_{n_{k}})_{k\in\N}$  equivalent to the
unit vector basis of $\ell_{p}$.  If $\lim \norm{y_{n_{k}}}=0$ passing to
a subsequence  we get that $(y_{n_{k}})_{k\in\N}$  admits upper
$\ell_{2}$-estimate with constant $2$.

From  the above we get  that the  subsequence
$(x_{n_{k}})_{k}$ satisfies the conclusion of the Lemma.
\end{proof}
\begin{lemma}\label{l2averages}
  Let  $(x_{n})_{n}$ be a normalized block sequence
having a non-zero   $\ell_{p}$-vector    $(c_{i})_{i\in \N}$  and let 
$(\brn_{i})_{i\in \N}$ be the associated set of branches.
  Let  $(y_{n})_{n\in\N}$ be a block sequence of $(x_{n})_{n}$ such that
  $y_{n}=\sum_{k\in F_{n}}a_{k}x_{k}$
and  $(a_{k})_{k\in F_{n}}\in S_{\ell_{2}}$. Then,   $(y_{n})_{n}$
has $(c_{i})_{i\in \N}$
as an $\ell_{p}$-vector  with  the same set of associated branches.
\end{lemma}
\begin{proof}
Let $\e>0$ and $\brn\ne\brn_{i}$ for every $i$.  From the definition of the
$\ell_{p}$-vector it follows that  there exists $n_{0}\in\N$ such that
\begin{equation*}
  \norm{x_{n|\brn}}<\e\,\,\text{for all}
  \,\, n\geq n_{0}.
\end{equation*}
From  the above inequality and Remark~\ref{remnormjt} (iv) it follows, using that   $(a_{k})_{k\in F_{n}}\in S_{\ell_{2}}$,
\[
\norm{y_{n|\brn}}=\norm{\sum_{k\in F_{n}}a_{k}x_{k|\brn}}=\left(\sum_{k\in
    F_{n}}a_{k}^{2}\norm{x_{k|\brn}}^{2}\right)^{1/2}\leq\e \,\,\text{for all}
  \,\, n\geq n_{0}.
\]
This yields that $\lim_{n}\norm{y_{n|\brn}}=0$. 
We show  now  that  for every $i\in\N$ it holds
$\lim_{n}\norm{y_{n|\brn_{i}}}=c_{i}$.  For  $i\in \N$ there exists  $n(i)\in\N$  such that
\begin{equation}
  \label{eq:23}
c_{i}-\e\leq \norm{x_{n|\brn_{i}}}\leq c_{i}+\e\,\,\text{for all}\,\, n\geq n(i).
\end{equation}
Using once more Remark~\ref{remnormjt} (iv)  and
\eqref{eq:23} it follows,
\[
c_{i}-\e\leq \norm{y_{n|\brn_{i}}}=\left(\sum_{k\in
    F_{n}}a_{k}^{2}\norm{x_{k|\brn_{i}}}^{2}\right)^{1/2}\leq c_{i}+\e\,\,
\text{for all}\,\, n\geq n(i).
\]
From the above it follows that  $(y_{n})_{n\in\N}$
has $(c_{i})_{i \in \N}$
as $\ell_{p}$-vector  with  the same set of associated branches.
\end{proof}
We show now that we can choose a subsequence
with the   additional property the norm of the $\ell_{2}$-averages 
converges to the $\ell_{p}$-norm of the  $\ell_{p}$-vector.
For the proof we shall use  the following  elementary lemma concerning elements of $\ell_{p}$.
\begin{lemma}\label{lpsplit}
For every   $(c_{n})_{n\in\N}\in B_{\ell_{p}}$ there exists  a strictly
  increasing sequence  $(I_{k})_{k\in\N}$ of initial segments of $\N$
  such that setting $d_{k}=(c_{n})_{n\in I_{k}\setminus I_{k-1}}$, $I_{0}=\emptyset$,
  $k\in\N$, the following holds:

  For every $\e \in (0,1)$ there exists $k_{0}\in\N$  such that
  \[\sum_{k=k_{0}+1}^{\infty}\norm{d_{k}}_{\ell_{p}}\leq \e\norm{(c_{n})_{n}}_{\ell_{p}}.
 \]
\end{lemma}
\begin{proof}
  Let   $A=\norm{(c_{n})_{n}}_{\ell_{p}}$  and  $\e_{n}\searrow 0$, $\e_{0}=1$,
  such that $\sum_{n\geq1}\e_{n}<1$.  Inductively choose  a strictly
  increasing sequence $(I_{k})_{k}$ of initial segments  of $\N$ such that
  \[
\sum_{n\in\N\setminus I_{k}}\abs{c_{n}}^{p}<\e^{p}_{k}A^{p}.
\]
Let $\e>0$.  Choose  $k_{0}\in\N$ such that   $\sum_{k=k_{0}}^{\infty}\e_{k}<\e$.  Then

\[\norm{d_{k}}_{\ell_{p}}\le\e_{k-1}A\Rightarrow
\sum_{k=k_{0}+1}^{\infty}\norm{d_{k}}_{\ell_{p}}\leq \sum_{k=k_{0}}^{\infty}\e_{k}A\leq \e A.
\]
\end{proof}
\begin{lemma}\label{lemmap1}
 Let $(x_{n})_{n}$ be a normalized block sequence  with an
 $\ell_{p}$-vector $(c_{i})_{i\in \N}$  such that
 $\norm{(c_{i})_{i\in \N}}_{\ell_{p}}=a\in ( 0,1)$. Then, there exists a subsequence
 $(x_{n_{i}})_{i\in\N}$  of $(x_{n})_{n\in\N}$  such that setting
 $w_{m}=\frac{1}{\sqrt{F_{m}}}\sum_{i\in F_{m}}x_{n_{i}}$, with $F_{1}<F_{2}<\dots$ and $\# F_{m}\nearrow+\infty$, we have that
 $\norm{w_{m}}\xrightarrow[m]{}a$ and $(w_{m})_{m}$
 admits $(c_{i})_{i\in \N}$   as an  $\ell_{p}$-vector  with the same set of
 associated branches as $(x_{n})_{n}$.
\end{lemma}
\begin{proof}
  From Lemma~\ref{l2averages}  and  Remark~\ref{remI3.4}(ii)  we have that
  $a\leq\liminf\norm{w_{n}}$  for any sequence $(w_{n})_{n}$ of
  $\ell_{2}$-averages of $\sxn$.
So   it is enough to find a subsequence of
  $(x_{n})_{n}$ such that for   any sequence $(w_{n})_{n}$ of successive $\ell_{2}$-averages
  of increasing length of the subsequence,  it holds
  $\limsup\norm{w_{n}}\leq a$.
Let $(I_{k})_{k\in\N}$ be the initial segments of  $\N$ we get from Lemma~\ref{lpsplit}.
  Let also $(x_{n_{k}})_{k}$ be a subsequence of $(x_{n})_{n\in\N}$
  satisfying the conclusion of Lemma~\ref{stabilization} for the
  intervals $I_{k}$.
  Set $\bar{x}_{n_{k}}=x_{n_{k}| \mc{B}_{k}} $ and
  $y_{n_{k}}=x_{n_{k}}-\bar{x}_{n_{k}}$
  where  $\mc{B}_{I_{k}}=\cup\{\brn_{i}:i\in I_{k}\}$.
Let $F_{1}<F_{2}<\dots$ be successive subsets  of $\N$  such that  $\#
F_{m}\nearrow+\infty$ and $w_{m}=  \frac{1}{\sqrt{F_{m}}}\sum_{i\in F_{m}}x_{n_{i}}$.
Define
\[z_{m}=
  \frac{1}{\sqrt{F_{m}}}\sum_{i\in F_{m}}y_{n_{i}}\,\,\text{and}\,\, u_{m}=\frac{1}{\sqrt{F_{m}}}\sum_{i\in F_{m}}\bar{x}_{n_{i}}.
\]
Since  $(y_{n_{i}})_{i}$ is dominated  by the unit vector basis  of $\ell_{p}$ it
follows that  $\lim_{m}\norm{z_{m}}=0$.
We show that  $\limsup_{m}\norm{u_{m}}\leq a$.
Let  $\e>0$  and choose $m_{0}\in\N$  such
that  $\e_{m_{0}}<\e/4$   and  $\sum_{k=m_{0}+1}^{\infty}\norm{d_{k}}_{\ell_{p}}<(\e/4)\cdot a$.

Let $F_{m}=\{k_{1}<k_{2}<\dots <k_{d}\}$ with  $\min F_{m}>m_{0}$.  
It follows that $m_{0}< k_{i}$ and hence
$1+\e_{k_{i}}<1+\sfrac{\e}{4}$  for every $i\leq d$.
Using that  $I_{k_{r}}=\cup_{j=1}^{r}I_{k_{j}}\setminus I_{k_{j-1}}
$,($I_{k_{0}}=\emptyset$),
and
that the final  segments $\brn_{i,>l_{k}}$, $i\in I_{k}$ are incomparable,
we  get that
\[
\bar{x}_{n_{k_{1}}}=\sum_{i\in I_{k_{1}}}x_{n_{k_{r}}|\brn_{i}}\,\,\text{and}\,\,
\bar{x}_{n_{k_{r}}}=\sum_{i\in I_{k_{1}}}x_{n_{k_{r}}|\brn_{i}}+
\sum_{j=2}^{r}\sum_{l\in I_{k_{j}}\setminus
    I_{k_{j-1}}}
  \sum_{i\in I_{l}\setminus I_{l-1}}x_{n_{k_{r}}|\brn_{i}},\, r\geq 2.
  \]
  The above analysis yields
  \begin{align}
    \label{eq:901}
    \sum_{r=1}^{d}\bar{x}_{n_{k_{r}}}&=
 \sum_{r=1}^{d}\sum_{i\in I_{k_{1}}}x_{n_{k_{r}}|\brn_{i}}+
 \sum_{r=2}^{d}\sum_{j=2}^{r}\sum_{l\in I_{k_{j}}\setminus I_{k_{j-1}}}
  \sum_{i\in I_{l}\setminus I_{l-1}}x_{n_{k_{r}}|\brn_{i}}
    \\
                                  &=
\sum_{i\in I_{k_{1}}}\sum_{r=1}^{d}x_{n_{k_{r}}|\brn_{i}}+
                                   \sum_{j=2}^{d}\sum_{l\in I_{k_{j}}\setminus
    I_{k_{j-1}}}\sum_{i\in I_{l}\setminus I_{l-1}}\sum_{r=j}^{d}x_{n_{k_{r}}|\brn_{i}}.
\notag
  \end{align}
Since  $(x_{n_{k}})_{k}$  satisfies properties  (i) and (ii) of
Lemma~\ref{stabilization},
for every $i\in I_{\ell}\setminus I_{\ell-1}$, $\ell \in I_{k_{j}}\setminus I_{k_{j-1}}$
  we get from  Remark~\ref{remnormjt}  (iv) that
  \[
\norm{\sum_{r=j}^{d}x_{n_{k_{r}}|\brn_{i}}}=
\left(\sum_{r=j}^{d}\norm{x_{n_{k_{r}}|\brn_{i}}}^{2}\right)^{1/2}
    \leq c_{i}(1+\e_{k_{j}})\sqrt{\#F_{m}}.
\]
Note that for every $\ell\in I_{k_{j}}\setminus I_{k_{j-1}}$ the final segments
$\brn_{i,>\ell_{k_{j}}}$, $i\in I_{\ell}\setminus I_{\ell-1}$ are incomparable and
$l_{k_{j}}<\minsupp x_{n_{k_r}}$ for every $r\geq j$. 
Hence 
using the above inequality, Lemma~\ref{suminbranches} yields
\begin{equation}
\label{eq:902}
\begin{aligned}
\norm{\sum_{i\in I_{l}\setminus I_{l-1}}\sum_{r=j}^{d}x_{n_{k_{r}}|\brn_{i}}}
  = &\left(\sum_{i\in I_{l}\setminus
      I_{l-1}}\norm{\sum_{r=j}^{d}x_{n_{k_{r}}|\brn_{i}}}^{p}\right)^{1/p}
  \\
    \leq & (1+\e_{k_{j}})\sqrt{\#F_m}  \left(\sum_{i\in I_{l}\setminus
      I_{l-1}}c^{p}_{i}\right)^{1/p}.
  \end{aligned}
\end{equation}
From \eqref{eq:901} and \eqref{eq:902} we get
\begin{equation*}
  \begin{aligned}
    \norm{\frac{1}{\sqrt{\# F_{m}}}
\sum_{r=1}^{d}\bar{x}_{n_{k_{r}}}}
&\leq
\frac{1}{\sqrt{\# F_{m}}}\left(
\norm{\sum_{i\in I_{k_{1}}}\sum_{r=1}^{d}x_{n_{k_{r}}|\brn_{i}}}+
  \sum_{j=2}^{d}\sum_{l\in I_{k_{j}}\setminus
    I_{k_{j-1}}}\norm{\sum_{i\in I_{l}\setminus
      I_{l-1}}\sum_{r=j}^{d}x_{n_{k_{r}}|\brn_{i}}}\right)
  \\
  &\leq
(1+\e_{k_{1}}) \left(\sum_{i\in I_{k_{1}}}c^{p}_{i}\right)^{1/p}+
  \sum_{j=2}^{d}\sum_{l\in I_{k_{j}}\setminus
    I_{k_{j-1}}}(1+\e_{k_{j}}) \left(\sum_{i\in I_{l}\setminus
      I_{l-1}}c^{p}_{i}\right)^{1/p}
  \\
  &\leq
(1+\e_{k_{1}}) \left(\sum_{i\in
    I_{k_{1}}}c^{p}_{i}\right)^{1/p}
+(1+\e_{k_{1}})\sum_{l>k_{1}}\left(\sum_{i\in I_{l}\setminus
    I_{l-1}}c^{p}_{i}\right)^{1/p}
\\
&\leq (1+\frac{\e}{4})a+ (1+\frac{\e}{4})\frac{\e}{4}a\leq (1+\e)a.
  \end{aligned}
\end{equation*}
From the above we get that   $\limsup_{m}\norm{u_{m}}\leq a$.

We have that   $ w_{m}=\frac{1}{\sqrt{F_{m}}}\sum_{k_{i}\in F_{m}}x_{n_{k_{i}}}=z_{m}+u_{m}$.
Using that  $\lim_{m}\norm{z_{m}}=0$ and $\limsup_{m}\norm{u_{m}}\leq a$
it follows
that  $\limsup_{m}\norm{w_{m}}\leq a$ and hence  $\lim_{m}\norm{w_{m}}=a$.

This complete  the proof 
that  $(x_{n_{i}})_{i\in\N}$ is  the desired subsequence.
\end{proof}
\begin{lemma}\label{concentration0}
  Let $(x_{n})_{n\in\N}$ be a normalized block sequence having
$(c_{i})_{i\in \N}\in S_{\ell_{p}}$ as  an $\ell_{p}$-vector  and
  $(\brn_{i})_{i\in \N}$ be the associated  branches.
 Then for every
 $\e>0$ there exist  an initial segment $I_{0}$ of $\N$  and
 $n_{0}\in\N$ such that setting $B_{I_{0}}=\{\brn_{i}:i\in I_{0}\}$ it holds
  $\norm{x_{n|\mc{B}_{I_{0}}^{c}}}<\e$ for every $n\geq n_{0}$.
\end{lemma}
\begin{proof}
Let $I_{0}\in\N$ be an initial segment of $\N$     such that
$\sum_{i\in I_{0}}c_{i}^{p}>1-\e/2$.
Let $\delta>0$ such that $1-\e\leq  (1-\delta)^{p}(1-\e/2)$. Chose $n_{0}\in\N$  such that
\begin{equation*}
c_{i}(1-\delta)\leq \norm{x_{n|\brn_{i}}}\leq c_{i}(1+\delta)\,\,\text{for every $i\in I_0$  and every $n\geq n_{0}$}
\end{equation*}
and the   final segments
$\brn_{i,\geq n_0}$,$i\in I_{0}$ be   incomparable. 
Then for every $n\geq  n_{0}$
\begin{align*} \norm{x_{n}}^{p}\geq&\norm{x_{n|\mc{B}_{I_{0}}^{c}}}^{p}+
\sum_{i\in I_{0}}\norm{x_{n|\brn_{i}}}^{p}
\\
& \geq \norm{x_{n|\mc{B}_{I_0}^{c}}}^{p}+\sum_{i\in I_{0}}c_{i}^{p}(1-\delta)^{p}
\\
 & \geq \norm{x_{n|\mc{B}_{I_0}^{c}}}^{p}+(1-\delta)^{p}(1-\e/2)
\\
&\Rightarrow \norm{x_{n|\mc{B}_{I_0}^{c}}}^{p}\leq1-(1-\delta)^{p}(1-\e/2)\leq\e.
\end{align*}
\end{proof}

\begin{proposition}\label{l2upperestimate}
  Let  $(x_{n})_{n}$ be a normalized block sequence in $\jtp$.
Then there exists a subsequence of  $(x_{n})_{n}$ that admits an upper
$\ell_{2}$ estimate with constant  $4$.
 \end{proposition}
\begin{proof} 
Let us observe that in the case  that there exists a subsequence
$(x_{n_{k}})_{k}$ such that $\lim\norm{x_{n_{k}|\brn}}=0$ for every
branch $\brn$ then it follows from Theorem~\ref{jtunique}  that  for every
$\e>0$ there exists
a subsequence that admits upper $\ell_{p}$-estimate with constant
$\sqrt[q]{2}+\e$ and hence admits upper $\ell_{2}$-estimate with constant
$4$. Hence, it suffices to prove the proposition
in the case  $(x_{n})_{n\in\N}$ has a
  non-zero $\ell_{p}$-vector $(c_{i})_{i\in \N}$.  Let  $(\brn_{i})_{i\in
    \N}$ be  the associated branches.  Let  $\e_{n}\searrow 0$ such that  $\e_{1}<1/3$.
Let also $(I_{k})_{k\in\N}$ be the  sequence  of  the strictly
increasing initial segments  of
$\N$ we get from  Lemma~\ref{lpsplit}
for $(c_{i})_{i\in \N}$ satisfying  also $\sum_{k>1}\norm{d_{k}}_{\ell_{p}}<\norm{(c_{i})_{i}}_{\ell_{p}}/2$.
From  Lemma~\ref{stabilization} for the sequence of the intervals
$I_{k}$, we get $n_{k}\nearrow+\infty$ and $(l_{k})_{k}\subset\N $ such that
\begin{enumerate}
\item[(i)] $ (1-\e_{k})c_{l}\leq \norm{x_{n_{k}|\brn_{l}}}\leq(1+\e_{k})c_{l}$   for every
  $l\in I_{k}$
\item[(ii)] the final  segments $\brn_{l,> l_{k}}$, $l\in I_{k}$
  are  incomparable and $\minsupp (x_{n_{k}})>l_{k}$, $k\in\N$.
\end{enumerate}
Let   $\mc{B}_{I_{k}}=\cup\{\brn_{l}:l\in I_{k}\}$,
$\bar{x}_{n_{k}}=x_{n_{k}|\mc{B}_{I_{k}}}$
and  $y_{n_{k}}=x_{n_{k}}-\bar{x}_{n_{k}}$.
From Lemma~\ref{stabilization}  we have that
$(y_{n_{k}})_{k\in N}$ admits upper $\ell_{2}$-estimate with constant
$2$.

We show that the subsequence  $(x_{n_{k}})_{k}$  admits  upper
$\ell_{2}$-estimate with constant $4$.
Since $x_{n_{k}}=\bar{x}_{n_{k}}+y_{n_{k}}$
it is enough to  show  that  $(\bar{x}_{n_{k}})_{k}$ also admits upper $\ell_{2}$-estimate with constant $2$.
Let $b_{1},\dots, b_{m}$ real numbers. We show that
\begin{equation*}
\norm{\sum_{k=1}^{m}b_{k}\bar{x}_{n_{k}}}\leq4\left(\sum_{k=1}^{m}b_{k}^{2}\right)^{1/2}\cdot\norm{(c_{i})_{i}}_{\ell_{p}}.
\end{equation*}
Let  $k\leq m$ and $l\in I_{k}$.  Remark~\ref{remnormjt}  (iv) yields that,
\[\norm{\sum_{i=k}^{m}b_{i}\bar{x}_{n_{i}|\brn_{l}}}=
      \left(\sum_{i=k}^{m}b_{i}^{2}
        \norm{\bar{x}_{n_{i}|\brn_{l}}}^{2}\right)^{1/2}.
    \]
Combining the above equality  with    (i)  we get that for
every $l\leq i_{j}$,
\begin{equation}
  \label{eq:25}
  \left(\sum_{i=k}^{m}b^{2}_{i}\right)^{1/2}c_{l}(1-\e_{k})\leq
  \norm{\sum_{i=k}^{m}b_{i}\bar{x}_{n_{i}|\brn_{l}}}     \leq
      c_{l}(1+\e_{k})\left(\sum_{i=k}^{m}b_{i}^{2}\right)^{1/2}.
\end{equation}

By the choice of $\bar{x}_{n_{k}}$  we have that
$z=\sum_{k=1}^{m}b_{k}\bar{x}_{n_{k}}$ is supported  by the branches
$\brn_{i}, i\in I_{m}$. Moreover  (ii) yields that
the final  segments $\brn_{l,\geq \minsupp (\bar{x}_{n_{j}})}$,
$l\in I_{j}$ are  incomparable.
Hence using that  $I_{k}=\cup_{j=1}^{k}I_{j}\setminus I_{j-1}$,$I_{0}=\emptyset$, we get
\[
\bar{x}_{n_{1}}=\sum_{i\in I_{1}}x_{n_{1}|\brn_{i}}\,\,\text{and}\,\,
\bar{x}_{n_{k}}=
  \sum_{l\in I_{1}}x_{n_{k}|\brn_{l}}+
  \sum_{j=2}^{k}\sum_{l\in I_{j}\setminus I_{j-1}}x_{n_{k}|\brn_{l}},\,  k\geq 2.
\]
It follows that
\begin{equation}\label{eqoo}
  \sum_{k=1}^{m}b_{k}\bar{x}_{n_{k}}=
  \sum_{l\in I_{1}}\sum_{k=1}^{m}b_{k}x_{n_{k}|\brn_{l}}+
  \sum_{j=2}^{m}  \sum_{l\in I_{j}\setminus
    I_{j-1}}\sum_{k=j}^{m}b_{k}x_{n_{k}|\brn_{l}}.
  \end{equation}
For any $j$ and any  vector  $\sum_{l\in I_{j}\setminus  I_{j-1}}\sum_{k=j}^{m}b_{k}x_{n_{k}|\brn_{l}}$ from (ii) and
Lemma~\ref{suminbranches} using    \eqref{eq:25}
we get ,
\begin{align}
  \norm{
\sum_{l\in I_{j}\setminus I_{j-1}}\sum_{k=j}^{m}b_{k}x_{n_{k}|\brn_{l}}
  }
  &\leq \left(
    \sum_{l\in  I_{j}\setminus I_{j-1} }
    \norm{\sum_{k=j}^{m}b_{k}x_{n_{k}|\brn_{l}}}^{p}
  \right)^{1/p}   \notag \\
  &\leq \left(
    \sum_{l\in I_{j}\setminus I_{j-1}}
    (\sum_{k=j}^{m}b_{k}^{2})^{p/2}c_{l}^{p}(1+\e_{j})^{p} \right)^{1/p}.
  \label{eqo1}
\end{align}
 Combining \eqref{eqoo} and \eqref{eqo1} we get
\begin{align*}
  \norm{\sum_{k=1}^{m}b_{k}\bar{x}_{n_{k}}}&
                                          \leq
                                          \norm{\sum_{l\in I_{1}}\sum_{k=1}^{m}b_{k}\bar{x}_{n_{k}|\brn_{l}}}+                                          
                                          \sum_{j=2}^{m}  \norm{\sum_{i\in
                                          I_{j}\setminus I_{j-1}}\sum_{k=j}^{m}b_{k}\bar{x}_{n_{k}|\brn_{l}}}
  \\  &\leq(1+\e_{1})\left(\sum_{k=1}^{m}b_{k}^{2}\right)^{1/2}\left(\sum_{l\in I_{1}}c_{l}^{p}\right)^{1/p}+
        \sum_{j=2}^{m}\left(
    \sum_{l=I_{j}\setminus I_{j-1}}
    (\sum_{k=j}^{m}b_{k}^{2})^{p/2}c_{l}^{p}(1+\e_{j})^{p} \right)^{1/p} 
  \\
       &\leq  (\sum_{k=1}^{m}b_{k}^{2})^{1/2} (1+\e_{1})(1+\frac{1}{2}) \norm{(c_{i})_{i\in\N}}_{\ell_{p}}
  \\
&\leq 2(\sum_{k=1}^{m}b_{k}^{2})^{1/2} \norm{(c_{i})_{i\in\N}}_{\ell_{p}}.
  \end{align*}
\end{proof}

\begin{corollary}
  Let $(x_{n})_{n\in\N}$ be a normalized weakly null sequence in $\jtp$. Then $\sxn$
  has a subsequence  which is   either equivalent  to the unit vector basis of $\ell_{p}$  or equivalent to the unit  vector basis of $\ell_{2}$.
\end{corollary}
Indeed if $(x_{n})_{n}$ does not contain a subsequence equivalent to
the unit vector basis of $\ell_{p}$  then  from Theorem~\ref{jtunique}
we get that there exists a branch $\brn$ such that $\limsup_{n}\norm{x_{n|\brn}}>0$.
Therefore there exists  a subsequence that dominates
the unit vector basis of $\ell_{2}$ and
from  Proposition~\ref{l2upperestimate}  we may assume also that
 admits an upper $\ell_{2}$-estimate.   
\begin{corollary}
 The space $\jtp$ is reflexive. 
\end{corollary}
\section{Hilbertian subspaces of $\jtp$}
In this section we provide a characterization of the subspaces  of
$\jtp$ which are  isomorphic to $\ell_{2}$ and we show that  each such
subspace is complemented in $\jtp$.

We start with the following lemma.
\begin{lemma}\label{incomparable}
Let $(\brn^{n})_{n\in\N}$ be sequence of distinct  branches. There exist
$M\in[\N]$ and $(l_{n})_{n\in M}\subset \N$ such that the final segments
$\brn^{n}_{>l_{n}}, n\in M$, are  incomparable.
\end{lemma}
\begin{proof}
  Passing to subsequence we get $M\in[\N]$  such that $(\brn^{n})_{n\in
    M}$  converges to  $\brnt$
where $\brnt$ is either a branch or an initial finite segment. 

  Assume first that $\brnt$ is an  branch.   Inductively, using that  $(\brn^{n})_{n\in
    M}$  converges to  $\brnt$, we choose
  $d_{1}<l_{1}<d_{2}<l_{2}<\dots$  such that $d_{k},l_{k}\in\brnt$
  and  $(n_{k})_{k\in \N}\subset M, n_{k}\nearrow+\infty$  satisfying
  \begin{enumerate}
  \item[(i)]    $\brn^{n}_{\leq d_{k}}=\brnt_{\leq d_{k}}$ for every $n\geq
    n_{k}$
    \item[(ii)]  the final segments $\brn^{n_{k}}_{>l_{k}},
      \brnt_{>l_{k}}$ are  incomparable.
  \end{enumerate}
It is easy to check that the final segments $\brn^{n_{k}}_{>l_{k}}$,
$k\in\N$ are  incomparable.

Assume now that $\brnt$ is a finite initial segment and
$t_{0}=\max\brnt$. Let   $\{t_{l}:l\in\N\}$ be the immediate successors
of $t_{0}$.  Since $(\brn^{n})_{n\in
    M}$  converges to  $\brnt$ it follows that  for every $l$ there 
exist  at most finite $n$ such that  $t_{l}\in\brn^{n}$. Using this we
choose inductively   $(n_{k})_{k}\subset M$ such  that
$t_{l}\notin\brn^{n_{k}}$ for every $l<k$, $k\in\N$.  It follows that the
final segments
$\brn^{n_{k}}_{>t_{0}}, k\in\N$ are   incomparable.
\end{proof}
\begin{lemma}\label{incomparable2}
  Let $F_{n}=\{\brn^{n}_{1},\brn^{n}_{2},\dots,\brn^{n}_{k_{n}}\}$ be
  a  sequence of disjoint finite families of branches each family
  consisting of distinct branches and $k_{n}\nearrow+\infty$. Then, there exist
  $M\in[\N]$  and $(l_{m})_{m\in M}\subset\N$
  such that the following holds:

  for every   $m<m_{1}<m_{2}\in M$     and  $k\leq  k_{m}$,
the  final segments $\brn_{k, >l_{m_{1}}}^{m_{1}},\brn_{k,
    >l_{m_{2}}}^{m_{2}}$
  are incomparable.
  \end{lemma}
  \begin{proof}
    Let $m_{0}=1$, $M_{0}=\N\setminus\{1\}$. Inductively, using  Lemma~\ref{incomparable}  we choose
    $M_{0}\supset M_{1}\supset\dots\supset M_{j}\supset \dots$    such setting $m_{j}=\min M_{j}$,   $m_{j-1}<m_{j}$ for every $j\geq1$, the following holds:

    for  every $k\in\{k_{m_{j-2}}+1,\dots, k_{m_{j-1}}\}$,
    $k_{m_{-1}}=0$, there exists $(l_{k}^{m})_{m\in M_{j}}\subset \N$ such that  
\begin{equation}
    \text{the   final segments}\,\,\,\brn^{n}_{k,>l_{k}^{n}},\,\, n\in
    M_{j} \,\text{are incomparable}.
   \label{eq:926}
 \end{equation}
 For every $j$ choose  $l_{m_{j}}>\max\{l_{m_{j-1}},l_{k}^{m_{j}}: k\in \{k_{m_{j-2}}+1,\dots,k_{m_{j-1}}\}\}$.

We show that   $(m_{j})_{j\in\N}$ and $(l_{m_{j}})_{j\in\N}$ satisfy the
conclusion. Let $m_{j}<m_{j_{1}}<m_{j_{2}}$ and  $k\leq  k_{m_{j}}$, $k\in
\{k_{m_{i-1}}+1,\dots, k_{m_{i}}\}$ for some $i\leq j$.
By \eqref{eq:926}
we have that  the final segments  $\brn_{k,>l_{k}^{n}}^{n}, n\in M_{i+1}$
are  incomparable.
Since $i\leq j<j_{1}<j_{2}$ it follows that $M_{i+1}\supset M_{j_{1}}\supset M_{j_{2}}$
and  for $p=1,2$,
$l_{m_{j_{p}}}\geq l_{m_{i+1}}\geq
l_{k}^{m_{i+1}}$  for every  $k\in \{k_{m_{i-1}}+1,\dots,k_{m_{i}}\}$.
This implies that  the final segments  $\brn_{k,>l_{m_{j_{1}}}}^{m_{j_{1}}},
\brn_{k,>l_{m_{j_{2}}}}^{m_{j_{2}}}$ are incomparable. 
  \end{proof}
  \begin{remark}
    In the general case, for an arbitrary  sequence $(F_{n})_{n\in\N}$ of
    finite families of branches there exists  a subsequence $(F_{n_{k}})_{k}$
such    that  either $\# F_{n_{k}}\nearrow +\infty$    or $\# F_{n_k}=\rho$ for every
$k$. In the later case the previous lemma is still valid and the proof
is simpler.
  \end{remark}
\begin{proposition}\label{lemmalp}
  Let $Y$ be a subspace of $\jtp$ with the property that for every finite
  set of branches $\tau_{1},\dots,\tau_{d}$  there exists a normalized
  weakly null  sequence $(y_{n})_{n}$ of $Y$ with associated
  $\ell_{p}$-vector $(c_{i})_{i\in \N}\in S_{\ell_{p}}$  such that
  $\lim_{n}\norm{y_{n|\cup_{i=1}^{d}\tau_{i}}}=0$.
Then $\ell_{p}$  embeds in $Y$.
\end{proposition}
\begin{proof}
Let $(\e_{k})_{k}$ be a sequence of positive numbers such that
$\sum_{k\in\N}\e_{k}<1$. Using the assumption of the lemma,
we get a  normalized weakly null sequence $(y_{n}^{1})_{n\in\N} $ in  $Y$ 
with $\ell_{p}$-vector $(c_{i}^{1})_{i\in \N}\in S_{\ell_{p}}$  and
associated set of branches  $(\brn_{i}^{1})_{n\in \N}$.
From   Lemma~\ref{block}  we get 
a block sequence  $(x_{n}^{1})_{n\in\N}$ having  the same $\ell_{p}$-vector
and associated set of branches with 
$(y_{n}^{1})_{n\in\N}$.
From Lemma~\ref{concentration0}  we get  
 an initial interval $I_{1}$  of $\N$  and $n(1)\in\N$ such that setting $B_{I_{1}}=\{\brn_{i}^{1}:i\in I_{1}\}$  it holds,
\[\norm{(c_{i}^{1})_{i\notin I_{1}}}_{\ell_{p}}<\e_{1}\,\,\text{and}\,\,
\norm{x_{n}^{1}-x_{n|\mc{B}_{I_{1}}}^{1}}
    <\e_{n}\,\,\text{ for every}\, n\geq n(1).\]
Moreover  from the definition of the  $\ell_{p}$-vector we may assume
  \begin{equation*}
   c_{i}^{1}(1-\e_{1})\leq \norm{x_{n|\brn_{i}^{1}}^{1}}\leq
   (1+\e_{1})c_{i}^{1} \,\,\,\,\text{for every}\, n\geq n(1)\,
\text{and   every
  $\brn_{i}^{1}$, $i\in I_{1}$}.
\end{equation*}
Set $T_{1}=\mc{B}_{I_{1}}$.
Continuing,  inductively  we choose an
array  of  normalized weakly null sequences $(y_{n}^{k})_{n\in\N}, k\in\N$ of $Y$
with $\ell_{p}$-vector $(c_{i}^{k})_{i\in \N}\in S_{\ell_{p}}$  and
associated set of branches  $(\brn_{i}^{k})_{n\in \N}$ and 
an array of block sequences  $(x_{n}^{k})_{n\in\N}$, $k\in\N$,  such that
 \begin{equation}
   \label{eq:4}
  \norm{y_{n}^{k}-x_{n}^{k}}<\e_{n}     \,\,\,\text{for every $n\in\N$
    and every $k$},
\end{equation}
satisfying the following:

For every $k$, there exist  $n(k)\in\N$ and an 
initial segment
$I_{k}(\supsetneq I_{k-1})$ of $\N$  such that  setting
$\mc{B}_{I_{k}}=\{\brn_{i}^{k}: i\in I_{k}\}$,
the following holds
\begin{enumerate}
\item[(i)]      $\lim_{n} \norm{y^{k}_{n|Z_k}}=0$ where
  $Z_{k}=\cup_{i\leq k-1}\mc{B}_{I_{k}}$
  \,\,  $(Z_{1}=\emptyset)$ . 
\item[(ii)]  $\norm{(c_{i}^{k})_{i\notin I_{k}}}_{\ell_{p}}<\e_{k}$.
\item[(iii)] 
For every $k$  and  every
  $\brn_{i}^{k}\in\mc{B}_{I_{k}}$ it holds
  \begin{equation}
    \label{eq:8a}
   c_{i}^{k}(1-\e_{k})\leq \norm{x_{n|\brn_{i}^{k}}^{k}}\leq
 (1+\e_{k})c_{i}^{k} \,\,\,\,\text{for every}\,\, n\geq n(k).
\end{equation}
\item[(iv)]  $\norm{x_{n}^{k}-x_{n|\mc{B}_{I_{k}}}^{k}}
    <\e_{n}$ for every   $n\geq n(k)$.
\end{enumerate}
Let us observe that  from  (i),(iv) and \eqref{eq:4} it follows that
the sets  of branches  $\mc{B}_{I_k}$, $k\in\N$,
are pairwise disjoint. Applying Lemma~\ref{incomparable2} we get $K\in[\N]$  and
$(l_{k})_{k}$ such that
\begin{align}
  \label{eq:12a}
 \forall   k<k_{1}<k_{2}\in K \,\,\text{and every}\,\, i\in I_{k}\\
 \text{the final segments}\,\,\ \brn_{i,>l_{k_{1}}}^{k_{1}},
 \brn_{i,>l_{k_{2}}}^{k_{2}} \text{are incomparable}. \notag
\end{align}
For every $k_{j}\in K$ choose   $x_{n_{k_{j}}}^{k_{j}}\in
(x_{n}^{k_{j}})_{n\in \N}$, $n_{k_{j}}>\max\{n_{k_{j-1}}, n(k_{j})\}$, such
that
\begin{enumerate}
\item[(v)]  $p_{j}=\min\supp x_{n_{k_{j}}}^{k_{j}}>l_{k_{j}}$ and 
\item[(vi)]   the final segments $\brn_{i,\geq p_{j}}^{k_{j}}$,\,\,
$i\in I_{k_{j}}$
are incomparable.
\end{enumerate}
Let us observe, using that the  $\ell_{p}$-vectors are decreasing sequences,
that
\begin{equation}
  \label{eq:10}
\forall \e>0\,\text{
 there exists  $i(\e)$ such that  $c^{k}_{i}<\e$ for every $i\geq
i(\e)$ and every $k$}.  
\end{equation}
Set  $z_{j}=x^{k_j}_{n_{k_j}|\mc{B}_{I_{k_j}}}$, $k_j\in K$.
\\

\textit{Claim:}
For every branch $\brn$ it holds that $\lim_{j}\norm{z_{j|\brn}}=0$.

\noindent \textit{Proof of Claim}.
On the contrary assume that there exist $\e>0$ and $J\in [\N]$ 
such that    $\norm{z_{j|\brn}}>\e$  for  every $j\in
J$.
Choose   $j_{0}\in J$   such that  $i(\e/2)<\max I_{k_{j_{0}}}$, where
$i(\e/2)$ is given by \eqref{eq:10}.

From (iv) we get that there exists a unique $i(k_{j})\in I_{k_{j}}$ such
that $\brn\cap \brn^{k_{j}}_{i(k_{j}),\geq p_{j}}\ne\emptyset$. Since $\supp(z_{j})\subset \cup\{\brn_{i,\geq p_{j}}: i\in I_{k_{j}}\}$ we get
\begin{equation}
  \label{eq:16} \e<\norm{z_{j|\brn}}=\norm{x^{k_{j}}_{n_{k_{j}}|}{}_{\brn\cap\brn_{i(k_{j})}^{k_{j}}}}\,\,\text{for a
unique}\,\, i(k_{j})\in I_{k_{j}}.
\end{equation}

Since $1+\e_{k}<2$ for every $k$,   (iii)
yields that 
\[
\e<\norm{x^{k_{j}}_{n_{k_{j}}|}{}_{\brn\cap\brn_{i(k_{j})}^{k_{j}}}}\leq 
(1+\e_{k_{j}})c_{i(k_{j})}\leq 2c_{i_{k(j)}}\Rightarrow i(k_{j})\leq i(\e/2)
\,\,\,\text{by \eqref{eq:10}}.
\]
Passing to a further
subsequence  we may assume that  $i(k_{j})=l\leq i(\e/2)$ for every $j$.
For  $j_{0}<j_{1}<j_{2}$  \eqref{eq:16} yields
$\brn\cap\brn_{l,\geq p_{j_{i}}}^{k_{j_{i}}}\ne
\emptyset$ and this leads to a contradiction since
$ p_{j_{i}}>l_{k_{j_{i}}}$ and
by  \eqref{eq:12a}
the final segments  $\brn_{l,>l_{k_{j_{i}}}}^{k_{j_{i}}}$, $i=1,2$  are 
incomparable $\qedhere$.

By Theorem~\eqref{jtunique} we get that $(z_{j})_{j}$ has a
subsequence  $(z_{j})_{j\in J_{0}}$  equivalent to the unit vector basis of 
$\ell_{p}$.
We show now that the sequence $(y_{n_{k_j}}^{k_{j}})_{j\in J_{0}}\subset Y$ has a
subsequence equivalent to  the unit vector basis of $\ell_{p}$.  By 
\eqref{eq:4}  and (iv) we get that
\[
  \norm{y_{n_{k_{j}}}^{m}-x_{n_{k_{j}}}^{k_{j}}}<\e_{n_{k_{j}}}\,\,\text{and}\,\,\,
  \norm{x_{n_{k_{j}}}^{k_{j}}-z_{j}}<\e_{n_{k_{j}}}\,\,\,\text{for every}\, j\in J_{0}.
\]
Passing to a further subsequence if it is necessary and using that
$(n_{k_{j}})_{j\in J_{0}}$ is strictly increasing sequence, we get that
$(y^{k_{j}}_{n_{k_{j}}})_{j\in J_{0}}$  is equivalent to $(z_{j})_{j\in J_{0}}$  and hence to the unit 
vector basis of $\ell_{p}$.
\end{proof}
A more careful inspection  of the previous proof shows that  $\ell_{p}$
embeds  in $Y$ once we have  an array of normalized block sequences
with certain properties whose the $\ell_{p}$-vector is not necessarily
normalized.
\begin{proposition}
  Let $(x_{n}^{k})_{n\in\N}, k\in\N$ be an array of seminormalized block
  sequences having   non-zero $\ell_{p}$-vectors $(c_{i}^{k})_{i\in \N}$
with  associated set of branches $(\brn_{i}^{k})_{i\in\N}$,
and $\e_{n}\searrow 0$  such that
there exists a strictly increasing sequence   $(I_{k})_{k\in\N}$ of
initial intervals of $\N$
satisfying the following
  \begin{enumerate}
  \item The sets $\mc{B}_{I_{k}}=\{\brn_{i}^{k}:i\in I_{k}\}$, $k\in\N$, 
 are  disjoint.
    \item For every $k,n$,  it holds that
      \[
        \norm{x_{n|\mc{B}_{I_{k}}}^{k}}>\frac{1}{2}\,\,\,\text{and}\,\,\,
        \norm{x_{n}^{k}-x_{n|\mc{B}_{I_{k}}}^{k}}<\e_{n}.\]
     \end{enumerate}
  Then, there exists a sequence
  $(x_{n_{j}}^{k_{j}})_{j}\subset\{x_{n}^{k}:n,k\in\N\}$    which is
  equivalent to the unit vector basis of $\ell_{p}$.
\end{proposition}

\begin{proof}[Sketch of the proof] The proof follows the arguments of
  the proof of Proposition~\ref{lemmalp}
  For every $k$ passing to a subsequence we may assume that
  \eqref{eq:8a} holds. Also by the assumptions  we have that  (iv) of 
   the proof of Proposition~\ref{lemmalp} holds.   By (1)
 the sets
   $\mc{B}_{I_{k}}, k\in\N$ are disjoint and thus  we may apply
   Lemma~\ref{incomparable2} to get  $K\in[\N]$ and $(l_{k})_{k\in K}$
   such that \eqref{eq:12a} holds.  Continuing as in the proof  of
   Proposition~\ref{lemmalp}
   we get the subsequence $(x^{k_{j}}_{n_{j}})_{j}$   which is equivalent to the 
   unit vector basis of $\ell_{p}$.
\end{proof}
\begin{proposition}\label{lpvectornorm1}
  Let $\brnt_{1},\dots\brnt_{d}$ be a finite set of branches and $Y$ be a subspace of $\jtp$ which contains a normalized weakly null sequence $(y_{n})_{n}$ such that
  $\lim_{n} \norm{y_{n|\cup\brnt_{i}}}=0$ and its 
  $\ell_{p}$-vector $(c_{i})_{i\in \N}$ has norm $\norm{(c_{i})_{i}}_{\ell_{p}}=a\in (0,1)$.
  Then, there exists a normalized weakly null sequence $(w_{n})_{n}$
  in $Y$ having an $\ell_{p}$-vector $(d_{i})_{i\in \N}\in S_{\ell_{p}}$ and the same
  associated set of branches as $(y_{n})_{n}$
  and moreover
    $\lim_{n} \norm{w_{n|\cup\brnt_{i}}}=0$.
  \end{proposition}
\begin{proof}
From Lemma~\ref{block}  we may assume that $(y_{n})_{n}$ is
equivalent to a block sequence $(x_{n})_{n\in\N}$. From Lemma~\ref{lemmap1} we get 
a subsequence  $(x_{n_{i}})_{i\in\N}$ such that if
$(z_{m})_{m}$ is a sequence of successive $\ell_{2}$-averages of
$(x_{n_{i}})_{i\in\N}$ of increasing length, $z_{m}=\frac{1}{\sqrt{\# F_{m}}}\sum_{i\in
F_{m}}x_{n_{i}}$,  then  $(z_{m})_{m}$  has the same  $\ell_{p}$-vector
and associated set of  brances  with $(y_{n})_{n}$ 
and  moreover $\norm{z_{m}}\to a$.

We show that $\lim_{m}\norm{z_{m|\cup_{i}\brnt_{i}}}=0$.
Let $\e>0$.  Since  $\lim_{n}\norm{x_{n|\cup\brnt_{i}}}$ there exists  $n_{0}\in\N$  such that
\begin{equation*}
  \norm{x_{n|\cup_{1}^{d}\brnt_{i} }}<\frac{\e}{d}\,\,\,\text{for every}\,\, n\geq n_{0}.
\end{equation*}
For every branch $\tau_{j}$  we have that
\[
  \norm{\frac{1}{\sqrt{\# F_{m}}}\sum_{i\in F_{m}}x_{n_{i}|\brnt_{j}}}\leq
\frac{1}{\sqrt{\# F_{m}}}
 \left(\sum_{i\in F_{m}}\norm{x_{n_{i}|\brnt_{j}}}^{2}\right)^{1/2}\leq \frac{\e}{d}\,\,\,\text{for every}\,\, m>n_{0}.
\]
From the triangle inequality we get $\lim_{m}\norm{z_{m|\cup_{j=1}^{d}\brnt_{j}}}=0$. Set $w_{m}=a^{-1}z_{m}$, $m\in\N$. Using that $(x_{n})_{n}$ is equivalent to $(y_{n})$, by standard perturbation arguments, passing if necessary to a further  subsequence, we get the desired weakly null sequence in $Y$.
\end{proof}
We state now a lemma, probably a standard one, whose proof  is presented for the sake of completeness,  which  will be used in the proof  of the   main result of this section. 
\begin{lemma}\label{standardlemma}
  Let $X$ be a Banach space,   $Z\subseteq Y$  be subspaces of $X$ such that
  $Z$ is a finite codimensional subspace  of  $Y$  and also  is complemented in $X$.  
 Then  $Y$ is also complemented in $X$.
\end{lemma}
\begin{proof}
  Let  $P:X\to Z$ be a bounded projection and
$Y=Z\oplus F$ where 
  $F=\spann\{y_{1},\dots,  y_{d}\}$,
  $y_{i}\in  S_{Y}$. Choose $y_{i}^{*}\in Y^{*}$ such that $y_{i}^{*}(y_{i})=1$ and $y^{*}_{i}(z+\lambda y_{j})=0$ for every $z\in Z$, $\lambda\in\R$ and $j\ne i$. Denote by $\hat{y}_{i}^{*}$ a norm-preserving extension of $y_{i}^{*}$ to $X$.
Note that $W=\cap_{i=1}^{d}\kerr(\hat{y}_{i}^{*})\supset Z$ so $P_{|W}:W\to Z$ is a projection.  Defining $P_{Y}:X=W\oplus F\to Y=Z\oplus F$ by $P_{Y}(w+f)=P_{|W}(w)+f$ we have that $P_{Y}$ is a projection on $Y$.
\end{proof}
We proceed now to the main result of this section.

\begin{theorem}\label{complemented2}
  Let $Y$ be a  subspace of $\jtp$  not containing  $\ell_{p}$.
  Then, $Y$ is   a  complemented subspace of $\jtp$ isomorphic to $\ell_{2}$.
\end{theorem}
To prove the theorem we shall need the following proposition.
\begin{proposition}
  Let $Y$ be a  subspace of $\jtp$ not containing $\ell_{p}$.  There exist a finite set  of 
  branches $\brnt_{1},\dots,\brnt_{d}$
  and a finite codimensional subspace $Y_{0}$ of $Y$
 such that  the bounded  projection  $P(y)=y_{|\text{T}}$,
$\text{T}=\cup\{\brnt_{i}: i\leq d\}$
restricted to $Y_{0}$  is an  isomorphism.
  \end{proposition}
\begin{proof}
Assume that the conclusion does not hold.
Then
for every  finite set of branches $T=\{\brnt_{i}: i\leq d\}$ and
a finite codimensional subspace $Y_{0}$ of $Y$,
the restriction of $P$ to $Y_{0}$   is not an isomorphism.
For every $k\in\N$ set $\displaystyle H_{k}=\cap_{i=1}^{k}\ker (e_{i}^{*})$
and $Y_{k}=Y\cap H_{k}$. Then  inductively   we choose  normalized vectors  $x_{k}\in Y_{k}$  such that $\norm{x_{k|\text{T}}}<k^{-1}$.  It follows that   the sequence $(x_{k})_{k}$ is a normalized weakly null sequence  and $\lim_{k}\norm{x_{k|\text{T}}}=0$. Moreover 
  the $\ell_{p}$-associated vector $(c_{i})_{i\in \N}$ of $(x_{n})_{n}$  has  norm
  $
  \norm{(c_{i})_{i\in \N}}_{\ell_{p}}=a>0$. 
From  the above and  Proposition~\ref{lpvectornorm1}  we get that the assumptions
of Proposition~\ref{lemmalp} holds.  It follows that  $\ell_{p}$ embeds
in $Y$ a contradiction.
\end{proof}

\begin{proof}[Proof of Theorem~\ref{complemented2}]
  From the previous proposition
we get
 a finite codimensional subspace  $Y_{0}$ of $Y$
and  a finite set of branches $T=\{\brnt_{1},\dots,\brnt_{d}\}$ such that  the bounded  projection  $P_{|Y_{0}}(y)=y_{|\text{T}}$  is an isomorphism.
Let $Z=[e_{n}:n\in \text{T}]$ and $n_{0}\in\N$ be such that  the final  segments  $\brnt_{i,>n_{0}}$, $i\leq d$, are incomparable.  We have that
$Z=(Z\cap [e_{n}:n\leq n_{0}] )\oplus (Z\cap [e_{n}:n>n_{0}])$.
The finite dimensional subspace $Z\cap [e_{n}:n\leq n_{0}] $ is isomorphic to $\ell_{2}^{k}$ and also from  Lemma~\ref{suminbranches} we get that
$Z\cap [e_{n}:n>n_{0}]$ isomorphic to $\sum_{i=1}^{d}\oplus\ell_{2}(\simeq\ell_{2})$.
It follows that $Z$ is isomorphic to $\ell_{2}$.  Since $Y_{0}$ embeds in
$Z$ and $Z$  is Hilbert space, it holds that  $Y_{0}$ is a   complemented subspace  of $\jtp$ isomorphic to $\ell_{2}$.

Since $Y=Y_{0}\oplus G$, $G$ finite dimensional it follows that $Y$ is also  a Hilbert space.
From Lemma~\ref{standardlemma}  we get   that $Y$ is complemented in $\jtp$ and this completes the proof.
\end{proof}
\begin{proposition}
If   $\jtp=Z\oplus\ell_{2}$ then  $Z$ is isomorphic to $\jtp$.
\end{proposition}
\begin{proof}
Let $P:\jtp\to\ell_{2}$ be a bounded projection.
It is enough to prove that  $Z$ contains a complemented subspace isomorphic to  $\ell_{2}$. Indeed if $Z_{0}$ is such a subspace and $Z=W\oplus Z_{0}$
then,
  \[
    Z\backsimeq W\oplus\ell_{2}\backsimeq (W\oplus\ell_{2})\oplus\ell_{2}\backsimeq Z\oplus\ell_{2}=\jtp. \]

  To prove that     $Z$ contains a complemented subspace isomorphic to  $\ell_{2}$ we start with the following claim.

  \textit{Claim:} For every $\e>0$ there exists a finite number of
  brances $\brn_{1},\dots,\brn_{k}$ such that  $\limsup_{k\in
    \brn_{i}}\norm{P(e_{k})}>\e$ for every $i\leq k$.

  Assume that we have prove the claim. Then it follows that there
  exists   an at most  countable  set of branches  $\brn_{i}, i\in\N$  such that for every $\brn\ne\brn_{i}$ it holds  $\lim_{k\in\brn}\norm{P(e_{k})}=0$.

  Let $\brn\ne\brn_{i}$ for every $i$. Then we have that $[e_{k}:k\in\brn ]$ is  isometric to $\ell_{2}$
and for every $k$,  $e_{k}=P(e_{k})+(I-P)(e_{k})= P(e_{k})+z_{k}$, $z_{k}\in Z$.  Since  $\lim_{k\in\brn}\norm{P(e_{k})}=0$ passing to a subsequence we may assume that $\sum_{k}\norm{e_{k}-z_{k}}<1$ and hence
  $(z_{k})_{k}\backsimeq(e_{k})_{k}$. It follows that  $[z_{k}: k\in\N ]$ is a complemented subspace of $Z$ isomorphic to $\ell_{2}$.
  \\

  \textit{Proof of the Claim}. On the contrary assume that  for some
  $\e>0$ there exist infinitely many branches $\brn_{n}, n\in \N$  such that   $\limsup_{k\in\brn_{n}}\norm{P(e_{k})}>\e$.

  Passing to a subsequence  we may assume that   $\brn_{n}\to\tau$, where
  $\tau$ is  a branch or an initial segment. Assume  that $\tau$ is a branch. The
  case $\brnt$ is a finite initial segment is similar.
  We show the existence of a sequence of incomparable nodes $(t_n)_{n}$ such that
  $\norm{P(e_{t_{n}})}>\e$.  We follow the arguments in the proof of Lemma~\ref{incomparable}.
We   choose inductively  initial segments $I_{1}\subsetneqq I_{2}\subsetneqq\dots$  of $\brnt$,
$n_{j}\nearrow +\infty$  and $l_{j}\nearrow+\infty$ such that
 \begin{enumerate}[resume]
\item $    \brn_{n|I_{j}}=\tau_{|I_{j}}$ for every  every $n\geq n_{j}$,
  \item  the final  segments    $ \brn_{n,>l_{j}}, n=1,\dots,n_j$ are incomparable,
     $l_{j}<\max I_{j+1}$,
\item   $t_{j}\in \brn_{n_{j}}$, $t_{j}>l_{l}$ and $\norm{P(e_{t_{j}})}\geq \e$
\end{enumerate}
Is easy to check that  the nodes $t_{j},j\in\N$ are  incomparable.

Since $(t_{j})_{j}$ are incomparable it follows that  the subspace $X=[e_{t_{j}} :j\in\N]$ is isometric to $\ell_{p}$.
The same holds for the subspace $X_{L}=[e_{t_{j}} :j\in L]$ for every  $L\in [\N]$. Since  $(e_{t_{j}})_{j}$ is w-null, passing to a
subsequence
we may assume also that $(P(e_{t_{j}}))_{j}$ is  a seminormalized block sequence.

So $P_{|X}:\ell_{p}\to \ell_{2}$ and by Pitt's theorem $P_{|X}$ is
compact. This leads to a contradiction since $(P(e_{t_{j}}))_{j}$
does not have convergence subsequence.
\end{proof}
\begin{remark}
 N. Ghoussoub and B. Maurey~\cite{GhM}, using an infinitely  branching   tree $\dt_{\infty}$, defined  the   James tree space $JT^{\infty}$. This space is defined   as the  James tree space, using    the tree $\dt_{\infty}$ instead  of the dyadic tree.   They showed that  the preduals  of these  spaces  are not isomorphic, by proving  that the  predual of $JT^{\infty}$ fails the PCP, while it is well known that the predual of $JT$ satisfies the PCP.
\end{remark}
Let   $JT^{\dtd}_{2,p}$ be the space  which is defined as  $\jtp$ but
using the dyadic tree $\dtd$ instead of an infinitely branching tree. We prove that these spaces are isomorphic.
  \begin{proposition}
    The spaces $\jtp$ and $JT^{\dtd}_{2,p}$ are isomorphic.
  \end{proposition}
 \begin{proof}
 We show that the conditions of Pelczynski's decomposition method holds.
First  we observe that  $JT^{\dtd}_{2,p}$ is isomorphic to  a complemented subspace of $\jtp$ and also   $\jtp$ is isomorphic to  a complemented subspace of $JT^{\dtd}_{2,p}$.
Indeed,taking a dyadic subtree $\dt$ of $\tree$ it readily follows that the subspace
$[e_{t}: t\in\dt]$  of   $\jtp$  is isometric to
$JT^{\dtd}_{2,p}$,  and also  is complemented  in   $\jtp$.
For the reverse embedding it is well known the there exists a subtree $\mc{T}$ of $\dtd$ such that considering the  partial order induced by $\dtd$ is an infinitely branching tree.   It readily follows  that   the subspace $[e_{t}: t\in\mc{T}]$  of $JT^{\dtd}_{2,p}$ is isometric to  $\jtp$  and also is complemented  in $JT^{\dtd}_{2,p}$.

Let us also  observe  that the space $\jtp$ is isomorphic to $\left(\sum_{i=1}^{\infty}\oplus\jtp\right)_{p}$. Indeed if
$S_{\emptyset}=\{t_{i}:i\in\N\}$  then it holds that the subspace $[e_{t}: t_{i}\preceq t]$ is isomorphic to $\jtp$. Since the nodes of $S_{\emptyset}$ are incomparable it follows
that $\jtp\backsimeq \left(\sum_{i=1}^{\infty}\oplus\jtp\right)_{p}$.
Hence  the assumptions of Pelczynski's   decomposition method  holds and therefore  $\jtp\backsimeq JT^{\dtd}_{2,p}$. 
\end{proof} 
\part{The $\jtg$-space}
This part concerns the study of the space $\jtg$ and is organized in
seven sections. In the first one we recall the definition of the
$(2,n)$-averages  and some of their
basic properties. In the second section we give the definition of the
generalized James Tree space and show how the spaces  $JT, JT^{\infty},
\jtp$ fall under this definition.
In the third section we obtain the space $\jtg$ by specifying 
for every branch $\brn$  the norm $\norm{\cdot}_{\brn}$ and the external norm $\norm{\cdot}_{*}$ which appear in the definition of the generalized James Tree space.
For every branch $\brn$ the norm $\norm{\cdot}_{\brn}$ is determined by a  
sequence of   functionals $(\phi_{n}^\brn)_{n}$.
Each  functional
$\phi_{n}^{\brn}$  is supported by a segment 
$\s_{n}^{\brn}$ of the branch $\brn$  which is a maximal
element of the Schreier family
 $\schreier[k_{n}^{\brn}]$. To each $\phi_{n}^{\brn}$
 we associate a normalized vector
 $x_{n}^{\brn}$  also supported by  $\s_{n}^{\brn}$, satisfying
 $\phi_{n}^{\brn}(x_{n}^{\brn})=1$.
For every branch $\brn$ of
the tree we define  the projection
$P_{\brn}=\sum_{n}\phi_{n}^{\brn}(\cdot)x_{n}^{\brn}$ 
which is   an essential tool in the study of the behavior  of  weakly null
sequences in the space  $\jtg$.
We show that for every finite set of  incomparable final segments
$\brn_{1,>k},\dots,\brn_{d,>k}$
the norm of the projection $\sum_{i\leq d}P_{\brn_{i,>k}}$ is one. Let us
note that the usage of the  functionals $\phi_{n}^{\brn}$  to define the
norms $\norm{\cdot}_{\brn}$ yields that the  basis of the space is not
unconditional
and hence the above inequality is not immediate.
The fourth section concerns
the study of weakly null sequences  $\sxn$
such that for every branch $\brn$ of the tree $\lim_{n}\norm{P_{\brn}(x_{n})}=0$.
In particular we prove  the following.
\begin{partproposition}
 Let $\sxn$ be a normalized weakly null  sequence in $\jtg$ such that
 $\lim_{n}\norm{\Pbrn(x_{n})}=0$, for every  branch $\brn\in\brT$. Then,
 there exists a subsequence $(x_{n})_{n\in L}$ of $\sxn$  which admits
 an upper $S^{(2)}$-estimate with constant $3$.
\end{partproposition}
This is the analogue of Proposition~\ref{l2upperestimate}
for the space  $\jtp$  and its proof
uses  a refinement of the Amemiya-Ito
lemma  in a similar manner.
As a consequence we get that the space $\jtg$ is not reflexive.
The following is proved.
\begin{partproposition}
 Every normalized block sequence $\sxn$ such that $\lim_{n}\norm{P_{\brn}(x_{n})}=0$ for every branch $\brn$ has a further block sequence that is $2$-equivalent to the unit vector basis of $c_{0}$.
\end{partproposition}
This yields the following.
\begin{partcorollary}
    The space $\jtg$ is not reflexive. In particular, for every infinite subset $N$ of $\N$ the subspace $X_{N}$
  of $\jtg$ generated  by the subsequence
  $(e_{n})_{n\in N}$ of the basis $(e_{n})_{n\in\N}$ contains a subspace isomorphic to $c_0$.
\end{partcorollary}
In the fifth section we define $\ell_{2}$-vectors   which are analogous to $\ell_{p}$-vectors which appeared in the first part.
They are defined in a similar manner with the usage of the projections $P_{\brn}$.
We prove the following.
\begin{partproposition}
  Let $\sxn$   be a normalized block sequence in $\jtg$. Then,  $\sxn$
either has a subsequence admitting an upper $S^{(2)}$-estimate or it has a
subsequence accepting  a non-zero $\ell_{2}$-vector.
\end{partproposition}
We also prove some properties of $\ell_{2}$-vectors.

In Section six we prove  two of the main properties of the space
$\jtg$.  More precisely we show  the following.
\begin{partproposition}
    Let  $(x_{n})_{n}$ be a normalized block sequence in $\jtg$.
Then there exists a subsequence of  $(x_{n})_{n}$ that admits an upper
$\ell_{2}$ estimate with constant  $9$.
\end{partproposition}
We also prove the following proposition  for sequences admitting a
non-zero $\ell_{2}$-vector
which is analogous   to the result of the first part for sequences
admitting a
non-zero $\ell_{p}$-vector.
\begin{partproposition}
    Let $\sxn$ be a normalized weakly  sequence  admitting a non-zero $\ell_{2}$-vector.  Then, there exists a subsequence $(x_{n_{k}})_{k}$ of $\sxn$ that is equivalent to the unit vector basis of $\ell_{2}$.
  \end{partproposition}
This also yields the following.
\begin{partcorollary}
    Let  $\sxn$ be a normalized block sequence in $\jtg$. Then there exists a subsequence $(x_{n_{k}})_{k}$ of $\sxn$  which is  either   equivalent to the unit vector basis of $\ell_{2}$
  or  $c_{0}$  embeds in the subspace generated by that subsequence.
\end{partcorollary}
In  the final section, i.e., Section seven, we study the Hilbertian subspaces  of
$\jtg$. We prove the following.
\begin{parttheorem}
   Let $Y$ be a  subspace of $\jtg$ such that
every normalized weakly null sequence has a subsequence equivalent to the unit vector basis of $\ell_{2}$.  Then $Y$ is   a  complemented subspace of $\jtg$ isomorphic to $\ell_{2}$. 
\end{parttheorem}
To prove the above theorem we employ some properties of the
$\ell_{2}$-vector of a normalized weakly null sequence to obtain the
following.
\begin{partproposition}
    Let $Y$ be a  subspace of $\jtg$ such that
every normalized weakly null sequence has a subsequence equivalent to the unit vector basis of $\ell_{2}$.
  Then, there exist a finite set  of 
  branches $T=\{\brnt_{1},\dots,\brnt_{d}\}$
  and a finite codimensional subspace $Y_{0}$ of $Y$
 such that  the bounded  projection  $P_{T}$,
restricted on $Y_{0}$  is an isomorphism.
\end{partproposition}

Combining the above we obtain  the following.
\begin{partcorollary}
Let   $Y$ be a  subspace of $\jtg$. Then either $Y$ is Hilbertian or  $c_{0}$ is isomorphic to a subspace of 
$Y$. In particular every reflexive  subspace of $\jtg$ is Hilbertian and  complemented.
\end{partcorollary}

\section{The Schreier families and repeated averages}
In this section we recall the definition of the Schreier families  $(\schreier[n])_{n}$ as well  the definition of repeated averages and some of their basic properties, which will be essential  in the sequel.

The Schreier families form an increasing sequence of families of finite subsets of the natural numbers and they first appeared in \cite{AA}. They are inductively defined in the following manner. Set
\[ \schreier[1]=\{F\subset\N: \# F\leq\min (F)\}\]
and if $\schreier[n]$ has been defined, set
\[
\schreier[n+1]=\{F\subset\N:  F=\cup_{i=1}^{d}F_{i}, \text{where}\,  F_{1}<\dots <F_{d}\in \schreier[n] \,\text{and}\,   d\leq\min (F_{1}) \}.
\]
The Schreier families $(\schreier[n])_{n}$ are regular families, i.e. hereditary, spreading and compact.

\subsection{$(2,n)$-averages}
We recall the definition of repeated averages, introduced in
\cite{AMT}, and some of  their
basic properties.
The $n$- repeated averages   $\alpha_{k}(n,L)$, $k\in\N$, are defined as
elements of $c_{00}(\N)$  as follows:

For $n=0$   we set $\alpha_{k}(0, L)=e_{l_{k}}$  where $L=(l_{k})_{k\in\N}$.

  Assume that $\alpha_{k} (n,L)$  have been defined for  some $n$, every $k\in\N$ and every 
  $L\in [\N]$.
  Let $L\in[\N]$.  We define $\alpha_{k}(n+1,L)$ as follows:
Let $l_{1}=\min L$.
  We define
  \[\alpha_{1}(n+1,L)=\frac{1}{l_{1}}(a_{1}(n, L)+\alpha_{2}(n, L)+\dots+\alpha_{l_{1}}(n,
    L))\]
  and $a_{k}(n+1,L)=a_{1}(n+1, L_{k})$, where
  $L_{k}=L\setminus\cup_{i=1}^{k-1}\supp a_{i}(n+1,L)$.

  The following properties were established in \cite{AMT},\cite{AT}.
\begin{lemma}\label{propaverages}
For every  $n,k\in\N$ and every $L\in [\N]$ it holds that
\begin{enumerate}
\item $\alpha_{k}(n,L)$ is a convex block
of the standard basis  $(e_{n})_{n}$  of $c_{00}(\N)$ and 
$L = \bigcup_{k=1}^{\infty} \supp a_{k}(n,L)$.
\item $supp \alpha_{1}(n,L)$  is the maximal initial segment of $L$ contained
  in  $\mathcal{S}_{n}$, 
\item $\alpha_{k}(n, L)\in S_{\ell_{1}}^{+}$  $\alpha_{k}(n, L)(i_{1})\geq \alpha_{k}(n,L)(i_{2})>0$ for every
  $i_{1}<i_{2}\in\supp \alpha_{k}(n,L)$
\item  For every $G\in \schreier[n-1]$, $\sum_{i\in
    G}\alpha_{1}(n,L)(i)<\frac{3}{l_{1}}$, $l_{1}=\min L$,
  \item  If $ L,M \in [\mathbb{N}]$
and $\supp \alpha_{k}(n,L)=\supp \alpha_{k}(n,M)$
for $k \leq m$, then $\alpha_{k}(n,L)=\alpha_{k} (n,M)$ for $k \leq m$.
\end{enumerate}
\end{lemma}
\begin{definition}
  Let  $\sxn$ be a block sequence    and $p_{n}=\minsupp(x_{n})$. A
vector $x=\sum_{i\in F}a_{i}x_{i}$ is said to be an   $(2,n)$-average if
$\sum_{i\in F}a^{2}_{i}e_{p_{i}}=
a_{1}(n,L)$  of some  $L\subset (p_{n})_{n\in F}$.
\end{definition}
From the definition  and (4) of Lemma~\ref{propaverages} it follows  that if
$x=\sum_{i\in F}a_{i}e_{p_{i}}$
is an $(2,n)$-average then for every  $G\in\schreier[n-1]$
we have that$(\sum_{i\in G\cap F}a_{i}^{2})^{1/2}<(3/p_{\min F})^{1/2}$.

\begin{notation}\label{partitionL}
  Let $L$ be an infinite subset of $\N$ and $n\in\N$.
  For every $k\in\N$ we set
  \[s_{k}(n,L)=\supp \alpha_{k}(n,L).\]
  It readily follows that  $(s_{k}(n,L))_{k}$ is a partition of $L$.
\end{notation}
The following proposition will be very useful in the sequal and  is easily proved by induction.
\begin{proposition}\label{averagesprop}
a) Let $L_{1}, L_{2}$ be  infinite subsets of $\N$ and $s$ be an initial segment of both
  $L_{1}$ and $L_{2}$.  If $\max s\in s_{k}(n,L_{1})$ then
  \[
s_{j}(n,L_{1})=s_{j}(n,L_{2})\,\,\text{for every $j<k$}\quad\text{and}\qquad s_{k}(n,L_{1})_{|s}=s_{k}(n,L_{2})_{|s}.
  \]
b)   Let $F,G$ be maximal
 elements of $\schreier[n]$ so that $F\cap G$ is an initial
segment of both $F$ and $G$. Let also 
$\sum_{k\in F}a_{k}e_{k}$ and $\sum_{k\in G}b_{k}e_{k}$ be  the $(2,n)$- averages
supported by $F$ and $G$ respectively. Then, for every $k\in F\cap G$, we have $a_{k}=b_{k}$.
\end{proposition}

          The Schreier space  $S^{(1)}$ is the completion of $c_{00}(\N)$ under
          the norm given by the formula
          \[\norm{\sum_{i=1}^{n}a_{i}e_{i}}_{S^{(1)}}=\sup\{\sum_{i\in
              F}\abs{a_{i}}: F\in\schreier[1]\}.
\]
The 2-convexification  of  $S^{(1)}$ is  denoted by
$S^{(2)}$ and it is  
the completion of $c_{00}(\N)$ under
          the norm
\[\norm{\sum_{i=1}^{n}a_{i}e_{i}}_{S^{(2)}}=\sup\{(\sum_{i\in F}a_{i}^{2})^{1/2}:
  F\in \mc{S}_{1}\}.
\]
\section{The definition of the generalized  $\jt$-space}
In this section we provide the definition of a variety of  generalized James tree spaces. Varieties of the James tree space have been given in the pasts and our definition can bee seen as combination of the two more general ones, the one of Bellenot, Haydon Odell,\cite{BHO} and the one of Albiac and Kalton \cite{AK}.
        \begin{notation}
As in the definition of $\jtp$ we consider a partial order $\prec$ on the
set  of natural numbers compatible  with the usual order, defining an infinitely
branching tree $\tree_{0}$ without a root and terminal nodes. Let $\tree$ be the tree we
obtain by subjoining to $\tree_{0}$ the  empty set as a root.
We shall denote by          $\brT$  the set of the branches  of the tree.
        \end{notation}
    \subsection{Definition of the general $JT$-space}

For every     branch  $\brn=(k_{n})_{n\in\N}$ of  $\tree$  let  $\norm{\cdot}_{\brn}$
denotes a norm such that   the sequence $(e_{k_{n}})_{n}$  is a
bimonotone basis  of the completion of $c_{00}(\brn)$ under
the norm  $\normm_{\brn}$, i.e.
the norms of the projections on the  segments  of $\brn$  are
uniformly bounded  by 1.

\begin{definition}
The family of the norms $(\norm{\cdot}_{\brn})_{\brn\in\brT}$ is said to be
tree compatible   if  for every segment  $s$ and every pair  of distinct
branches $\brn_{1},\brn_{2}$ such that $s\subset\brn_{1}\cap\brn_{2}$  we have that
\begin{equation*}
 \norm{x}_{\brn_{1}}=\norm{x}_{\brn_{2}} \,\,\,\text{for all $x$ such
   that $\supp (x)\subset s$}.
\end{equation*}  
\end{definition}
If $(\norm{\cdot}_{\brn})_{\brn\in\brT}$ is a family of 
tree compatible norms then for every segment  $s$ of $\tree$ we define
\begin{equation*}
  \norm{x}_{s}=\norm{x_{|s}}_{\brn} \,\,\,\text{where $\brn$ is a branch
  such that $s\subset\brn$}.
\end{equation*}
The tree compatibility of the norms implies that  $\norm{\cdot}_{s}$ is
well defined.Let $\normm_{*}$  be  a $1$-unconditional norm on $c_{00}(\N)$.
The general  $JT$-space corresponding to $(\|\cdot\|_\sigma)_{\sigma\in\mathcal{B}(\mathcal{T})}$, $\|\cdot\|_{*}$
is defined to be the completion of
$c_{00}(\N)$  under the norm
\begin{equation}
  \label{eq:13}
 \norm{x}=\max\left\{\norm{x}_{\infty},\sup\{
 \norm{
 \sum_{i=1}^{n}\norm{x}_{s_{i}}e_{\min (s_i)}
 }_{*}:
\s_{1},\dots,\s_{n}\,\text{are disjoint segments}, n\in\N\}\right\}.
\end{equation}
\subsection{Examples of generalized $JT$-spaces}
\begin{enumerate}[leftmargin=17pt]
\item
The first such example is the  space defined by  
N. Ghoussoub and B. Maurey \cite{GhM}, denoted by
$JT^{\infty}$. This space is defined as  the  James tree space \cite{J2}, using
an infinitely  branching tree
instead of the dyadic tree.
Following our definition, let  for every branch $\brn$, $\normm_{\brn}$ be the completion of
$c_{00}(\brn)$ under the  summing norm.
Setting $\normm_{*}=\normm_{\ell_{2}}$,   the generalized  James tree space
 $JT^{\infty}$ is
 the space we  get  by the completion of $c_{00}(\N)$  using the norm
 given by \eqref{eq:13}.
\item
  The second example of generalized $JT$-space is the space
  $\jtp$,  $2<p<+\infty$. For every branch $\brn$   let
  $\normm_{\brn}=\normm_{\ell_{2}}$ and $\normm_{*}=\normm_{\ell_{p}}$.
Then the corresponding generalized JT-space is $\jtp$.
  \item
  Let $\brn$ be a branch of  $\tree$  and let
  $(\phi_{n}^{\brn})_{n}$ be a sequence of   successive elements of
  $c_{00}(\brn)$  such that    $\cup_{n}\supp(\phi_{n}^{\brn})=\brn$  and
  $\norm{\phi_{n}^{\brn}}_{\infty}\leq 1$ for every $n\in\N$. Assume additionally that $(\phi_{n}^{\brn})_{n}$, $\sigma\in\mathcal{B}(\mathcal{T})$,  are chosen in such way so they are tree
  compatible, i.e.,
if  $s$ is a segment of  $\tree$ and $\brn_{1},\brn_{2}$ are distinct
branches such that $s\subset\brn_{1}\cap\brn_{2}$  then
$\phi_{n|s}^{\brn_{1}}=\phi_{n|s}^{\brn_{2}}$ for every $n\in\N$.

For every $\brn$  we set
\[
W(\brn)=\{E\sum_{i=1}^{n}\lambda_{i}\phi_{i}^{\brn}:   (\lambda_{i})_{i=1}^{n}\in
B_{\ell_{2}},  (\lambda_{i})_{i=1}^{n}\subset\Q, n\in\N, \,E\,\text{interval of $\N$}\}.
\]
We define for $x\in c_{00}(\brn)$
\[
\norm{x}_{\brn}=\sup\{\phi(x):\phi\in W(\brn)\}.
\]
Using the tree compatibility of $(\phi_{n}^{\brn})_{n}$  
it follows that the norms $\normm_{\brn}, \brn \in \brT$  are tree compatible.

Let $\normm_{*}$ be the norm of  $\mc{S}^{(2)}$.
Under these  settings the norm of the corresponding generalized  James tree space is
given by the formula
\[
\norm{x}=\max\left\{\norm{x}_{\infty},\sup\{\norm{\sum_{i=1}^{n}\norm{x}_{s_{i}}e_{\min
      (s_i)}}_{S^{(2)}}:\s_{1},\dots,\s_{n}\,\,\text{are disjoint segments},n\in\N\}\right\}
  .
\]
\end{enumerate}
\section{The space $\jtg$}
In this section we are going to define the space $\jtg$,  which is a specific version of the last  example.
As we have mention in the introduction the norming set of the space $\jtg$ we be an essential component of the norming set of the space $\xeh$ and will have an significant role  to the proof of the properties of $\xeh$.

Let  $(m_{j})_{j\in\N}, (n_{j})_{j\in\N}$ be two strictly increasing sequences of positive
integers such that $m_{j}\leq n_{j}$ for every $j$.

Let $\cod:\N\cup\{\emptyset\}\to \N$ be  a strictly increasing function
such that 
$\cod(\emptyset)<\cod(n)$ for every $n\in\N$.
Let $\tree_{0}$ be an infinite branching tree  with no terminal nodes.
We consider the complete, infinitely branching tree $\tree$  of $\tree_{0}$  with the additional property,  
\begin{equation}\label{imms}
k(n)=\min\{k\in\N: k\,\,\text{is an immediate successor  of $n$}\}>(3m_{2\gamma(n)+1}^{2})^{2} .
\end{equation}

\begin{notation}
  We partition each branch $\brn$ of $\tree$ into successive segments  $\s_{k}^{\brn}$,$k\in\N$. This partition is a more carefully defined version of the one from Notation \ref{partitionL}.

 Let  $\brn$ be a branch of $\tree$  and $j_{1}=\cod(\emptyset)$. Then $s_{1}^{\brn}$ is the initial  segment  of $\brn$ which  is maximal in $\mc{S}_{n_{2j_{1}+1}}$.
Assume we have define
$\segs[1],\dots,\segs[k-1]$.
Let
$j_{k}=\cod(\max\segs[k-1])$.
We take $\s_{k}^{\brn}$ 
be the initial segment of
$\brn\setminus\cup_{i=1}^{k-1}\segs[i]$
which is maximal in $\schreier[n_{2j_{k}+1}]$.

It readily follows that  $\s_{k}^{\brn}<\s_{k+1}^{\brn}$ for every $k$ and that $(\s^{\brn}_{k})_{k}$ is a partition of $\brn$.
\end{notation}
Proposition~\ref{averagesprop}(a)  yields  that  if $\brn_{1},\brn_{2}$ are distinct
branches and $s$ is the maximal segment such that
$s\subset\brn_{1}\cap\brn_{2}$  and $\max\s\in \segss{n_{0}}{1}$ then
\begin{equation*}
  \segss{n|s}{1}=\segss{n|s}{2}\,\,\,\text{for all $n\leq n_{0}$}.
\end{equation*}
In particular   $\segss{n}{1}=\segss{n}{2}$  for all $n< n_{0}$
and $\segss{n_{0}|s}{1}=\segss{n_{0}|s}{2}$.

For a branch $\brn$ we define the sequence of pairs 
$(x_{k}^{\brn},\phi_{k}^{\brn})_{n}$ where

(i) $x_{k}^{s}=m_{2j_{k}+1}^{2}\sum_{l\in\segs[k]}\beta_{l}e_{l}$ where
$\sum_{l\in\segs[k]}\beta_{l}e_{l}$ is the  $(2,n_{2j_{k}+1}^{\brn})$-basic average supported
by the set  $\segs[k]$,

(ii) $\phi_{k}^{\brn}=m_{2j_{k}+1}^{-2}\sum_{l\in\segs[k]}\beta_{l}e_{l}$.

We set
\[W(\brn)=\{E\sum_{i=1}^{n}\lambda_{i}\fns[i]: \lambda_{i}\in\Q\,\,\text{for every $i$}, \,\,(\lambda_{i})_{i=1}^{n}\in
  B_{\ell_{2}},\,E\,\,\text{an interval of $\N$}, n\in\N\}.
\]
Using that the partial order of $\tree$ is compatible with the usual order of $\N$, it follows  that for every interval $E$ of  $\N$ such that $E\cap\brn\ne\emptyset$ it holds that
$E\cap\brn=\s$ is a segment of $\brn$. This yields that
\[W(\brn)=\{E\sum_{i=1}^{n}\lambda_{i}\fns[i]: \lambda_{i}\in\Q\,\,\text{for every $i$},\,\,(\lambda_{i})_{i=1}^{n}\in
  B_{\ell_{2}},\,E\,\,\text{a segment  of $\brn$}, n\in\N\}.
\]
Let $\normm_{\brn}$ be the norm defined on  $c_{00}(\brn)$ by the formula
\begin{equation*}
  \norm{x}_{\brn}= \sup\{f(x): f\in W(\brn)\}.
\end{equation*}
It readily follows   that $(e_{n})_{n\in\brn}$ is a   
bimonotone basis for the completion of $c_{00}(\brn)$ under the norm
$\normm_{\brn}$.

We show now that the norms $\normm_{\brn},\brn\in\brT$ are tree
compatible.
Let  $\brn$ be branch of $\tree$. For a segment $\s$ of $\brn$ we set
\[
W^{\brn}(\s)=\{f: f\in W(\brn)\,\,\text{and}\,\,\, \supp(f)\subset \s\}.
\]
From Proposition~\ref{averagesprop}b) it follows that if $\s\subset\brn_{1}\cap\brn_{2}$
and $\max(s)\in \s_{n_{0}}^{\brn_{1}}$ then
\[
  \fnss{n}{1}=\fnss{n}{2}\,\text{for every $n< n_{0}$ and }
\,\, \fnss{n_{0}|s}{1}=\fnss{n_{0}|s}{2}.
\]
This implies  that $W^{\brn_{1}}(\s)=W^{\brn_{2}}(\s)$ hence $W^{\brn}(s)$ is independent of the branch $\brn$ such that $\s\subset\brn$. This yields
that the norms   $\normm_{\brn},\brn\in\brT$ are tree
compatible.

In the rest of this part given  a  segment $\s$, unless otherwise  is stated, $t_{\s}$  will denote the minimum of  $\s$, i.e., $t_{\s}=\min\s$.

Let $\normm_{*}$ be the norm  of the $2$-convexification of the Schreier space.  The generalized James tree space we get
under these settings is denoted $JT_{G}$ and is the completion of
$c_{00}(\N)$  under the norm
\begin{equation*}
\begin{aligned}
  \norm{x}=\max\Big\{
  \norm{x}_{\infty},
\sup\{
\norm{
\sum_{i=1}^{n}
\norm{x}_{\s_{i}}e_{t_{\s_{i}}}}_{S^{(2)}}
 &:
 s_{1},\dots,s_{n},\,\text{are disjoint segments and}
 \\
 \,
 & \qquad\qquad  \{\min\s_{i}:i\leq n\}\in\schreier[1]
 \}
\Big\}.
\end{aligned}
\end{equation*}
\subsection{An alternative definition of the norm}
We consider the set
\begin{align*}G=\{\sum_{i=1}^{n}\beta_{i}f_{i}&: (\beta_{i})_{i=1}^{d}\in  B_{\ell_{2}},
  f_{i}\in W(\brn_{r_{i}})\,\,\text{for some branch $\brn_{r_{i}}$},
  \\
&  \supp f_{i}\cap\supp f_{j}=\emptyset\,\,\text{if $i\ne j$ and}\,\,
  \{\min\supp  f_{i}:i\leq n\}\in\schreier[1]
  \}.
\end{align*}
Let $\norm{\cdot}_{G}$  be the norm
\[
\norm{x}_{G}=\sup\{e^{*}_{n}(x):n\in\N\}\vee \sup\{f(x): f\in G\}.
\]
It readily follows that  $\norm{x}=\norm{x}_{G}$ for any $x\in c_{00}(\N)$.
\begin{remark}\label{numberds}
For every $\e>0$
 there exists  $n(\e)\in\N$ satisfying the following: if $x\in c_{00}(\N)$ is such that $\normg{x}\leq 1$ and $\s_{1},\dots,\s_n$  are disjoint segments such that $\norm{x}_{\s_{i}}\geq \e$ for
 every $i\leq n$ 
 then   $n\leq n(\e)$.

 Indeed, assume that there exist $\s_{1},\dots,\s_{k}$ such that
 $\norm{x}_{\s_{i}}\geq\e$ for every $i$. Then at least $\floor{k/2}$   of them are
 Schreier admissible. Let the $\schreier[1]$ admissible segments be
 the  $\{s_{i_{1}},\dots, \s_{i_{d}}\}$,
 $d\geq\floor{\frac{k}{2}}$.  It follows:
 \[
  1\geq \normg{x}\geq\norm{\sum_{j=1}^{d}\norm{x}_{\s_{i_{j}}}e_{t_{\s_{i_{j}}}}}_{S^{(2)}}
   \geq \left(\sum_{j=1}^{d}\norm{x}_{\s_{i_{j}}}^{2}\right)^{1/2}\geq d^{1/2}\e.
 \]
This yields  $\floor{\sfrac{k}{2}}\leq d\leq(1/\e)^{2}$ and hence $k\leq\floor{\sfrac{2}{\e^{2}}}+1$.
\end{remark}
\begin{lemma}
  For every branch $\brn$ and every $\xns=m_{2j_{n}+1}^{2}\sum_{j\in\s_{n}^{\brn}}\beta_{j}e_{j}$,
  $n\in\N$,  it holds that $\norm{\xns}=1$.
 \end{lemma}
\begin{proof}
 From the definition of the pairs $(\xns,\fns)$  it readily follows
 that
 \begin{equation*}
 \norm{\xns}\geq\fns(\xns)=1.   
 \end{equation*}
To prove that $\norm{\xns}\leq 1$, let   $\s$ be a segment  such that $\s\cap\segs\ne\emptyset$.
Using  the compatibility of the norms
$\normm_{\brn}, \brn\in\brT$, it follows that
\[
  \norm{\xns}_{s}=\norm{x_{n|\s\cap\s_{n}^{\brn}}^{\brn}}_{\brn}=
  \phi_{n}^{\brn}(x_{n|\s\cap\s_{n}^\brn}^{\brn})=\sum_{j\in \s\cap\s_{n}^{\brn}}\beta_{j}^{2}\leq
 \left(\sum_{j\in\s\cap\s_{n}^{\brn}}\beta_{j}^{2}\right)^{1/2}.
\]
 Using the above inequality,   for every disjoint segments $\s_{1},\dots,\s_{n}$
such that  $\{\min\s_{i}:i\leq n\}\in\schreier[1]$,
 we get
  \begin{equation}
    \norm{\sum_{i=1}^{n}\norm{\xns}_{\s_{i}}e_{t_{\s_{i}}}}_{S^{(2)}}
    =\left(\sum_{i=1}^{n}\norm{\xns}_{\s_{i}}^{2}\right)^{1/2}
    \leq \left(\sum_{i=1}^{n}
\sum_{j\in\s_{i}\cap\s_{n}^{\brn}}\beta_{j}^{2}
    \right)^{1/2}\leq 1.
    \label{norma1-1}
  \end{equation}
 Also it holds that  $\norm{\xns}_{\infty}<1$.  Indeed since  $\sum_{j\in\snbrn}\beta_{j}^{2}e_{j}$ is an $\alpha_{1}(2,n_{2j_{n}+1})$-average
 from  (4) of Lemma~\ref{propaverages}  we have that $\max\{\beta^{2}_{j}:j\in\snbrn\}\leq\frac{3}{\min\snbrn}$. 
 By the choice of the subtree $\tree$, see \eqref{imms},  we  get
 \[
 \max\{\beta^{2}_{j}:j\in \snbrn\}\leq
\frac{3}{\min\snbrn}<
\frac{3}{9m_{2j_{n}+1}^{4}}.\]
 From the above inequality  we get that $\norm{\xns}_{\infty}=
 m_{2j_{n}+1}^{2}\max\{\beta_{j}:j\in\snbrn\}\leq 1$.
Combining the above inequality with \eqref{norma1-1}   we have that $\norm{\xns}\leq 1$.
\end{proof}

\begin{lemma}\label{upperl2}
For every   $\brn\in\brT$ the sequence  $(\xns)_{n}$ admits an upper
$\ell_{2}$-estimate with constant 1. In particular, $(\xns)_{n}$  is 
$1$-equivalent  to the  unit vector basis of $\ell_{2}$.
\end{lemma}

\begin{proof}
  Let $d\in\N$,  $a_{1},\dots,a_{d}\in\R$  and  $x=\sum_{n=1}^{d}a_{n}\xns$.
  Let $\s$ be a segment such that  $\s\cap\supp
x\ne\emptyset$. Set $n_{1}=\min\{n:\s\cap\segs \ne\emptyset\}$ and
$n_{2}=\max\{n:\s\cap \segs \ne\emptyset\}$. From the compatibility of the norms
$\normm_{\brn},\brn\in\brT$ it follows that
$ \norm{x}_{s}=\norm{x_{|s}}_{\brn}$.   Then
\begin{equation}
\begin{aligned}
  \norm{x}_{s}&=\norm{x_{|s}}_{\brn}=
  \sup\{\phi(x_{|s}): \phi \in W(\brn)\}\\ 
&= (a_{n_{1}}^{2}(\sum_{j\in\s\cap\segss{n_{1}}{}}\beta_{j}^{2})^{2}+\sum_{n=n_{1}+1}^{n_{2}-1}a_{n}^{2}+a_{n_{2}}^{2}(\sum_{j\in\s\cap\segss{n_{2}}{}}\beta_{j}^{2})^{2}
                                       )^{1/2}
\\
 & \leq
(a_{n_{1}}^{2}\sum_{j\in\s\cap\segss{n_{1}}{}}\beta_{j}^{2}+\sum_{n=n_{1}+1}^{n_{2}-1}a_{n}^{2}+a_{n_{2}}^{2}\sum_{j\in\s\cap\segss{n_{2}}{}}\beta_{j}^{2}
)^{1/2}.
    \label{eq:8aaa}
\end{aligned}
\end{equation}
Let $\s_{1},\dots,s_{m}$ disjoint segments such that  $m\leq \min\s_{i}$
for every $i\le m$.
Define for every  $\s_{i}$,   $n_{i,1}, n_{i,2}$ as we defined
$n_{1},n_{2}$ above.

Setting  $C=\{n\leq d: \segs\cap\s_{i}\ne\emptyset\,\,\text{for some}\,\, i\leq m\}$,
from   \eqref{eq:8aaa} and the fact that
$\s_{i}$ are disjoint segments we get,
\begin{align*}
  \norm{\sum_{i=1}^{m}\norm{x}_{\s_{i}}e_{t_{i}}}_{S^{(2)}}&=
  \left(\sum_{i=1}^{m}\norm{x}_{\s_{i}}^{2}\right)^{1/2}\\
  &\leq\left(\sum_{i=1}^{m}
    (a_{n_{i,1}}^{2}\sum_{j\in\s_{i}\cap\segss{n_{i,1}}{}}    \beta^{2}_{j}
    +\sum_{n=n_{i,1}+1}^{n_{i,2}-1}a_{n}^{2}+a_{n_{i,2}}^{2}
    \sum_{j\in\s_{i}\cap\segss{n_{i,2}}{}}\beta_{j}^{2})
    \right)^{1/2}
  \\
 &\leq\left(\sum_{n\in C}a_{n}^{2}\right)^{1/2}.                                                   
\end{align*}
Since it holds also that
$\norm{\sum_{n=1}^{d}a_{n}\xns}_{\infty}\leq\left(\sum_{n=1}
  ^{d}a_{n}^{2}\right)^{1/2}$, 
we get
\[
\normg{\sum_{n=1}^{d}a_{n}\xns}\le\left(\sum_{n=1}^{d}a_{n}^{2}\right)^{1/2}.
\]
This completes the proof that $(\xns)_{n}$ admits upper
$\ell_{2}$-estimate with constant 1.

To see that $(\xns)_{n}$ is  $1$-equivalent  to the unit vector basis of 
$\ell_{2}$,
let $(\beta_{n})_{n=1}^{d}\in B_{\ell_{2}}$ be the $\ell_{2}$-conjugate  sequence of $(a_{n})_{n=1}^{d}$.
From the definition of the
 norm it follows that
 \[
   \normg{x}\geq\norm{x}_{\brn}\geq
   \sum_{n=1}^{d}\beta_{n}\fns(x)=
   \sum_{n=1}^{d}\beta_{n}a_{n}=
   \left(\sum_{n=1}^{d}a_{n}^{2}\right)^{1/2}.
 \]
This completes the proof that $(\xns)_{n}$ is isometric to the unit vector basis 
of $\ell_{2}$. 
\end{proof}
\begin{notation}
  Throughout the rest of Part 2, for a branch $\brn$ of $\tree$ and $k\in\N$  we denote by $\brn_{,>k}$ the final segment
  \[\brn_{,>k}=\cup_{j=j(k)}^{+\infty}\s_{j}^{\brn}
\,\,\,\,\text{where}\,\, j(k)=\min\{j:k\leq\min\s_{j}^{\brn}\}.
  \]
\end{notation}
\begin{notation}
  For every  branch $\brn$ of $\tree$ we denote by  $Z_{\brn}$ the subspace of $\jtg$
  generated  by the sequence $(\xns)_{n}$.

  We also denote by $Z_{\brn,>k}$    the subspace of $\jtg$
  generated  by the $\xns$ such that $\snbrn\subset\brn_{,>k}$.
\end{notation}
\begin{lemma}\label{supperl2}
  Let $k,m\in\N$, $\brn_{1},\dots,\brn_{m}$ be  distinct branches
  such that $\brn_{j,>k}$, $1\leq j\leq m$ are  incomparable.  Then,  if $ x_{j}\in Z_{\brn_{j},>k}$, $1\leq j\leq m$,   are finitely supported 
and $m\leq \minsupp x_{j}$ for every $j\leq m$
we have that,
  \[ \normg{\sum_{j=1}^{m}x_{j}}=\left(\sum_{j=1}^{m}\norm{x_{j}}^{2}\right)^{1/2}.
    \]
\end{lemma}
\begin{proof} Let $x=\sum_{j=1}^{m}x_{j}$ and $p\leq\s_{1},\dots,\s_{p}$  be  disjoint segments
  such that   $\s_{i}\cap\supp(x)\ne\emptyset$  for every $i\leq p$.
 Since the  final
  segments $\brn_{j,>k}$ are  incomparable we get that for every $i$ it holds
  $\s_{i}\cap\brn_{j,>k}\ne\emptyset$ for a unique $j$.
Clearly we may assume that $m\leq\s_{i}$ for every $i\leq p$.
  For every $j$ set $C_{j}=\{i\leq
  p:\s_{i}\cap\brn_{j,>k}\ne\emptyset\}$.
  It  follows that
  \begin{equation*}
   \sum_{i\in C_{j}}\norm{x}_{\s_{i}}^{2}=\sum_{i\in
     C_{j}}\norm{x_{j}}_{s_{i}}^{2}\leq\norm{x_{j}}^{2}.
  \end{equation*}
  Therefore
  \begin{align}
    \label{eq:12}
    \norm{ \sum_{i=1}^{p}\norm{x}_{\s_{i}}e_{t_{\s_{i}}}}_{S^{(2)}}=
    \left(\sum_{j=1}^{m}\sum_{i\in C_{j}}\norm{x_{j}}_{\s_{i}}^{2}\right)^{1/2}
    \leq\left(\sum_{j=1}^{m}\norm{x_{j}}^{2}\right)^{1/2}.
  \end{align}
 Inequality  \eqref{eq:12} yields that
  $\normg{\sum_{j=1}^{m}x_{j}}\leq\left(\sum_{j=1}^{m}\norm{x_{j}}^{2}\right)^{1/2}$.
Taking the disjoint segments $\s_{j}=\rang(x_{j})\cap\brn_{j}$, $j\leq m$ we have
that $\{\min\s_{j}:j\leq m\}\in \schreier[1]$ and
\[\normg{\sum_{j=1}^{m}x_{j}}\geq\norm{\sum_{j=1}^{m}\norm{x_{j}}_{\s_{j}}e_{t_{\s_{j}}}}_{S^{(2)}}
  =\left(\sum_{j=1}^{m}\norm{x_{j}}^{2}\right)^{1/2}.
\]
Combining the above inequality with \eqref{eq:12} we have the result.
\end{proof}

\begin{definition}
For  $\brn\in\brT$  we denote  by  $P_{\brn}$  the projection
$P_{\brn}:JT_{G}\to Z_{\brn}$,  defined as
\[
P_{\brn}(x)=\sum_{n=1}^{\infty}\fns(x)\xns.
\]
Also,  for $j\in\N$  let  $P_{\brn_{,>j}}$ denote the projection
\[
P_{\brn_{,>j}}(x)=\sum_{n=n(j)}^{{\infty}}\fns(x)\xns\quad \text{where}\,\,  n(j)=\min\{n:\s_{n}^{\brn}\subset\brn_{,>j}\}.
\]

\end{definition}
\begin{notation}
  Let $F$  be a finite subset of $\N$, $\mc{B}_{F}=\{\brn_{i}:i\in F\}$
  a set of branches and $k\in\N$ such that  the final segments
  $\brn_{i,>k}$, $i\in F$, are incomparable.
  By $P_{\mc{B}_{F,>k}}$ we shall denote the projection
 \[
P_{\mc{B}_{F,>k}}=\sum_{i\in F}P_{\brn_{i,>k}}.
  \]
\end{notation}
\begin{lemma}
 For  $\brn\in\brT$  and $j\in\N$ it holds  that $\norm{P_{\brn,>j}}=\norm{P_{\brn}}=1$.
\end{lemma}
\begin{proof}
 Let $x\in c_{00}(\N)$.  We may assume that $\fns(x)\ne0$ for at least one
 $n$, otherwise  trivially
 holds that  $\norm{P_{\brn}(x)}\leq\norm{x}$.  Let $I=\{n:\supp(\phi_{n}^{\brn})\cap\supp(x)\ne\emptyset\}$ . 
 Then 
 \begin{align*} \normg{P_{\brn}(x)}=
 \normg{\sum_{n}\fns(x)\xns}&
 \leq\left(\sum_{n\in I}\fns(x)^{2}\right)^{1/2}\,\,\text{by
   Lemma~\ref{upperl2}}\\
     &\leq\norm{x}_{\brn_{j}}\leq\normg{x}.                           
 \end{align*}
 Also it is clear that  $\norm{P_{\brn,>j}}\leq\norm{P_{\brn}}$.
Since $\norm{\xns}=1$ for every $n$  it follows that  $\norm{P_{\brn,>j}}\geq\norm{P_{\brn,>j}(x_{j}^{\brn})}=1$.
\end{proof}

\begin{lemma}\label{sumprojections}
  Let  $F$ be a finite subset  of $\N$, $ \{\brn_{j}:j\in F\}$ be  distinct branches
and $k\in\N$  such that the final segments $\brn_{j,>k}$, $j\in F$, are  incomparable.
  Then
if $\#F\leq k$ and $x\in
c_{00}(\N)$ is such that $\minsupp(x)\geq\max\{\min \brn_{j,>k}:j\in F\}$
we have that
 \[
   \norm{P_{\mc{B}_{F,>k}}(x)}=
  \left( \sum_{j\in F}\norm{P_{\brn_{j,>k}}(x)}^{2}\right)^{\frac{1}{2}}
   \leq\normg{x}.
\]
\end{lemma}
\begin{proof}
Let $x\in c_{00}(\N)$ such that
$\minsupp(x)\geq\max\{\min \brn_{j,>k}:j\in F\}$. It follows that the vectors
  $P_{\brn_{j}}(x)=P_{\brn_{j,>k}}(x)$, $j\in F$, satisfies the assumption
of Lemma~\ref{supperl2}. 
Hence
  \begin{equation*}
    \norm{P_{\mc{B}_{F,>k}}(x)}=
    \left(\sum_{j\in F}\norm{P_{\brn_{j,>k}}(x)}^{2} \right)^{1/2}
.
\end{equation*}
Let $\s_{1},\dots,\s_{p}$ be disjoint segments with $p\leq\min\s_{i}$ for every $i\leq p$.
    Since
the final segments $\brn_{j,>k}$,$j\in F $, are  incomparable  it follows
that for every $i\leq  p$ there exists at most one $j\in F$ such that
$\s_{i}\cap\brn_{j,>k}\ne\emptyset$.
For every $j\in F$, set $C_{j}=\{i\leq p: \s_{i}\cap\brn_{j,>k}\ne\emptyset\}$.
For every $i\in C_{j}$ we have that
\[P_{\brn_{j,>k}}(x)=\sum_{n\geq n(k)}\fns(x)\xns\Rightarrow
P_{\brn_{j,>k}}(x)_{|s_{i}}=\sum_{n\geq n(k)}\fns(x)x_{n|\s_{i}}^{\brn}.
\]
Let
$I_{j}=\{n: \s_{n}^{\brn}\subset\brn_{j,>k}\,\text{and}\,\,\snbrn\cap\s_{i}\ne\emptyset\,\,\text{for some}\, i\in C_{j}\}$.
For every $n\in I_{j}$ and $i\in C_{j}$  it holds $\fns(x_{n|\s_i}^{\brn})=
\sum_{l\in s_{i}\cap\snbrn}\beta_{l}^{2}$.
Using the compatibility  of the norms $(\norm{\cdot}_{\brn})_{\brn\in\brT}$ it follows
\begin{equation}
	\begin{aligned}
  \label{eq:28}
    \norm{P_{\brn_{j,>k}}(x)}^{2}_{\s_{i}}&=
    \norm{P_{\brn_{j,>k}}(x)_{|\s_{i}}}_{\brn_{j}}^{2}=
 \sum_{n\in I_{j}}\fns(x)^{2}\fns(x_{n|\s_{i}}^{\brn})^{2}
\\
& \leq\sum_{n\in I_{j}}\fns(x)^{2}\sum_{l\in s_{i}\cap\snbrn}\beta_{l}^{2}.
\end{aligned}
\end{equation}
Let $l_{j}=\min(\brn_{j,>k}\cap \ran(x_{j}))$. 
Using that  the  segments $\s_{i}, i\leq p$ are disjoint 
and \eqref{eq:28} it follows,
\begin{equation}
	\begin{aligned}
		\label{eq:280}
		\sum_{i\in C_{j}}\norm{P_{\brn_{j,>k}}(x)}^{2}_{\s_{i}}&\leq
\sum_{n\in I_{j}}\fns(x)^{2}\sum_{i\in C_{j}}\sum_{l\in s_{i}\cap\snbrn}\beta_{l}^{2}\\
&\leq \sum_{n\in I_{j}}
\fns(x_{|\brn_{j,\geq l_{j}}})^{2}\leq \norm{x_{|\brn_{j,\geq l_{j}}}}^{2}.
	\end{aligned}
\end{equation}
Finally we set $B=\{j\in F: C_{j}\ne\emptyset\}$.  Note that
$\# B\leq k\leq \min \brn_{j,\geq l_{j}}$ for every $j\in B$.
Combining the above  we get
\begin{align}
  \label{eq:29}
  \sum_{i=1}^{p}\norm{P_{\mc{B}_{F,>k}}(x)}^{2}_{\s_{i}}
  =\sum_{j\in B}\sum_{i\in C_{j}}\norm{P_{\mc{B}_{F,>k}}(x)}^{2}_{\s_{i}}
  \leq\sum_{j\in B}\norm{x_{|\brn_{j,\geq l_{j}}}}^{2}
  \leq\norm{x}^{2}.
\end{align}
Also  $\norm{P_{\mc{B}_{F,>k}}(x)}_{\infty}\leq\norm{x}$.   This 
inequality combined with
\eqref{eq:29} yields that 
$\norm{P_{\mc{B}_{F,>k}}}\leq 1$. In particular $\norm{P_{\mc{B}_{F,>k}}}= 1$  since $P_{\mc{B}_{F,>k}}$ is a projection.
\end{proof}
\section{Upper $S^{(2)}$-estimates}
In this section we will prove the theorem bellow. By slightly abusing language, we will say that a sequence $(x_n)_{n\in\N}$ in a Banach space admits an upper $S^{(2)}$-estimate with constant $C$, if it is $C$-dominated by a subsequence of the unit vector basis of $S^{(2)}$.

\begin{theorem}\label{upperS2}
 Let $\sxn$ be a normalized block sequence in $\jtg$ such that
 $\lim_{n}\norm{\Pbrn(x_{n})}=0$, for every  branch $\brn\in\brT$. Then,
 there exists a subsequence $(x_{n})_{n\in L}$ of $\sxn$  which admits
 an upper $S^{(2)}$-estimate with constant $3$.
\end{theorem}
We start with the following remark.
\begin{remark}
 Note that if  for some  $f\in W(\brn)$ and an   interval $I$ we have that
$\abss{f_{|I}(x)}\geq\e$  it does not follow that
$\normg{P_{\brn}(x)}$ is  necessarily  comparable to
$\abss{f_{|I}(x)}$.

For example, let $x_{n}^{\brn}=m_{2j_{n}+1}^{2}\sum_{k\in\segs}\beta_{k}e_{k}$
and partition $\segs=\s_{1}\cup \s_{2}$ where $s_{1}$ is the minimal
initial segment of $\segs$ such that $\sum_{k\in
  s_{1}}\beta^{2}_{k}\geq\frac{1}{2}$.  Let $I=\s_{1}$
and  $x=x_{n|I}^{\brn}-x_{n|I^{c}}^{\brn}$. It readily follows that
$\phi_{n|I}^{\brn}(x)\geq\frac{1}{2}$ while
\[
  \normg{P_{\brn}(x)}=\norm{\fns(x)\xns}=\abs{\fns(x)}\leq2\beta^{2}_{\max\s_{1}}
  \leq\frac{2}{(\min\xns)}.
\]
\end{remark}
\begin{lemma}\label{fprojection2}
Let $\brn\in\brT$, $f\in W(\brn)$ and  $x\in c_{00}(\N)$ such that $\supp(x)\cap\brn\subset\ran(f)$. Then  $\abs{f(x)}\leq\norm{P_{\brn}(x)}$.
\end{lemma}
\begin{proof}
  Let  $f=E\sum_{n=1}^{d}\lambda_{n}\fns\in W(\brn)$.
  Set $n_{1}=\min\{n\leq d: \supp(E\fns)\cap\supp (x)\ne\emptyset\}$
  and $n_{2}=\max\{n\leq d: \supp(E\fns)\cap\supp (x)\ne\emptyset\}$. Then
  \begin{align*}
    \abss{f(x)}&=\abss{\lambda_{n_1}E\phi_{n_1}^{\brn}(x)+\sum_{n=n_{1}+1}^{n_{2}-1}
    \lambda_{n}\fns(x)+\lambda_{n_2}E\phi_{n_2}^{\brn}(x)}\\
    &=\abss{\lambda_{n_1}\phi_{n_1}^{\brn}(x)+\sum_{n=n_{1}+1}^{n_{2}-1}
      \lambda_{n}\fns(x)+\lambda_{n_2}\phi_{n_2}^{\brn}(x)}\quad\text{since $\supp(x)\cap\brn\subset\supp(f)$}
    \\
               &=\abs{\sum_{n=n_{1}}^{n_{2}}\lambda_{n}\fns(\sum_{j}\phi_{j}^{\brn}(x)x_{j}^{\brn})}
                 =\abs{\sum_{n=n_{1}}^{n_{2}}\lambda_{n}\fns(P_{\brn}(x))}\leq\norm{P_{\brn}(x)}.
  \end{align*}
\end{proof}
The following lemma, which is a variation of I. Amemiya and T.Ito  Lemma \cite{AI},
is analogous to  Lemma~\ref{jtramsey}  for  the space $\jtp$.

\begin{lemma}\label{ramseyupperS2}
  Let $\xn$ be a normalized block sequence such that
  $\lim_{n}\norm{P_{\brn}x_{n}}=0$ for every branch $\brn\in\brT$.
  Then, for every $\e>0$ and $n_{0}\in\N$ there exists $L\in[\N]$ such that
  for every segment $\s$ with $\min\s\leq n_{0}$  and every $f\in W(\s)$ with $\minsupp(f)\leq n_{0}$ it holds that
  $\abs{f(x_{n})}<\e$ for all but one  $n\in L$.
\end{lemma}
\begin{proof}
Assume that the result does not hold. Then by  Ramsey theorem
we get $L\in [\N]$ such that for every $m<n\in L$ there exist
a segment $\s_{m,n}$   with $\min\s_{m,n}\leq n_{0}$ and an $f_{\s_{m,n}}\in W(\s_{m,n})$
with
\begin{equation}
  \label{eq:30}
\minsupp f_{\s_{m,n}}\leq n_{0},\quad
\abss{f_{\s_{m,n}}(x_{m})}\geq\e\,\,\text{and}\,\,
\abss{f_{\s_{m,n}}(x_{n})}\geq \e.
\end{equation}
Let $p_{n}=\min\supp x_{n}$, $n\in\N$.
For every segment $\s_{m,n}$ we denote by $\bar{\s}_{m,n}$ the unique
initial segment determined by  $\s_{m,n}$.

Using \eqref{eq:30} and  Remark~\ref{numberds} it  follows  similarly to
Claim 1 of Lemma~\ref{jtramsey} that for  every $n\in L$ we have that
\begin{center}$\#\{\bar{s}_{m,n|[1,p_{n})}: m<n\}=d_{n}\leq d= \floor{2/\e^{2}}+1$.
\end{center}
Enumerate for every $n\in L$ the set of the different segments as
$\{\s_{n,1},\dots,\s_{n,d_{n}}\}$.
Set for $i\leq d_{n}$,
\[
L_{i}^{n}=\{m<n:   \abss{f_{\s_{n,i}}(x_{m})}\geq\e\}\quad
\text{
  and $L_{i}^{n}=\emptyset$ for every $d_{n}<i\leq d$}.
\]
Passing to a further infinite subset $M $of $L$ we may assume that
$L_{i}^{n}\xrightarrow[n\in M]{} L_{i}$ for every  $i\leq d$.  Note that for every $n\in L$,
$\{m\in L:m<n\}=\cup_{i=1}^{d}L_{i}^{n}$ and  hence $L=\cup_{i=1}^{d}L_{i}$.
Since $L$ is an infinite subset of $\N$ it follows that there exists
$i_{0}\leq d$ such that $L_{i_{0}}$ is infinite.
Moreover for every $n\in M$ and every $m<n$, $m\in L_{i_{0}}$, 
there exists 
$f_{\s_{n,i_{0}}}\in W(\s_{n,i_{0}})$  with $\minsupp(f_{\s_{n,i_{0}}})\leq n_{0}$
such that  $\abss{f_{\s_{n,i_{0}}}(x_{m})}\geq\e$.
In particular  if  $n_{0}<\minsupp(x_{m})$ and $m<n$
it follows that
$\supp(x_{m})\cap\s_{n,i_{0}}\subset\rang (f_{\s_{n,i_{0}}})
$. Lemma~\ref{fprojection2}  yields that
\begin{equation}
  \label{eq:32}
    \norm{P_{\brn}(x_{m})}\geq\e\,\,\text{for every branch
$\brn$ such that}\,\, \supp(x_{m})\cap \s_{n,i_{0}}\subset\brn.
\end{equation}
Passing to further subset $P$ of $M$ we get 
that  $(\s_{n,i_{0}})_{n\in P}$ converges pointwise  to  an initial segment
or a  branch.  Similarly to Claim 2 of Lemma~\ref{jtramsey} 
we get  that   $(s_{n,i_{0}})_{n\in P}$ converges to a branch $\brn$.
From \eqref{eq:32}  we get  that   $\norm{P_{\brn}(x_{m})}\geq\e$ for infinitely many $m$ which yields  a contradiction.
\end{proof}

\begin{lemma}\label{lemmaupperS2}
Let $\e>0$   and $\xn$ be a normalized block sequence such
that $\lim_{n}\norm{P_{\brn}(x_{n})}=0$ for every branch $\brn\in\brT$.
Let $q_{n}=\max\supp (x_{n}), n\in\N$. Then there exist a strictly
increasing sequence $(n_{k})_{k\in\N}$ of positive integers and a
strictly decreasing sequence $(\e_{n})_{n}$ of positive numbers such
that
\begin{enumerate}
\item for every $k\in\N$, for every segment $\s$  and $f\in W(s)$ such that  $\minsupp(f)\leq q_{n_{k}}$
  there exists  at most one $k^{\prime}>k$ such that $\abss{f(x_{n_{k^{\prime}}})}\geq\e_{k}$ and
  \item $\sum_{k=1}^{\infty}q_{n_{k}}\sum_{i=k}^{\infty}(i+1)\e_{i}<\e$.
\end{enumerate}
\end{lemma}
The proof of the lemma is identical to the proof of Lemma~\ref{l13}.
\begin{proposition}\label{propupperS2}
 Let  $\e>0$, $\e_{n}\searrow 0$ and $\sxn$  be a normalized block
 sequence.  Let $q_{n}=\max\supp x_{n}$,$n\in \N$ and assume that
 \begin{enumerate}
\item[a)] for every $n\in\N$ and every segment $\s$ and $f\in W(s)$ with $\min f\leq q_{n}$
  there exists  at most one $p>n$ such that $\abss{f(x_{p})}\geq\e_{n}$,
  \item[b)] $\sum_{n=1}^{\infty}q_{n}\sum_{i=n}^{\infty}(i+1)\e_{i}<\e$.
\end{enumerate}
Then,
\begin{enumerate}[label=\roman*)]
\item for every $n\in\N$,  every choice of scalars $a_{1},\dots,a_{n}$, setting $p_{i}=\minsupp (x_{i})$, $1\leq i\leq n$
it holds that,
 \[
\normg{\sum_{i=1}^{n}a_{i}x_{i}}\leq(\sqrt{6}+\e)\norm{\sum_{i=1}^{n}a_{i}e_{p_{i}}}_{S^{(2)}}.
\]
\item For every branch $\brn$ and every finite subset $I$ of $\N$ we have $\norm{\sum_{n\in I}\pm x_{n}}_{\brn}\leq 3$.
\end{enumerate}
\end{proposition}
\begin{proof}
 Let $ \brn_{1},\dots, \brn_{k}$ be branches
and  $f_{j}\in W(\brn_{j})$ for $j=1,\dots,k$ be disjointly supported
functionals with $k\leq\min\supp f_{j}$ for every $j$.
For every  $1\leq j\leq k$ we denote by
$i_{j,1}$ the  unique $i\leq n$ such
that  $q_{i-1}<\minsupp f_{j}\leq q_{i}$.
By a)  we get
\begin{equation*}
\mbox{for every $j$ there exists at most  one  $i>i_{j,1}$
  such that $\abs{f_{j}(x_{i})}\geq\e_{i_{j},1}$.}
\end{equation*}
Denote this $i$, if it exists, by
$i_{j,2}$.   
We set $G_{i}=\{j\leq k: i_{j,1}=i\}$. It follows that  $G_{i}$ are
pairwise disjoint and $\{1,\dots,k\}=\cup_{i=1}^{n}G_{i}$.
We also set $J_{i}=G_{i}\cup\{j\leq k:  i_{j,2}=i\}$. It
follows
\begin{align*}
  \abs{(\sum_{j=1}^{k}\beta_{j}f_{j})(\sum_{i=1}^{n}a_{i}x_{i})}&=
                                                        \abs{\sum_{j=1}^{k}\beta_{j}f_{j}(\sum_{i\geq i_{j,1}}^{n}a_{i}x_{i})}
  \\
  &\leq \sum_{j=1}^{k}\abs{\beta_{j}}\left(\abs{a_{i_{j,1}}}\abs{f_{j}(x_{i_{j,1}})}
    +\abs{a_{i_{j,2}}}\abs{f_{j}(x_{i_{j,2}})}
+\sum_{i_{j,1}<i\ne i_{j,2}}\abs{a_{i}f_{j}(x_{i})}
\right)
  \\
  &=\sum_{i:J_{i}\ne\emptyset}\abs{a_{i}}\left(\sum_{j\in
    J_{i}}\abs{b_{j}f_{j}(x_{i})}\right)+
\sum_{i=1}^{n} \sum_{j\in G_{i}}\abs{b_{j}}\sum_{i<l\ne i_{j,2}}\abs{a_{l}}\abs{f_{j}(x_{l})}.
\end{align*}
Note that
\begin{equation}
  \label{eq:16a}
  \sum_{j\in   J_{i}}\abs{b_{j}f_{j}(x_{i})}\le\left(\sum_{j\in J_{i}}b_{j}^{2}\right)^{1/2}\norm{x_{i}}.
\end{equation}
Also  every $j$ can appear in at most two $J_{i}$. Moreover this implies
 that $\{\minsupp x_{i}: J_{i}\ne\emptyset\}$ is the union of at most
 three  Schreier sets. Namely if $i_{0}=\min\{i\leq n: G_{i}\neq\emptyset\}$ then
 \[
\{\minsupp x_{i}: J_{i}\ne\emptyset\}=\{\minsupp x_{i_{0}}\}\cup\{\minsupp x_{i}: i_{0}<i, J_{i}\ne\emptyset\}.
 \]
The later set contains at most $2k$ elements after $k$ and hence is
the union of at most two Schreier sets. Let
\[
\{\minsupp x_{i}: J_{i}\ne\emptyset\}=A_{1}\cup A_{2}\cup A_{3}, \,\, A_{1}<A_{2}<A_{3}\in\schreier[1]
\]
with the latter  possible empty.  Note that setting $p_{i}=\minsupp x_{i}$ for every $i$, for every $q=1,2,3$ it holds,
\begin{equation}
  \label{eq:31}
 \sum_{i\in A_{q}}a_{i}^{2}\leq\norm{\sum_{i\in A_{q}}a_{i}e_{p_{i}}}_{S^{(2)}}^{2}\leq\norm{\sum_{i=1}^{n}a_{i}e_{p_{i}}}_{S^{(2)}}^{2}
\end{equation}
From  H\"older inequality, using that every $j$ can appear in at most two $J_{i}$, \eqref{eq:16a} and \eqref{eq:31} we get,
\begin{align}
  \label{eq:15a}
\sum_{i:J_{i}\ne\emptyset}\abs{a_{i}}\left(\sum_{j\in
    J_{i}}\abs{b_{j}f_{j}(x_{i})}\right)
&\leq
\left(\sum_{i:J_{i}\ne\emptyset}\abs{a_{i}}^{2}\right)^{1/2}\left(
  \sum_{i:J_{i}\ne\emptyset}\left(\sum_{j\in
    J_{i}}\abs{b_{j}}^{2}\right)\norm{x_{i}}^{2}\right)^{1/2}\\
&\leq\sqrt{2}\left(\sum_{i:J_{i}\ne\emptyset}\abs{a_{i}}^{2}\right)^{1/2}
=\sqrt{2}
 \left(\sum_{q=1}^{3}\sum_{i\in A_{q}}\abs{a_{i}}^{2}
                                                               \right)^{1/2}
 \notag \\
  &\leq\sqrt{6}\norm{\sum_{i=1}^{n}a_{i}e_{p_{i}}}_{S^{(2)}}.
    \notag
\end{align}
It remains to find un upper bound for  the term
$\sum_{j\in
  G_{i}}\abs{b_{j}}\sum_{i<l\ne i_{j,2}}\abs{a_{l}}\abs{f_{j}(x_{l})}$.
  Let us observe that for every $n\in\N$, every segment $\s$
and $f\in W(s)$  with  $\min\s\leqq q_{n}$,
setting $E_{k}=\{p>n: \e_{k}\leq \abss{f(x_{p})}<\e_{k-1}\}$, $k>n$, the following holds due to a)
  \begin{enumerate}
  \item[(c)]  $\# E_{k} 
  \leq k$    for every $k>n$.
  \end{enumerate}
Using $(c)$ we get,
\begin{align*}
\sum_{i<l\ne i_{j,2}}\abs{a_{l}}\abs{f_{j}(x_{l})}\leq
  \sum_{k=i+1}^{\infty}\max_{l\in E_{k}}\abs{a_{l}}\sum_{l\in E_{k}}\abss{f_{j}(x_{l})}\leq\sum_{k=i}^{+\infty}(k+1)\e_{k}.
\end{align*}
Combining the above inequality with (b) we get
\begin{equation}
  \label{eq:errors}
  \sum_{i=1}^{\infty}\sum_{j\in
    G_{i}}\abss{b_{j}}\sum_{i< l\ne i_{j,2}}\abs{a_{l}}\abs{f_{j}(x_{l})}
\leq\sum_{i=1}^{\infty}q_{i}\sum_{k=i}^{+\infty}(k+1)\e_{k}<\e.
\end{equation}
Combining \eqref{eq:errors}  with  \eqref{eq:15a} we have the result.

ii)   Let $\brn\in \brT$  and $f\in W(\brn)$.
Following the notation of (i)
we denote by $i_{1}$ the  unique $i\in  I$ such
that  $q_{i-1}<\minsupp (f)\leq q_{i}$
and  by 
$i_{2}$  the unique 
$i>i_{1}$
  such that $\abss{f(x_{i})}\geq\e_{i_{1}}$, if such an $i$ exists.

From \eqref{eq:errors}  for $\# G_{i}=1$ we get  that
\[
\abss{f(\sum_{i_{1}<i\ne i_{2}}\e_{i}x_{i})}\leq 1.
\]
It follows that
\[
  \abss{f(\sum_{i\in I}\e_{i}x_{i})}\leq\abs{f(x_{i_{1}})}+\abs{f(x_{i_{2}})}+
\abss{f(\sum_{i_{1}<i\ne i_{2}}\e_{i}x_{i})}\leq 3.  
\]
\end{proof}

\begin{proof}[Proof of Theorem~\ref{upperS2}]
  The proof of the Theorem follows combining  Lemma~ \ref{lemmaupperS2} and Proposition~\ref{propupperS2} choosing  $\e<1/10$.
\end{proof}
\begin{proposition}\label{c0}
 Every normalized block sequence $\sxn$ such that $P_{\brn}(x_{n})\xrightarrow[n]{} 0$ for every branch $\brn$, has a further block sequence that is $2$-equivalent to the unit vector basis of $c_{0}$.
\end{proposition}
\begin{proof}
  Let $\sxn$   be a normalized block sequence  that  $\lim_{n}\norm{P_{\brn}(x_{n})}=0$ for every branch $\brn$.  Using  Lemmas~\ref{ramseyupperS2},\ref{lemmaupperS2} and passing to a subsequence we may assume that item (ii) of  Proposition~\ref{propupperS2} holds i.e. for every  finite subset  $I$ of $\N$ and every choice of signs $\e_{n}, n\in I$,
  \[
\norm{\sum_{n\in I}\e_nx_{n}}_{\brT}=\sup\{\norm{\sum_{n\in I}\e_{n}x_{n}}_{\brn}:\brn\in\brT \}\leq 3.
  \]
If $\sxn$ is not equivalent to the unit vector basis of $c_{0}$ we get a block sequence $(y_{n})_{n}$  of $\sxn$
such that  $y_{n}=\sum_{i\in I_{n}}\e_{i}x_{i}$ and $\norm{y_{n}}\to+\infty$.   Set $z_{n}=y_{n}/\norm{y_{n}}$. Note that  since $\norm{y_{n}}_{\brT}\leq 3$  it holds that $\norm{z_{n}}_{\brT}\to 0$. Passing to a subsequence we may assume that  $\sum_{n=k+1}^{+\infty}\norm{z_{n}}_{\brT}\leq1/\max\supp(z_{k})$ for every $k$.
Let  $p\in\N$ and $a_{n}\in [-1,1]$, $n\leq p$. We show that  $\norm{\sum_{n=1}^{p}a_{n}z_{n}}\leq 2$.

Let $f=\sum_{j=1}^{d}\lambda_{j}f_{j}\in W$ with $f_{j}\in W(\brn_{r_{j}})$. Set 
$n_{0}=\min\{n: \rang(z_{n})\cap\rang(f_{j})\ne\emptyset\,\text{for some }\,j \}$ and notice that
$d\leq\max\supp(z_{n_{0}})$. It follows 
\[
  \abss{f_{j}(\sum_{n>n_{0}}a_{n}z_{n})}<1/\max\supp(z_{n_{0}})
\Rightarrow\abss{f(\sum_{n>n_{0}}a_{n}z_{n})}<d/\max\supp(z_{n_{0}})\leq1.
\]
So we conclude
$\abss{f(\sum_{n=1}^{p}a_{n}z_{n})}<\abss{f(a_{n_{0}}z_{n_{0}})}+\abss{f(\sum_{n>n_{0}}a_{n}z_{n})}\leq
2$ which yields the desired result.
\end{proof}
\begin{corollary}
  The space $\jtg$ is not reflexive. In particular, for every infinite subset $N$ of $\N$ the subspace $X_{N}$
  of $\jtg$ generated  by the subsequence
  $(e_{n})_{n\in N}$ of the basis $(e_{n})_{n\in\N}$ contains a subspace isomorphic to $c_0$.
\end{corollary}
Indeed, for every  sequence $(n_k)_{k}$   we have that $P_{\brn}(e_{n_{k}})\to 0$ for every $\brn$ and from the previous proposition we have the result.

The   last  result of the section is the following proposition.
\begin{proposition}\label{s2averages}
  Let $\sxn$ be a normalized block sequence  such that
  $\norm{P_{\brn}(x_{n})}\to 0$ for every branch $\brn\in\brT$. Then there
  exists a subsequence $(x_{n})_{n\in M}$ of $\sxn$ such that  for every
  block sequence $(y_{n})_{n\in N}$ of  $(x_{n})_{n\in M}$, where each
  $y_{n}$ is an $(2,n)$-average of  $(x_{n})_{n\in M}$, it holds that $\lim_{n}\norm{y_{n}}=0$.
\end{proposition}
\begin{proof}
 From Theorem~\ref{upperS2}  we get a subsequence $(x_{n})_{n\in M}$ of
 $\sxn$ satisfying upper $\schreier[]^{(2)}$-estimates with constant $3$.
 Let  $y_{n}=\sum_{k\in F_{n}}a_{k}x_{k}$  be an
  $(2,n)$- average of  $(x_{n})_{n\in M}$. From
  the properties of the repeated averages we get  that  for every
  $G\in\schreier[1]$ it holds $\sum_{k\in
    G}a_{k}^{2}\leq 3/p_{n}$ where $p_{n}=\minsupp (x_{\min F_{n}})$.
  It follows that
  \[
    \norm{y_{n}}=\norm{\sum_{k\in F_{n}}a_{k}x_{k}}
    \leq\sup\{    (\sum_{k\in G}a_{k}^{2})^{1/2}: G\in \schreier\}
    \leq  (3/p_{n})^{1/2}.
\]
The above inequality yields the result.
\end{proof}
\section{The $\ell_{2}$-vector of a block sequence}
In this  section we present the definition of the $\ell_{2}$-vectors, which are analogous to $\ell_{p}$-vectors that appeared  in the study of the space  $\jtp$.
\begin{definition}  Let $(x_{n})_{n\in\N}$ be a normalized block sequence in $\jtg$ . A decreasing
sequence $(c_{i})_{i\in N}$, where $N$ is either an initial finite interval of $\N$ or the set of
the natural numbers, is said to be the  $\ell_2$ -vector of the sequence $\xn$ with
associated set of
branches $\mc{B}=(\brn_{i})_{i\in N}$ if
\begin{equation*}
  \lim_{n}\norm{P_{\brn_{i}}x_{n}}=c_{i}>0\,\,\text{for every $i\in N$ and}
  \lim_{n}\norm{P_{\brn}x_{n}}=0\,\,\text{for every $\brn\ne\brn_{i}$, $i\in N$}.  
\end{equation*}
If  $\lim_{n}\norm{P_{\brn}x_{n}}=0$ for every branch $\brn$,  the
$\ell_{2}$-vector of $\xn$ is the zero one and 
and the associated set of branches $\mc{B}$ is the empty one.
\end{definition}
\begin{proposition}\label{pl2vector}
Let $\sxn$   be a normalized block sequence in $\jtg$. Then,  $\sxn$
either has a subsequence admitting an upper $S^{(2)}$-estimate or it has a
subsequence accepting  a non-zero $\ell_{2}$-vector.
\end{proposition}
For the proof we need the following lemma which  is a variation of Lemma~\ref{lslemma}.
\begin{lemma}\label{lslemma2}
  Let $\sxn$ be a normalized block sequence in $\jtg$
  such that there exists a branch $\brn\in\brT$ with $\limsup_{n}\norm{P_{\brn}(x_{n})}=c>0$.
Then there   exist  a subsequence $(x_{n})_{n\in L}$ of $\sxn$, an at
most  countable set of distinct branches $(\brn_{i})_{i\in N}$
and  a decreasing sequence $(c_{i})_{i\in N}\in B_{\ell_{2}}$ 
such that
\begin{equation}
  \label{eq:33}
  \lim_{n\in L}\norm{P_{\brn_{i}}x_{n}}=c_{i}>0\,\text{for all  $i\in N$
  and }\,
\lim_{n\in L}\norm{P_{\brn}x_{n}}=0\,\,\text{for all}\, \brn\ne\brn_{i}, i\in N.
\end{equation}
\end{lemma}
\begin{proof}
  The proof is similar to the proof of Lemma~\ref{lslemma}. We
start with the following observation which is critical for our result.
Using Lemma~\ref{sumprojections} it follows
 that for every  choice if distinct branches $\brn_{1},\dots,\brn_{k}$
 we have that
 \begin{equation}
   \label{eq:19a}
\limsup_{n}\left(\sum_{i=1}^{k}\norm{P_{\brn_{i}}x_{n}}^{2}\right)^{1/2}\leq\sup_{n}\norm{x_{n}}=1.
\end{equation}
Combining \eqref{eq:19a} and Remark~\ref{numberds} we get that for every $\e>0$ and
every subsequence $(x_{n})_{n\in L }$ there exist at most
$k\leq d(\e)=\floor{2/\e^{2}}+1$ distinct branches $\brn_{1},\dots,\brn_{k}$
such that
$\lim_{n\in L}\norm{P_{\brn_{i}}x_{n}}=c_{i}\ge\e$ and
$\limsup_{n\in L}\norm{P_{\brn}x_{n}}<\e$ for all $\brn\ne\brn_{i}$.
Let $\e_{j}\searrow 0$. Continuing as  in Lemma~\ref{lslemma},   we get 
a subsequence $(x_{n})_{n\in L}$ of $\sxn$, an at most countable  set of branches
$(\brn_{i})_{i\in N}$ 
and a decreasing sequence $(c_{i})_{i \in N}$ such that \eqref{eq:33} holds.

 We show now that $(c_{i})_{i\in N}\in B_{\ell_{2}}$. On the contrary assume
 that for some finite subset $F$ of $\N$ it holds  $\sum_{i\in
   F}c_{i}^{2}>1$.
 Chose $\rho\in\N$, $\rho>\# F$, such that the  final segments
 $\brn_{i,>\rho}$, $i\in F$ are  incomparable
 and  $\e>0$  such that $(1-\e)(\sum_{i\in F}c_{i}^{2})^{1/2}>1$. Let $n\in\N$ be such
 that
 $\min\supp x_{n}>\max\{\min\brn_{i,>\rho}: i\in F\}$ and
$\norm{P_{\brn_{i}}(x_{n})}=\norm{P_{\brn_{i,>\rho}}(x_{n})}>(1-\e)c_{i}$ for every
$i\in F$. Setting $\mc{B}_{F}=\{\brn_{i}:i\in F\}$  and using Lemma~\ref{sumprojections} we get
\begin{align*}
  \normg{x_{n}}\geq\norm{P_{\mc{B}_{F,>\rho}}(x_{n})}=
\left(\sum_{i\in F}\norm{P_{\brn_{i,>\rho}}x_{n}}^{2}\right)^{1/2}
\geq
(1-\e)\left(\sum_{i\in F}c_{i}^{2}\right)^{1/2}>1
\end{align*}
a contradiction.
\end{proof}
\begin{proof}[Proof of Proposition~\ref{pl2vector}]
If  $\sxn$ has a subsequence  $(x_{n})_{n\in L}$ with  $\ell_{2}$ vector the zero one then Theorem~\ref{upperS2} yields the result. Otherwise Lemma~\ref{lslemma2} yields a subsequence accepting non-zero $\ell_{2}$-vector.
\end{proof}
For the rest of  Part II, we will assume that a non-zero
$\ell_{2}$-vector  is an infinitely supported vector of $\ell_{2}$.  The
cases that  the $\ell_{2}$-vector  is  finitely supported  are treated in
similar manner and in most cases  the proofs are simpler.
\begin{lemma}\label{lowerl2}
  Let $(x_{n})_{n}$ be  bounded block sequence admitting a non-zero $\ell_{2}$-vector
$(c_{i})_{i\in\N}$. Then  $\liminf_{n}\norm{x_{n}}\geq\norm{(c_{i})_{i\in\N}}_{\ell_{2}}$.
\end{lemma}
\begin{proof}
Let $(\brn_{i})_{i\in\N}$ be the set of the associated  branches, 
$I$ be an initial interval of $\N$. 
Let  $l\in\N$  be such that   the final segments $\brn_{i,>l}$, $i\in I$, are
incomparable  and $\#I<l$.  Let $\e>0$. Using the definition of the $\ell_{2}$-vector
we choose  $n_{0}\in\N$  such that
$\norm{P_{\brn_{i}}(x_{n})}\geq(1-\e)c_{i}$ for every $i\in I$ and every
$n\geq n_{0}$.  Then for every $x_{n}$ such that $\minsupp
(x_{n})\geq\max\{\min\brn_{i,>l}:i\in I\}$,
Lemma~\ref{sumprojections} yields
\[
\norm{x_{n}}\geq\norm{P_{\mc{B}_{I,>l}}(x_{n})}=\left(\sum_{i\in I}\norm{P_{\brn_{i,>l}}
    (x_{n})}^{2}\right)^{1/2}\geq(1-\e)\left(\sum_{i\in I}c_{i}^{2}\right)^{1/2}.
\]
Since  the above inequality holds for any $\e>0$ and every initial
segment of $\N$  we get that $\liminf_{n}\norm{x_{n}}\geq\norm{(c_{i})_{i\in\N}}_{\ell_{2}}$.
\end{proof}
\section{Upper $\ell_{2}$-estimates in $\jtg$.}
In this section we shall  prove that every normalized block sequence
contains a subsequence which admits an upper $\ell_{2}$-estimate with constant $9$. 
We start with some lemmas concerning the  $\ell_{2}$-vectors.
\begin{lemma}\label{stabilization2}
  Let  $(x_{n})_{n}$ be a normalized block sequence with a non-zero
  $\ell_{2}$-vector  $(c_{i})_{i\in \N}$
  and let $(\brn_{i})_{i\in \N}$ be the associated set of branches.
  Let $(I_{k})_{k}$ be a strictly  increasing sequence of initial
  intervals of $\N$ and $\e_{k}\searrow 0$.
  Then, there exist  $n_{k}\nearrow +\infty$
  and $(l_{k})_{k}\subset\N$
 such that  
 \begin{enumerate}
 \item[(i)]  $    (1-\e_{k})c_{i}\leq  \norm{\Pbrn[\brn_{i}]x_{n}}\leq (1+\e_{k})c_{i}$ for every
   $i\in I_{k}$ and $n\geq n_{k}$,
 \item[(ii)]  the final  segments $\brn_{i,>l_{k}}, i\in I_{k}$ are
    incomparable, $ \# I_{k}<l_{k}$, and  \\
$\max\{\min \brn_{i,>l_{k}}:i\in I_{k}\}<\minsupp(x_{n_{k}})$.
 \end{enumerate}
 In particular, 
 setting $\bar{x}_{n_{k}}=P_{\mc{B}_{I_{k}}}(x_{n_{k}})$
 the
 following  holds:
 \begin{equation*}
   \lim_{k}\norm{\bar{x}_{n_{k}}}=\norm{(c_{i})_{i\in \N}}_{\ell_{2}}\,\,\text{and}\,\,
  \lim_{k}\norm{\Pbrn(x_{n_{k}}-\bar{x}_{n_{k}})}=0\,\, \text{for every branch}\, \brn.
\end{equation*}
Moreover the   sequence $(x_{n_{k}}-\bar{x}_{n_{k}})_{k}$ 
admits an upper $\schreier^{(2)}$-estimate with constant $6$.
\end{lemma}
\begin{proof}
Using the definition of the $\ell_{2}$-vector and that $\sxn$ is a block sequence,  we choose inductively $n_{k}\nearrow+\infty$ and
$(l_{k})_{k}\subset\N$
such that  (i) and (ii) holds.
Let $\bar{x}_{n_{k}}=P_{\mc{B}_{I_{k}}}(x_{n_{k}})
$.
From (ii)  and  Lemma~\ref{supperl2} we get
 \begin{equation*}
 \norm{\bar{x}_{n_{k}}}=\left(\sum_{i\in
       I_{k}}\norm{P_{\brn_{i}}(x_{n_{k}})}^{2}\right)^{1/2}.
\end{equation*}
Using  (i)  it follows that
\begin{equation*}
(1-\e_{k})\left(\sum_{i\in I_{k}}c_{i}^{2}\right)^{1/2}
\leq\norm{\bar{x}_{n_k}}\leq (1+\e_{k})\left(\sum_{i\in
  I_{k}}c_{i}^{2}\right)^{1/2}.
\end{equation*}
The above inequalities, using that  $\#I_{k}\nearrow\infty$, yield that $\norm{\bar{x}_{n_k}}\xrightarrow[k]{}\norm{(c_{i})_{i\in\N}}_{\ell_{2}}$.

Note also that $(\bar{x}_{n_{k}})_{k}$ has as $\ell_{2}$-vector the
$(c_{i})_{i\in \N}$ with the same set of associated brances.
Indeed from (ii) it  follows that for every
$\brn_{i}$,
\[\lim_{k}\norm{P_{{\brn_{i}}}(\bar{x}_{n_{k}})}
=\lim_{k}\norm{P_{\brn_{i}}(x_{n_{k}})}=c_{i}.\]
If  $\brn\ne\brn_{i}$ for every $i$,
from (ii) we get that for every $k$ there exists 
at most one  
$i_{k}\in I_{k}$ such that  $\brn \cap\brn_{i_{k}, >l_{k}}\ne\emptyset$.  If such $i_{k}$ exists, 
\[
  P_{\brn}(\bar{x}_{n_{k}})=  P_{\brn}P_{\mc{B}_{I_{k}}}(x_{n_{k}})=
  P_{\brn}P_{\brn_{i_{k}}}(x_{n_{k}}).
\]
Otherwise  $P_{\brn}(\bar{x}_{n_{k}})=0$. It follows that
\begin{equation}
  \label{eq:26}
\normg{P_{\brn}(\bar{x}_{n_{k}})}\leq\normg{P_{\brn_{i_{k}}}(\bar{x}_{n_{k}})}
\leq(1+\e_{k})c_{i_{k}}\leq 2c_{i_{k}}. 
\end{equation}
Since  for every $i\in\N$  $\brn\ne \brn_{i}$,  we have that
for every $i\in \N$, 
\begin{equation}\label{star}
 \text{
there exist
at most finitely  many  $k$ such that
$i\in I_{k}$ and $ \brn_{i,>l_{k}}\cap \brn\ne\emptyset$}.
\end{equation}
Combining   $\eqref{star}$ and \eqref{eq:26} it follows that $
\lim_{k}\norm{P_{\brn}\bar{x}_{n_{k}}}=0$ for every $\brn\ne\brn_{i}$, $i\in \N$.

We show now that  $\lim_{k}\norm{P_{\brn}(x_{n_{k}}-\bar{x}_{n_{k}})}=0$.
If $\brn=\brn_{i}$ for some $i$, then for every $k>i$, from  (ii) we get 
\[
P_{\brn_{i}}(x_{n_{k}}-\bar{x}_{n_{k}})=P_{\brn_{i}}(x_{n_{k}})-P_{\brn_{i}}(x_{n_{k}})=0.
\]
If $\brn\ne\brn_{i}$ for every $i\in \N$, then
\begin{align*} P_{\brn}(x_{n_{k}}-\bar{x}_{n_{k}})=P_{\brn}(x_{n_{k}})-P_{\brn}P_{\brn_{i_{k}}}(x_{n_{k}})\Rightarrow
  \\
  \normg{P_{\brn}(x_{n_{k}}-\bar{x}_{n_{k}})}\leq
  \normg{P_{\brn}(x_{n_{k}})}+2c_{i_{k}}\xrightarrow[k]{}0.
  \end{align*}
The moreover  part follows from Theorem~\ref{upperS2} using that $\norm{x_{n_{k}}-\bar{x}_{n_{k}}}\leq2$.
\end{proof}
\begin{definition}
 Let $\sxn$ be a normalized block sequence with a non-zero $\ell_{2}$-vector $(c_{i})_{i\in \N}$ and let $(\brn_{i})_{i\in \N}$ be the associated
 set of  branches. We say that $\sxn$ is skipped block with respect to the
 branches $(\brn_{i})_{i\in\N}$  if for every $i\in\N$ and every $n\geq i$   there exists
 $j_{n}\in\N$ such that  $\rang(x_{n})<\rang(\phi_{j_{n}}^{\brn_{i}})<\rang(x_{n+1})$.
\end{definition}
\begin{lemma}\label{skipped}
	 Let $\sxn$ be a normalized block sequence with a non-zero $\ell_{2}$
	vector $(c_{i})_{i\in \N}$ and let $(\brn_{i})_{i\in \N}$ be the associated
	set of  branches. Then, there  exists a subsequence $(x_{n_k})_{k\in\N}$ such that
	
	a) $(x_{n_k})_{k\in\N}$ is skipped block with respect to the branches $(\brn_{i})_{i\in\N}$.
	
	b)  for every  branch $\brn_{i}$, $i\in\N$, we have that
	\begin{equation*}
		\norm{\sum_{k=i}^{\infty}a_{k}P_{\brn_i}(x_{n_k})}=
		\left(\sum_{k=i}^{\infty}a_{k}^{2}\norm{P_{\brn_{i}}(x_{n_{k}})}^{2}\right)^{1/2}.
	\end{equation*}
\end{lemma}
\begin{proof}
a) Using that $\sxn$ is  block sequence it readily follows  that  for every branch $\brn_{i}$ and every subsequence $(x_{n_{k}})_{k\in\N}$ 
we can choose a further subsequence such that  for every $k$, $\rang(x_{n_{k}})<\ran \phi_{l_{k}}^{\brn_{i}}<\rang(x_{n_{k+1}})$  for some $l_{k}$. 
Using a diagonal argument it is  easily follows
that we can choose  a subsequence $(x_{n_k})_{k\in\N}$  which is skipped block with respect to the branches $(\brn_{i})_{i\in\N}$.

b)  For every branch $\brn_{i}$, using that $(x_{n_k})_{k=i}^{\infty}$ is 
skipped block with respect to the $\brn_{i}$,  we get that  
$P_{\brn_{i}}(x_{n_{k}}))_{k=i}^{\infty}$ is also skipped block w.r the 
$(x_{j}^{\brn_{i}})_{j\in\N}$.  Since  by Lemma~\ref{upperl2}, $(x_{j}^{\brn_{i}})_{j\in\N}$ is 
1-equivalent to the unit vector basis of $\ell_{2}$ 
it follows that
\[\norm{\sum_{k=i}^{\infty}a_{k}P_{\brn_i}(x_{n_k})}=		
\left(\sum_{k=i}^{\infty}a_{k}^{2}\norm{P_{\brn_{i}}(x_{n_{k}})}^{2}\right)^{1/2
}.\]
\end{proof}

\begin{lemma}\label{llemma}
Let  $(x_{n})_{n}$ be a normalized block sequence with a non-zero
  $\ell_{2}$-vector  $(c_{i})_{i\in \N}$
  and $(\brn_{i})_{i\in \N}$ be the associated set of branches. Then
there  exist a strictly increasing sequence $(I_{k})_{k\in\N}$ of initial intervals of
$\N$ and a subsequence $(x_{n_{k}})_{k\in\N}$ such that setting
$\bar{x}_{n_{k}}=P_{\mc{B}_{I_{k}}}(x_{n_{k}})$  the following holds:

a)  for every $\e>0$ and 
for every block sequence $(\bar{w}_{n})_{n\in\N}$ of
$(\bar{x}_{n_{k}})_{k}$ such that $\bar{w}_{n}=\sum_{k\in
  F_{n}}a_{k}\bar{x}_{n_{k}}$ is an $\ell_{2}$-convex combination,  there exists
$n_{0}$  such that for every initial interval  $I$ of $\N$ containing
$I_{n_{0}}$ there exists  $n(I)\in\N$ such that for all $n>n(I)$,
  \begin{equation*}
 \norm{\bar{w}_{n}-P_{\mc{B}_{I}}(\bar{w}_{n})}<\e.
\end{equation*}

b)  the sequence $(\bxnk)_{k}$ admits upper $\ell_{2}$-estimate with constant 3.
\end{lemma}
\begin{proof}
  From  Lemma~\ref{lpsplit} we get a strictly increasing sequence
  $(I_{n})_{n\in\N}$ of initial intervals of $\N$ such that
  for all
  $\e>0$ there exists
$n_{0}\in\N$ such that
 \begin{equation}
  \label{eq:008}\sum_{n\in\N\setminus I_{n_{0}}}(\sum_{i\in I_{n}\setminus
    I_{n-1}}c_{i}^{2})^{1/2}<\e/2.
\end{equation}
In particular we may assume that$\sum_{n>1}(\sum_{i\in I_{n}\setminus
    I_{n-1}}c_{i}^{2})^{1/2}<1$.

 Let $\e_{k}\searrow0$, $\e_{1}<1$.  Using the definition of the $\ell_{2}$-vector and that $\sxn$ is a block sequence,  we choose inductively $n_{k},l_{k}\nearrow+\infty$ 
such that
 \begin{enumerate}
 \item[(i)]  $    (1-\e_{k})c_{i}\leq  \norm{\Pbrn[\brn_{i}]x_{n}}\leq (1+\e_{k})c_{i}$ for every
   $i\in I_{k}$ and $n\geq n_{k}$.
 \item[(ii)]  The final  segments $\brn_{i,>l_{k}}, i\in I_{k}$ are
   incomparable, $\#I_{k}<l_{k}$   and \\
   $\minsupp (x_{n_{k}})>\max\{\min
   \brn_{i,>l_{k}}:i\in I_{k}\}$.
 \item[(iii)]  $(x_{n_{k}})_{k}$ is skipped block w.r the branches
   $(\brn_{i})_{i\in\N}$.
 \end{enumerate}

 To prove  a)   let  $\e>0$  and   $n_{0}$ be given by
 \eqref{eq:008}. Let $I$ be an initial interval of $\N$ containing $I_{n_{0}}$ and let
 $k_{0}\in\N$ such that  $I_{k_{0}-1}\subsetneq I\subset I_{k_{0}}$. Let $n(I)\in \N$
 such that $n(I)\geq k_{0}$  and $\min F_{n}>k_{0}$ for every $n\geq n(I)$.

Let  $\bar{w}_{n}=\sum_{k\in F_{n}}a_{k}\bxnk$ be  an $\ell_{2}$-average
of $(\bxnk)_{k}$, $n\geq n(I).$
 Property (ii) yields that for every $\ell\in 
I$ and every $k_{r}\in F_{n}$ it holds that 
$P_{\brn_{\ell}}(\bar{x}_{n_{k_{r}}})=
P_{\brn_{\ell}}P_{\mc{B}_{ I_{k_{r}},>l_{k_{r}}}}(x_{n_{k_{r}}})
=P_{\brn_{\ell}}( x_{n_{k_{r}}})$.
 Using that   for every $r$,
\[I_{k_{r}}=I\cup (I_{k_{0}}\setminus I)\cup (I_{k_{r}}\setminus I_{k_{0}})
  =I\cup(I_{k_{0}}\setminus I)\cup \cup_{j=1}^{r} I_{k_{j}}\setminus I_{k_{j-1}}\] we get
\begin{align}
\bxnkr &=P_{\mc{B}_{I}}(\bxnkr)+P_{\mc{B}_{(I_{k_{0}}\setminus
         I})}(\bxnkr)+\sum_{j=1}^{r}
P_{\mc{B}_{(I_{k_{j}}\setminus I_{k_{j-1}}})}(\bxnkr
         )
\notag
  \\
&=
P_{\mc{B}_{I}}(\bxnkr)+\sum_{i\in I_{k_{0}}\setminus I}P_{\brn_{i}}(\bxnkr)+\sum_{j=1}^{r}\sum_{l\in I_{k_{j}}\setminus I_{k_{j-1}}}\sum_{i\in I_{l}\setminus I_{l-1}} P_{\brn_{i}}(\bxnkr). \label{splitxnk}
\end{align}
The above analysis yields that
\begin{align}\label{0101}
\bwn-P_{\mc{B}_{I}}\bwn&=\sum_{r=1}^{d}a_{k_r}P_{\mc{B}_{I_{k_r}}}(\bxnkr)-\sum_{r=1}^{d}a_{k_r}P_{\mc{B}_{I}}(\bxnkr)\\
&=\sum_{r=1}^{d}a_{k_{r}}\sum_{i\in I_{k_{0}}\setminus
                                                                                                                 I}P_{\brn_{i}}(\bxnkr)+\sum_{r=1}^{d}a_{k_{r}}\sum_{j=1}^{r}                                                                                                                 \sum_{l\in I_{k_{j}}\setminus I_{k_{j-1}}}\sum_{i\in I_{l}\setminus I_{l-1}}P_{\brn_{i}}(\bxnkr)
\notag \\
&=
\sum_{i\in I_{k_{0}}\setminus I}\sum_{r=1}^{d}a_{k_{r}}P_{\brn_{i}}(\bxnkr) +
            \sum_{j=1}^{d}\sum_{l\in I_{k_{j}}\setminus I_{k_{j-1}}}\sum_{i\in I_{l}\setminus
            I_{l-1}}\sum_{r=j}^{d}a_{k_{r}}P_{\brn_{i}}(\bxnkr).
\notag
\end{align}
 For $l\in I_{k_{j}}\setminus I_{{k_{j-1}}}$,  and any vector $\sum_{i\in I_{l}\setminus  I_{l-1}}\sum_{r=j}^{d}a_{k_{r}}P_{\brn_{i}}(\bxnkr)$
 from  the above property (ii) and Lemma~\ref{supperl2} we get
 \begin{equation}
   \label{eq:36}
 \norm{\sum_{i\in I_{l}\setminus  I_{l-1}}
   \sum_{r=j}^{d}a_{k_{r}}P_{\brn_{i}}(\bxnkr)}
 \leq
\left(
 \sum_{i\in I_{l}\setminus  I_{l-1}}\norm{\sum_{r=j}^{d}a_{k_{r}}P_{\brn_{i}}(\bxnkr)}^{2}
 \right)^{1/2}.
 \end{equation}
 From property (iii) and Lemma~\ref{skipped}b)
 we get
\begin{align}
  \norm{\sum_{r=j}^{d}a_{k_{r}}P_{\brn_{i}}(\bxnkr)}
  &=\left(\sum_{r=j}^{d}a^{2}_{k_{r}}\norm{P_{\brn_{i}}(\bxnkr)}^{2}
    \right)^{1/2}
    \notag
 \\
  & \leq
  (1+\e_{k_{j}})c_{i}
 \left(\sum_{r=j}^{d}a^{2}_{k_{r}}
    \right)^{1/2}\quad\text{using property (i)}
   \notag
  \\
  &\leq(1+\e_{k_{j}})c_{i}\quad\text{using  that $\sum_{r=1}^{d}a_{k_{r}}^{2}\leq 1$}.
  \label{eq:40}
\end{align}
Combining \eqref{0101},\eqref{eq:36} and \eqref{eq:40} we get,
\begin{equation}
\begin{aligned}
  \norm{\bwn-P_{\mc{B}_{I}}(\bwn)}
  &\leq
\norm{\sum_{i\in I_{k_{0}}\setminus I}\sum_{r=1}^{d}a_{k_{r}}P_{\brn_{i}}(\bxnkr)}
    +\norm{\sum_{j=1}^{d}\sum_{l\in I_{k_{j}}\setminus I_{k_{j}}}
 \sum_{i\in I_{l}\setminus I_{l-1}}\sum_{r=j}^{d}a_{k_{r}}P_{\brn_{i}}(\bxnkr)}
 \label{eq:38}
   \\
 &
\leq
\norm{\sum_{i\in I_{k_{0}}\setminus I}\sum_{r=1}^{d}a_{k_{r}}P_{\brn_{i}}(\bxnkr)}+   \sum_{j=1}^{d}\sum_{l\in I_{k_{j}}\setminus I_{k_{j}}}
   \norm{\sum_{i\in I_{l}\setminus I_{l-1}}\sum_{r=j}^{d}a_{k_{r}}P_{\brn_{i}}(\bxnkr)
   }
  \\  
&                                      \leq
(1+\e_{k_{0}})\left(\sum_{i\in I_{k_{0}}\setminus I}c_{i}^{2}\right)^{1/2}+                                      
 \sum_{j=1}^{d}\sum_{l=k_{j-1}+1}^{k_{j}}
                                    (1+\e_{k_{j}})\left(\sum_{i\in I_{l}\setminus I_{l-1}} c_{i}^{2} \right)^{1/2}
  \\
  &\leq(1+\e_{k_{0}})\sum_{l\in\N\setminus I_{n_{0}}}\left(\sum_{i\in I_{l}\setminus I_{l-1}}
    c_{i}^{2} \right)^{1/2}\\
  &\leq
    2\frac{\e}{2}=\e \quad\text{by \eqref{eq:008}}.
  \end{aligned}
  \end{equation}
This completes the proof of a).

b)  We have that $\bxnk=P_{\mc{B}_{k}}x_{n_{k}}$. As in \eqref{splitxnk}    we get  that
\[
\bxnk=\sum_{l=1}^{k}\sum_{i\in I_{l}\setminus I_{l-1}}P_{\brn_{i}}(\bxnk)\qquad (I_{0}=\emptyset),
\]
and  hence
\begin{equation*}
 \begin{aligned}
  \sum_{k=1}^{d}a_{k}\bxnk&=\sum_{k=1}^{d}a_{k}\sum_{l=1}^{k}\sum_{i\in I_{l}\setminus
    I_{l-1}}P_{\brn_{i}}(\bxnk)\\
  &=
  \sum_{i\in  I_{1}}\sum_{k=1}^{d}a_{k}P_{\brn_{i}}(\bxnk)+
\sum_{l=2}^{d}\sum_{i\in  I_{l}\setminus I_{l-1}}\sum_{k=l}^{d}a_{k}P_{\brn_{i}}(\bxnk)
.
\end{aligned}
\end{equation*}
Working as  in  \eqref{eq:38} we get that
\[
  \norm{\sum_{k=1}^{d}a_{k}\bxnk}\leq (\sum_{k=1}^{d}a_{k}^{2})^{1/2}
  (1+\e_{1})2.
\]
This yields the desired result.
\end{proof}
\begin{proposition}\label{uppers2-2}
  Let  $(x_{n})_{n}$ be a normalized block sequence in $\jtg$.
Then there exists a subsequence of  $(x_{n})_{n}$ that admits upper
$\ell_{2}$ estimate with constant  $9$.
\end{proposition}
\begin{proof}
Let us observe that in the case   there exists a subsequence
$(x_{n_{k}})_{k}$ such that $\lim\norm{\Pbrn(x_{n_{k}})}=0$ for every
branch $\brn$ it follows from Theorem~\ref{upperS2}  that  there  exists
a further subsequence that admits upper $S^{(2)}$-estimate with constant
$3$. Hence it suffices to prove the proposition
in the case  where  the sequence  admits a non-zero $\ell_{2}$  vector
$(c_{i})_{i\in \N}$.  Let  $(\brn_{i})_{i\in \N}$    be the associated branches.
Passing to a subsequence we may assume that $\sxn $ is  skipped block
with respect to the branches $(\brn_{i})_{i\in \N}$.

Let $\e>0$ and $\e_{n}\searrow 0$ such that  $\sum_{n}\e_{n}<1$. Passing to a
furhter subsequence if necessary we may assume that the conclusion of Lemmas~\ref{stabilization2} and \ref{llemma} holds i.e. there exists a strictly increasing sequence $(I_{k})_{k\in\N}$ of initial intervals of $\N$ such that setting $\bxnk=P_{\mc{B}_{k}}\bxnk$ the following holds
\begin{enumerate}
\item the sequence $(\bxnk)_{k}$  admits an upper $\ell_{2}$-estimate with constant 3,
 \item  
   setting
   $y_{n_{k}}=x_{n_{k}}-\bxnk$ it holds
   $P_{\brn}(y_{n_{k}})\to 0$ for every branch $\brn$.
\end{enumerate}
From  Theorem~\ref{upperS2} passing to a further subsequence 
we may assume that $(y_{n_{k}})_{k}$ admits upper $S^{(2)}$-estimate with constant $6$.
Since  $x_{n_{k}}=\bxnk+y_{n_{k}}$  we get that for every $b_{1},\dots, b_{m}$,
\[
  \norm{\sum_{k=1}^{m}b_{k}x_{n_{k}}}
  \leq
  \norm{\sum_{k=1}^{m}b_{k}\bar{x}_{n_{k}}}+
   \norm{\sum_{k=1}^{m}b_{k}y_{n_{k}}}\leq9\left(\sum_{k=1}^{m}b_{k}^{2}\right)^{1/2}.
 \]
 The above inequality yields that $(x_{n_{k}})_{k}$ admits upper
 $\ell_{2}$-estimate with constant $9$.
\end{proof}
\begin{proposition}\label{equivl2}
  Let $\sxn$ be a normalized block sequence  admitting non-zero $\ell_{2}$-vector.  Then, there exists a subsequence $(x_{n_{k}})_{k}$ of $\sxn$ that is equivalent to the unit vector basis of $\ell_{2}$.
\end{proposition}
\begin{proof}
   From  Proposition~\ref{uppers2-2}  we get a subsequence $(x_{n_{k}})_{k}$ of
   $\sxn$  that admits an upper $\ell_{2}$-estimate with constant $9$
   i.e.
   \begin{equation}
  \label{eq:35}
  \norm{\sum_{k}b_{k}x_{n_{k}}}\leq9\left(\sum_{k}b_{k}^{2}\right)^{1/2} \qquad\forall (b_{k})_{k}.
\end{equation}
Since  $\sxn$ admits a non-zero $\ell_{2}$-vector there exist a branch
$\brn$  and  $c>0$ such that  $\lim_{k}\norm{P_{\brn}(x_{n_{k}})}=c$. Passing to a further subsequence  we may assume that
$\norm{P_{\brn}(x_{n_{k}})}>\sfrac{c}{2}$ for every $k$ and that 
   $(x_{n_{k}})_{k}$ is skipped w.r. the branch  $\brn$. 
This yields that  $(P_{\brn}(x_{n_{k}}))_{k}$ is  also skipped block sequence of $(x_{n}^{\brn})_{n}$.  Since $(x_{n}^{\brn})_{n}$ is isometric to the unit vector basis of $\ell_{2}$ we get,
\begin{equation}
 \begin{aligned}
  \label{eq:34}
    \norm{\sum_{k}b_{k}x_{n_{k}}}&\geq\norm{P_{\brn}(\sum_{k}b_{k}x_{n_{k}})}
     =\norm{\sum_{k}b_{k}P_{\brn}(x_{n_{k}})}\\
&=\left(\sum_{k}b_{k}^{2}\norm{\Pbrn(x_{n_{k}})}^{2}\right)^{1/2}
\geq \frac{c}{2}\left(\sum_{k}b_{k}^{2}\right)^{1/2}.
\end{aligned}
\end{equation}
From \eqref{eq:34} and \eqref{eq:35} we conclude that $(x_{n_{k}})_{k\in\N}$ is equivalent to the unit vector basis of  $\ell_{2}$.
\end{proof}
\begin{corollary}\label{l2orc0}
  Let  $\sxn$ be a normalized block sequence in $\jtg$. Then there exists a subsequence $(x_{n_{k}})_{k}$ of $\sxn$  which is  either   equivalent to the unit vector basis of $\ell_{2}$
  or  $c_{0}$  embeds in the subspace generated by that subsequence.
\end{corollary}

Indeed if $\sxn$ has a subsequence admitting a non-zero $\ell_{2}$-vector then   
Proposition~\ref{equivl2} yields  a further subsequence equivalent to the unit 
vector basis of $\ell_{2}$.
On the opposite, Proposition~\ref{c0}  yields that  $c_{0}$ embeds in the  subspace
generated by every subsequence of $\sxn$.
\section{Hilbertian subspaces of $\jtg$.}
In this section we provide a characterization of the subspaces  of
$\jtg$ which are  isomorphic to $\ell_{2}$ and we show that  each such
subspace is complemented in $\jtg$.
We start with some  preparatory lemmas.

\begin{lemma} \label{averages1a}
  Let  $(x_{n})_{n}$ be a normalized block sequence
having a non-zero   $\ell_{2}$-vector    $(c_{i})_{i\in\N}$ and
$(\brn_{i})_{i\in\N}$ be the associated set of branches. 
Then, there exists a subsequence $(x_n)_{n\in N}$ of $\sxn$ such that
if   $(w_{n})_{n\in\N}$ is  a  block sequence of $(x_n)_{n\in N}$ consisting of $(2,n)$  averages
the following holds:
\begin{enumerate}
\item[a)] for every $\e>0$ there exists an initial interval $I(\e)$ of $\N$
such  that  for every initial segment  $I$ of $\N$ containing $I(\e)$
there exists $n(I)\in\N$ such that
  $\norm{w_{n}-P_{\mc{B}_{I}}w_{n}}<\e$ for every $n\geq n(I)$,
\item[b)] the sequence $(w_{n})_{n}$
has the same 
$\ell_{2}$-vector  and  associated branches as $\sxn$,
\item[c)] $\norm{w_{n}}\to \norm{(c_{i})_{i\in\N}}_{\ell_{2}}$.
 \end{enumerate}
\end{lemma}
\begin{proof}
  Passing  if necessary to a subsequence  we may assume that $\sxn$ is skipped with respect to
 the branches $(\brn_{i})_{i\in\N}$.   From Lemmas~\ref{stabilization2}
 and \ref{llemma} we get  sequences  $n_{k}, l_{k}\nearrow+\infty$, a  strictly
 increasing sequence $(I_{k})_{k}$ of initial intervals  of $\N$  and
$\e_{k}\searrow 0$ such that
 the following holds:
 
\begin{enumerate}
 \item[(i)]  $    (1-\e_{k})c_{i}\leq  \norm{\Pbrn[\brn_{i}]x_{n}}\leq (1+\e_{k})c_{i}$ for every
   $i\in I_{k}$ and $n\geq n_{k}$,
 \item[(ii)]  the final segments $\brn_{i,>l_{k}}$, $i\in I_{k}$,  are
   incomparable, $\# I_{k}<l_{k}$  and\\
   $\minsupp(x_{n_{k}})>\max\{\min \brn_{i,>l_{k}}:i\in
   I_{k}\}$,
   \item[(iii)]
   Setting $\bxnk=P_{\mc{B}_{I_{k}}}(x_{n_{k}})$, $k\in\N$, it holds that 
   for every $\e>0$ and 
for every block sequence $(\bar{w}_{n})_{n\in\N}$ of
$(\bar{x}_{n_{k}})_{k}$ such that $\bar{w}_{n}=\sum_{k\in
  F_{n}}a_{k}\bar{x}_{n_{k}}$ is an $\ell_{2}$-convex combination  there exists $n_{0}\in\N$  such that for every initial segment  $I$ of $\N$ containing $I_{n_{0}}$  there exists $n(I)\in\N$
  \begin{equation*}
 \norm{\bar{w}_{n}-P_{\mc{B}_{I}}(\bar{w}_{n})}<\e\,\,\,\forall n>n(I).
\end{equation*}
\item[(iv)] Setting $y_{n_{k}}=x_{n_{k}}-\bxnk$  it holds that the sequence
  $(y_{n_{k}})_{k}$ admits  upper $S^{(2)}$-estimation with constant $6$.
 \end{enumerate}
We choose a block sequence $(w_{n})_{n}$ of $(x_{n_{k}})_{k}$  consisting of $(2,n)$-averages, $w_{n}=\sum_{k\in F_{n}}a_{k}x_{n_{k}}$.
We set $\bwn=\sum_{k\in F_{n}}a_{k}\bxnk$ and $v_{n}=\sum_{k\in
  F_{n}}a_{k}y_{n_{k}}$. Then  $w_{n}=\bwn+v_{n}$. Using 
property (iv) of the sequence $(y_{n_{k}})_{k}$ we get   as in
Proposition~\ref{s2averages} that $\norm{v_{n}}\to 0$.

a)    Let $\e>0$ and  $I_{n_{0}}$ be the  initial segment   we get from
property (iii).  Set $I(\e)=I_{n_{0}}$.
We have that for every initial segment  $I$ of $\N$ containing $I_{n_{0}}$  there exists $n(I)\in\N$ such that
$\norm{\bwn-P_{\mc{B}_{I}}\bwn}<\sfrac{\e}{2}$    for every $n\geq n(I)$. We may also assume that
$\norm{v_{n}}<\sfrac{\e}{4}$ for every $n\geq  n(I)$. Then
\[
  \norm{w_{n}-P_{\mc{B}_{I}}(w_{n})}\leq\norm{\bwn-P_{\mc{B}_{I}}\bwn}+
  \norm{v_{n}-P_{\mc{B}_{I}}v_{n}}<\frac{\e}{2}+2\frac{\e}{4}=\e.
\]

b)   Let $\brn\ne\brn_{i}$ for every $i\in\N$.  We show that
$\lim_{n}\norm{P_{\brn}(w_{n})}\to 0$. On the contrary assume that for
some $\e>0 $  and $M\in[\N]$ it holds $\norm{P_{\brn}(w_{n})}>\e$ for
every  $n\in M$.  Since $\norm{v_{n}}\to 0$ we may assume that
$\norm{P_{\brn}(\bwn)}>\e$ for every $n\in M$.  Using (iii)   we  choose  an    interval $I_{n_{0}}$
such that  $\norm{\bwn-P_{\mc{B}_{I_{n_{0}}}}(\bwn)}<\sfrac{\e}{2}$.
Choose  $l\in\N$ such that  the final segments $\brn_{,>l},\brn_{i,>l}$, $i\in I_{n_0}$
are  incomparable. Choose  also $n(\e)$ such that
$\minsupp (x_{n_{k}})> l$ for every $k\in F_{n}$ and $n\geq n(\e)$.
This yields that $P_{\brn}(P_{\mc{B}_{I_{n_{0}}}}(\bwn))=0$ and hence
\[
  \e<\norm{P_{\brn}(\bwn)}=
  \norm{P_{\brn}(\bwn-P_{\mc{B}_{I_{n_{0}}}}(\bwn))}
\leq\norm{\bwn-P_{\mc{B}_{I_{n_{0}} }}(\bwn)}<\frac{\e}{2}\,\,\,\text{ a contradiction}.
\]
Therefore  $\lim_{n}\norm{P_{\brn}(w_{n})}\to 0$.

We show now that    $\lim_{n}\norm{P_{\brn_{i}}(w_{n})}=c_{i}$ for
every branch $\brn_{i}$, $i\in\N$.
Choose  a branch $\brn_{i}$. Since  $\sxn$ is skipped block with respect to the branches
$(\brn_{i})_{i\in\N}$,  from Lemma~\ref{skipped} we get
that    for  every $n>i$,
\begin{align}
  \norm{P_{\brn_{i}}(w_{n})}=
                              \norm{\sum_{k\in F_{n}}a_{k}P_{\brn_{i}}(x_{n_{k}})}
 &=\left(\sum_{k\in F_{n}}a^{2}_{k}\norm{P_{\brn_{i}}(x_{n_{k}})}^{2}\right)^{1/2}\xrightarrow[n]{} c_{i}
  \notag
\end{align}
since $\norm{P_{\brn_{i}}(x_{n_{k}})}\to c_{i}$.

From  the above we get  that  $(w_{n})_{n}$ has  $\ell_{2}$-vector and associated
set of brances the same with $\sxn$.

c) We show that  $\lim_{n}\norm{w_{n}}=\norm{(c_{i})_{i\in\N}}_{\ell_{2}}$.

From b) and  Lemma~\ref{lowerl2}  we get that  $\norm{(c_{i})_{i\in\N}}_{\ell_{2}}\leq\liminf\norm{w_{n}}$.
Assume on the contrary that
$\limsup\norm{w_{n}}>\norm{(c_{i})_{i\in\N}}_{\ell_{2}}$.
Hence passing to a subsequence  we may assume that
for some   $\e>0$ it holds
$\norm{w_{n}}>\norm{(c_{i})_{i\in\N}}_{\ell_{2}}+\e$ for every $n$.

From the proof  of a) we have  an   $n_{0}\in\N $ such that
for every initial interval of $\N$ containing $I_{n_{0}}$ there exists
$n(i)$ such that
$\norm{w_{n}-P_{\mc{B}_{I}}(w_{n})}<\e/4$  for every $n\geq n(I)$. 
It follows that
\begin{equation}
  \label{eq:39}
\e+ \norm{(c_{i})_{i\in\N}}_{\ell_{2}}<
\norm{w_{n}}\leq\norm{w_{n}-P_{\mc{B}_{I_{n_{0}}}}(w_{n})}+
\norm{P_{\mc{B}_{I_{n_{0}}}}(w_{n})}<\frac{\e}{4}+\norm{P_{
    \mc{B}_{I_{n_{0}}}}(w_{n})}.
\end{equation}
By (ii) we have that the final segments
$\brn_{i,>l_{n_{0}}}, i\in I_{n_{0}}$, are incomparable and $\#I_{n_{0}}<l_{n_{0}}$.
Choose  $k_{0}>n_{0}$ such that
\begin{center}
  $\e_{k_{0}}\norm{(c_{i})_{i\in\N}}_{\ell_{2}}<\frac{\e}{4}$\,\,\,  and \,\,\,
$\minsupp x_{n_{k}}>\max\{\min\brn_{i,>l_{n_{0}}}:i\in I_{n_{0}}\}\,
\forall k\geq
k_{0}$.  
\end{center}
Then  for every  $n$ such that $\min F>k_{0}$ 
we get,
\begin{align*}
 \norm{P_{\mc{B}_{I_{n_{0}}}}(w_{n})}&=\norm{\sum_{i\in I_{n_{0}}}\sum_{k\in
  F_{n}}a_{k}P_{\brn_{i}}(x_{n_{k}})}
  \\
  &=(\sum_{i\in I_{n_{0}}}\norm{\sum_{k\in
  F_{n}}a_{k}P_{\brn_{i}}(x_{n_{k}})}^{2})^{1/2}\quad \text{by
  Lemma~\ref{supperl2}}
  \\
  &=(\sum_{i\in I_{n_{0}}}
    \sum_{k\in
  F_{n}}a_{k}^{2}\norm{P_{\brn_{i}}(x_{n_{k}})}^{2})^{1/2}\quad \text{by
    Lemma~\ref{skipped}}
  \\
                    &\leq
                      (\sum_{i\in I_{n_{0}}}(1+\e_{k_{0}})^{2}c_{i}^{2})^{1/2}
\leq (1+\e_{k_{0}})\norm{(c_{i})_{i\in\N}}_{\ell_{2}} \quad  \text{by (i)}.                     
\end{align*}
Combining the above inequality with \eqref{eq:39} we get
\[\e+\norm{(c_{i})_{i\in\N}}_{\ell_{2}}
<\frac{\e}{4}+\norm{P_{
    \mc{B}_{I_{n_{0}}}}(w_{n})}\leq
  \frac{\e}{4}+(1+\e_{k_{0}})
  \norm{(c_{i})_{i\in\N}}_{\ell_{2}}<\frac{\e}{2}+ \norm{(c_{i})_{i\in\N}}_{\ell_{2}}
\]
a contradiction. Therefore
\[
 \norm{(c_{i})_{i\in\N}}_{\ell_{2}}\leq\liminf\norm{w_{n}}\leq\limsup\norm{w_{n}}\leq \norm{(c_{i})_{i\in\N}}_{\ell_{2}}
\]
and this completes the proof.
\end{proof}
\begin{lemma}\label{block2}
Let $(x_{n})_{n}$ be a normalized weakly  null sequence with
a non-zero $\ell_{2}$-vector  $(c_{i})_{i\in \N}$ and let $\mc{B}$  be the
associated set of branches. Then, if $(y_{n})_{n}$
is a  block sequence  such that
$\lim_{n}\norm{x_{n}-y_{n}}=0$,  the  sequence $(y_{n})_{n}$ also has
 $(c_{i})_{i\in\N}$ as an $\ell_{2}$-vector and the same set of
associated branches as $(x_{n})_{n}$.
  \end{lemma}
  \begin{proof}
  The proof is immediate consequence    of the triangle inequality
  and the fact that $\norm{\Pbrn(x)}\leq\norm{x}$  for every  branch $\brn$.
\end{proof}
\begin{notation}
  Let $\e_{n}\searrow 0$ and   $Y$  be a subspace of $\jtg$. We shall say
  that a seminormalized block sequence $\sxn$ is equivalent to a
  sequence  in  $Y$ if there exists a normalized weakly null sequence $\syn$ in $Y$ such
  that  $\norm{y_{n}-x_{n}}<\e_{n}$ for every $n\in\N$.
\end{notation}

\begin{proposition}\label{lastprop}
Let $Y$ be a subspace of $\jtg$ with the property that for every finite
  set of branches $T=\{\tau_{1},\dots,\tau_{d}\}$  there exists a seminormalized
  block   sequence $(x_{n})_{n}$ equivalent to a sequence in $Y$,
  admitting  an  $\ell_{2}$-vector $(c_{i})_{i\in\N}\in S_{\ell_{2}}$ with associated branches $(\brn_{i})_{i\in\N}$,   satisfying the following
\begin{enumerate}
  \item $\lim_{n}\norm{P_{T}x_{n}}=0$,
  \item for every $\e>0$  there exists an initial interval  $I(\e)$ of $\N$
  such that for every initial segment $I$ of  $\N$ containing $I(\e)$ there exists  $n(I)\in\N$
such that  $\norm{x_{n}-P_{\mc{B}_{I}}x_{n}}<\e$ for every
  $n\geq n(I)$.
\end{enumerate}
  Then, there exists a normalized sequence $(y_{n})_{n}\in Y$ such that $\norm{P_{\brn}(y_{n})}\to 0$ for every branch $\brn$.
\end{proposition}
\begin{proof}
  Let $\e_{k}\searrow 0$ and $\sum_{k}\e_{k}<1$.
  From the hypothesis we get  a seminormalized block
  sequence $(x_{n}^{1})_{n\in\N}$   equivalent to a sequence  in $Y$,  admitting an
  $\ell_{2}$-vector $(c_{i}^{1})_{i\in\N}\in S_{\ell_{2}}$. Let $(\brn_{i}^{1})_{i\in\N}$ be the associated set of branches.  Choose an initial interval $I_{1}$ of $\N$  such that
  $\norm{x_{n}^{1}-P_{\mc{B}_{I_{1}}}x_{n}^{1}}<\e_{1}$ for every  $n\geq  n(1)$.
 Using the definition of the $\ell_{2}$-vector we may also assume that  
 \begin{center}
  $(1-\e_{1})c_{i}^{1}\leq\norm{P_{\brn_{i}^{1}}x_{n}^{1}}\leq (1+\e_{1})c_{i}^{1}$
for every $i\in I_{1}$ and $n\geq  n(1)$.
\end{center}
Set $T_{1}=\mc{B}_{I_{1}}$. From the hypothesis we get that there
exists a seminormalized block sequence $(x_{n}^{2})_{n\in\N}$
equivalent to a sequence  in $Y$
satisfying $\lim_{n}\norm{P_{T_{1}}(x_{n}^{2})}=0$  and
admitting  an $\ell_{2}$-vector $(c_{i}^{2})_{i\in\N}\in
S_{\ell_{2}}$. Let $(\brn_{i}^{2})_{i\in\N}$ be the associated set of branches.
Choose an initial interval $I_{2}$ of $\N$, $I_{1}\subsetneq I_{2}$,  such that
  $\norm{x_{n}^{2}-P_{\mc{B}_{I_{2}}}x_{n}^{2}}<\e_{2}$ for every  $n\geq  n(2)$.
  Using the definition of the $\ell_{2}$-vector we may also assume that
\begin{center}  
$(1-\e_{2})c_{i}^{2}\leq\norm{P_{\brn_{i}^{2}}x_{n}^{2}}\leq (1+\e_{2})c_{i}^{2}$
for every $i\in I_{2}$ and $n\geq  n(2)$.
\end{center}
Set  $T_{2}=\cup_{l=1}^{2}\mc{B}_{I_{l}}$.
Continuing we   choose inductively block sequences
$(x_{n}^{k})_{n\in\N}$, $k\in\N$, a strictly increasing sequence
$(I_{k})_{k\in\N}$ of initial intervals of $\N$  and $n(k)\in\N$, such that for every $k$,  $(x_{n}^{k})_{n}$ is
 equivalent to a sequence  in $Y$,  admits an $\ell_{2}$-vector $(c_{i}^{k})_{i\in\N}\in
S_{\ell_{2}}$ such that if $(\brn_{i}^{k})_{i\in\N}$ is the associated set
of branches then
\begin{enumerate}
\item[(i)] $\lim_{n}\norm{P_{T_{k-1}}x_{n}^{k}}=0$, where
  $T_{k-1}=\cup_{l\leq k-1}\mc{B}_{I_{l}}$,
\item[(ii)] $\norm{x_{n}^{k}-P_{\mc{B}_{I_{k}}}x_{n}^{k}}<\e_{k}$ for every $n\geq  n(k)$,
\item[(iii)]  $(1-\e_{k})c_{i}^{k}\leq\norm{P_{\brn_{i}^{k}}x_{n}^{k}}\leq
  (1+\e_{k})c_{i}^{k}$
    for every  $i\in I_{k}$ and $n\geq n(k)$.
    \end{enumerate}
Properties  (i) and (iii)  implies that  the sets 
$\mc{B}_{I_{k}}$ are pairwise disjoint.
 Applying Lemma~\ref{incomparable2} we get $K\in[\N]$  and
$(l_{k})_{k\in K}$ such that
\begin{equation}
\begin{aligned}
 \forall   k<k_{1}<k_{2}\in K \,\,\text{and every}\,\, i\leq\max I_{k}
\\
 \text{the final segments}\,\,\ \brn_{i,>l_{k_{1}}}^{k_{1}},
 \brn_{i,>l_{k_{2}}}^{k_{2}} \text{are incomparable}.
\label{eq:0042}
\end{aligned}
\end{equation}
For every $k_{j}\in K$ choose   $x_{n_{k_{j}}}^{k_{j}}\in
(x_{n}^{k_{j}})_{n\in \N}$, $n_{k_{j}}>\max\{n_{k_{j-1}}, n(k_{j})\}$, such
 that
\begin{enumerate}
\item[(iv)]  $p_{j}=\minsupp x_{n_{k_{j}}}^{k_{j}}>l_{k_{j}}$,
\item[(v)]  the final segments $\brn_{i,\geq p_{j}}^{k_j}$,$i\in I_{k_j}$ are incomparable.
\end{enumerate}
Let us observe, using that the  $\ell_{p}$-vectors are decreasing sequences,
that
\begin{equation}
  \label{eq:010}
\forall \e>0\,\text{
 there exists  $i(\e)$ such that  $c^{k}_{i}<\e$ for every $i\geq
i(\e)$ and every $k$}.  
\end{equation}

\noindent \textit{Claim:}
For every branch $\brn$ it holds that $\lim_{j}\norm{\Pbrn(x_{n_{k_{j}}}^{k_{j}})}=0$.

\noindent \textit{Proof of Claim}.
On the contrary assume that there exist $\e>0$ and $J\in [\N]$ 
such that   $\norm{\Pbrn(x_{n_{k_{j}}}^{k_{j}})}>\e$  for  every $j\in
J$.
Set   $z_{j}=P_{\mc{B}_{I_{k_{j}}}}(x^{k_{j}}_{n_{k_{j}}})$, $j\in\N$. Using
the above property (ii) we may assume that $\norm{P_{\brn}z_{j}}>\sfrac{\e}{2}$
for every $j\in J$.

Let  $j_{0}\in J$  be such that   $[1, i(\e/4)]\subset I_{k_{j_{0}}}$.
Then if  $j>j_{0}$
from (iv) and  (v) we get that there exists  a unique $i(k_{j})\in
I_{k_{j}}$  such that
$P_{\brn}(z_{j})=P_{\brn}(P_{\brn^{k_{j}}_{i(k_{j})}}x_{n_{k_{j}}}^{k_{j}})$. It
follows
\[
\frac{\e}{2}\leq\norm{P_{\brn}z_{j}}=\norm{P_{\brn}(P_{\brn^{k_{j}}_{i(k_{j})}}x_{n_{k_{j}}}^{k_{j}})}\leq (1+\e_{k_{j}})c_{i(k_{j})}^{k_{j}}<2c_{i(k_{j})}^{k_{j}},
\]
using  the above property (iii) and that $\e_{k}<1$ for every $k$.
Since $c_{i}^{k_{j}}<\sfrac{\e}{4}$ for every $i\geq i(\e/4)$,
from \eqref{eq:010}
we get  that  $i(k_{j})< i(\e/4)$. Passing to a subsequence
we may assume that  $i(k_{j})=i< i(\sfrac{\e}{4})$ for every $j$.
The above property implies  that  for $j_{0}<j_{1}<j_{2}$  it holds
that $\brn\cap\brn_{i,>l_{k_{j_{i}}}}^{k_{j_{i}}}\ne\emptyset $ for $i=1,2$,
a contradiction since by \eqref{eq:0042} these final  segments are incomparable.
$\square$.

From the choice of the sequences $(x_{n}^{k})_{n\in\N}$, $k\in\N$, we have
that for every $j$ there exists  $y_{j}\in S_{Y}$  such that
$\norm{x_{n_{k_{j}}}^{k_{j}}-y_{j}}<\e_{k_{j}}$.  Since $\e_{k_{j}}\searrow
0$ it follows that  $(y_{j})_{j\in\N}$   is a normalized sequence in
$Y$ such that  $\lim_{j}\norm{P_{\brn}(y_{j})}=0$ for every branch $\brn$.
\end{proof}
A more careful inspection  of the previous proof shows that  $\ell_{2}$
embeds  in $Y$ once we have  an array of normalized block sequences
with certain properties whose the $\ell_{2}$-vector is not necessarily
normalized.
\begin{proposition}
  Let $(x_{n}^{k})_{n\in\N}, k\in\N$ be an array of normalized block
  sequences having   non-zero $\ell_{2}$-vectors $(c_{i}^{k})_{i\in \N}$
with  associated set of branches $(\brn_{i}^{k})_{i\in\N}$,
an $\e_{n}\searrow 0$  and  a strictly increasing sequence   $(I_{k})_{k\in\N}$ of
initial intervals of $\N$ such that 
  \begin{enumerate}
  \item The sets $\mc{B}_{I_{k}}=\{\brn_{i}^{k}:i\in I_{k}\}$, $k\in\N$, 
 are  disjoint.
    \item For every $k,n$,  it holds that
      \[
        \norm{P_{B_{I_{k}}}(x_{n}^{k})}>\frac{1}{2}\,\,\text{and}\,\, \norm{x_{n}^{k}-P_{B_{I_{k}}}(x_{n}^{k})}<\e_{k}.\]
      \end{enumerate}
  Then, there exists a sequence
  $(x_{n_{j}}^{k_{j}})_{j}\subset\{x_{n}^{k}:n,k\in\N\}$    satisfying
$\lim_{j}\norm{P_{\brn}(x_{n_{j}}^{k_{j}})}=0$ for every branch  $\brn$.
\end{proposition}
\begin{proof}[Sketch of the proof] The proof follows the arguments of
  the proof of Proposition~\ref{lastprop}.
   By the
  assumptions  we have  that  (ii) of the proof of Proposition~\ref{lastprop} holds.
  For every $k$ passing to a subsequence we may assume that
  (iii)  also holds.  By  (1) the sets
   $\mc{B}_{I_{k}}, k\in\N$ are disjoint and thus   we may apply
   Lemma~\ref{incomparable2} to get  $K\in[\N]$ and $(l_{k})_{k\in K}$
   such that \eqref{eq:0042} holds.  Continuing as in the proof  of
   Proposition~\ref{lastprop}
   we get the subsequence $(x^{k_{j}}_{n_{j}})_{j}$
satisfying   $\lim_{j}\norm{P_{\brn}(x_{n_{j}}^{k_{j}})}=0$ for every branch  $\brn$.
\end{proof}

Next, we prove the following proposition which is analogous to
Proposition~\ref{lpvectornorm1}. The proof we  give  is almost
identical to the proof  of Proposition~\ref{lpvectornorm1} and we
present it for sake of completeness.
\begin{proposition}\label{l2vectornorm1}
  Let $T=\{\brnt_{1},\dots\brnt_{d}\}$ be a finite set of branches
and   $Y$ be a subspace of $\jtg$ which contains a normalized weakly null sequence $(y_{n})_{n}$
such that $\lim_{n} \norm{P_{T}y_{n}}=0$,  
and  has  an $\ell_{2}$-vector $(c_{i})_{i\in\N}$ with  norm
$\norm{(c_{i})_{i}}_{\ell_{2}}=a\in (0,1)$.  Let
$(\brn_{i})_{i\in\N}$ be the associated set of branches.
Then there exists a normalized weakly null sequence $(w_{n})_{n}$   in $Y$ such
that
\begin{enumerate}
\item[a)] the sequence  $(w_{n})_{n\in\N}$ 
has  an $\ell_{2}$-vector $(d_{i})_{i\in\N}\in S_{\ell_{2}}$ with the same set of
associated branches as $(y_{n})_{n}$.
\item[b)]
  $\lim_{n} \norm{P_{T}w_{n}}=0$.
\item[c)] For every $\e>0$ there exists  an  initial interval $I(\e)$
  of $\N$
  such that for  every initial segment  $I$ of $\N$
  containing  $I(\e)$ there exists  $n(I)\in\N$
  such that
  $\norm{w_{n}-P_{\mc{B}_{I}}w_{n}}<\e$ for every $n\geq  n(I)$.
  \end{enumerate}
  \end{proposition}
\begin{proof}
From  Lemma~\ref{block2}  we may assume that $(y_{n})_{n}$ is
  equivalent to a block sequence $(x_{n})_{n\in\N}$.
Choose  $l_{0}\in\N$ such that the final segments $\brnt_{i,>l_{0}}$, $i\leq
d$, are incomparable and $d<l_{0}$.
Passing to a subsequence  of  $(x_{n})_{n\in\N}$ if necessary we may
assume
that is skipped block w.r the branches $\tau_{i}, i\leq d$ and
$\minsupp(x_{n})>l_{0}$
for every $n$.
From Lemma~\ref{averages1a} passing to a further subsequence we get
that for every  block
sequence $(z_{n})_{n}$ of $(x_{n})_{n\in\N}$  where 
each $z_{n}$ is  a $(2,n)$-average, $\norm{z_{n}}\to a$ and moreover
 $(z_{n})_{n}$
 has the same $\ell_{2}$-vector and set of branches with $(y_{n})_{n}$.
 Lemma~\ref{averages1a}  also yields that $(z_{n})_{n}$ satisfies
 property c).

We show that $\lim_{n}\norm{P_{T}(z_{n})}=0$.
Let $z_{n}=\sum_{k\in F_{n}}a_{k}x_{k}$, $n\in\N$.
From  Lemma~\ref{supperl2} we get 
 \[
\norm{P_{T}(x_{n})}=\left(\sum_{i=1}^{d}\norm{P_{\brnt_{i}}(x_{n})}^{2}\right)^{1/2}.
\]
Let $\e>0$. Using the above equality and that $P_{T}(x_{n})\to 0$,
choose $n_{0}\in\N$ such that
$\norm{P_{\brnt_{i}}(x_{n})}<\frac{\e}{d}$ for every $i\leq d$ and every
$n\geq n_{0}$. Using that $(x_{n})_{n\in\N}$ is skipped block w.r the
branches  $\tau_{i}, i\leq d$, it follows
\[
  \norm{P_{\brnt_{i}}(\sum_{k\in F_{n}}a_{k}x_{k})}=
  \norm{\sum_{k\in F_{n}}a_{k}P_{\brnt_{i}}(x_{n})}=\left(\sum_{k\in F_{n}}a_{k}^{2}\norm{P_{\brnt_{i}}(x_{n})}^{2}\right)^{1/2}\leq\frac{\e}{d}.
\]
From the triangle inequality we get
\[
  \norm{P_{T}(z_{n})}\leq\sum_{i=1}^{d}
 \norm{P_{\brnt_{i}}(\sum_{k\in F_{n}}a_{k}x_{k})}\leq\e.
\]
Hence we have shown that
$\lim_{n}\norm{P_{T}(z_{n})}=0$.  Set $v_{n}=a^{-1}z_{n}$, $n\in\N$.
Using that  $(x_{n})_{n}$ is equivalent to a sequence of  $Y$ by
standard  perturbations   arguments, passing if
necessary to a subsequence of $(v_{n})_{n}$,  we get a  sequence
$(w_{n})_{n}$  in $Y$  equivalent  to $(v_{n})_{n}$.
From  the above we get that the $(w_{n})_{n}$ satisfies the conclusion.
\end{proof}
We proceed now to the main result of this section.
\begin{theorem}\label{complemented3}
  Let $Y$ be a  subspace of $\jtg$ such
every normalized weakly null sequence has a subsequence equivalent to the unit vector basis of $\ell_{2}$.  Then $Y$ is   a  complemented subspace of $\jtg$ isomorphic to $\ell_{2}$.
\end{theorem}
To prove the theorem  we shall need the following proposition.
\begin{proposition}
  Let $Y$ be a  subspace of $\jtg$ such
every normalized weakly null sequence has a subsequence equivalent to the unit vector basis of $\ell_{2}$.
Then there exist a finite set  of 
  branches $T=\{\brnt_{1},\dots,\brnt_{d}\}$
  and a finite codimensional subspace $Y_{0}$ of $Y$
 such that  the bounded  projection  $P_{T}$,
restricted on $Y_{0}$  is an isomorphism.
\end{proposition}
\begin{proof} Assume that the conclusion does not hold.
Then
for every  finite set of branches $T=\{\brnt_{1},\dots,\brnt_{d}\}$ and
every  finite codimensional subspace $Y_{0}$ of $Y$
the restriction of $P_{T}$ to $Y_{0}$   is not an isomorphism.
For every $k\in\N$ set $\displaystyle H_{k}=\cap_{i=1}^{k}\ker (e_{i}^{*})$
and $Y_{k}=Y\cap H_{k}$. Then inductively   we choose  normalized vectors
$x_{k}\in Y_{k}$  such that $\norm{P_{T}x_{k}}<k^{-1}$.  It follows that
the sequence $(x_{k})_{k}$ is a normalized weakly null sequence  and
$\lim_{k}\norm{P_{T}x_{k}}=0$. Moreover our assumption yield
that there exists a subsequence  of $\sxn$ which admits 
an  $\ell_{2}$ vector $(c_{i})_{i\in\N}$
with  norm
  $  \norm{(c_{i})_{i\in\N}}_{\ell_{2}}=a>0$. 
  From   Proposition~\ref{l2vectornorm1}  we conclude that
  the assumptions
of Proposition~\ref{lastprop} are fulfilled.  Hence  the subspace  $Y$
contains a normalized weakly null sequence that admits upper $S^{(2)}$-estimate yielding  a contradiction.
\end{proof}

  \begin{proof}[Proof of Theorem~\ref{complemented3}]

From the above proposition  we  get  a  finite codimensional subspace  $Y_{0}$ of $Y$
and  a finite set of branches $T=\{\brnt_{1},\dots,\brnt_{d}\}$ such that  the bounded  projection  $P_{T|Y_{0}}$  is an isomorphism.
The rest of the proof follows the same steps as the proof of
Theorem~\ref{complemented2}.
\end{proof}

Combining the above we get   the following.
\begin{corollary}
Let   $Y$ be a  subspace of $\jtg$. Then either $Y$ is Hilbertian or  $c_{0}$ is isomorphic to a subspace of 
$Y$. In particular every reflexive  subspace of $\jtg$ is Hilbertian and  complemented.
\end{corollary}
\begin{proof}
Let $Y$ be a subspace of $\jtg$ such that $c_{0}$ does not embeds  in $Y$
and   $(y_{n})_{n}$ be a normalized  weakly null sequence in $Y$. Passing to a subsequence
we may assume  that  is equivalent to a block basis $(x_{n})_{n}$ .
By Corollary~\ref{l2orc0} $(x_{n})_{n\in\N}$ has a subsequence which is
either equivalent to the unit vector  of $\ell_{2}$ or  $c_{0}$ embeds in the
subspace generated by  the subsequence.
The assumption excludes the  later case. It follows that 
every normalized weakly null sequence in $Y$  has a subsequence equivalent to the unit vector basis of $\ell_{2}$  and hence Theorem~\ref{complemented3}
yields that $Y$ is  Hilbertian and complemented in $\jtg$.
\end{proof}
 \begin{remark}
All the results proved for the space $\jtg$ remains valid if we consider a complete subtree of $\phi_{n}^{\brn}$, $n\in\N$, $\brn\in\brT$ and the corresponding norming set $G^{\prime}$.  In particular if a complete subtree  is well founded then  for  every normalized weakly null sequence  $\sxn$ in $JT_{G^{\prime}}$ and every branch  $\brn$ of the subtree  it holds  $\lim_{n}\norm{P_{\brn}(x_{n})}=0$, hence $\sxn$  has a subsequence that admits upper  $S^{(2)}$-estimate. This yields that the space 
$JT_{G^{\prime}}$   is $c_{0}$ saturated.
 \end{remark}

\end{document}